\documentclass[11pt]{amsart}
\usepackage {amsmath, amssymb, amscd, mathrsfs, enumerate, url, graphicx, color}
\usepackage[all, cmtip]{xy}
\usepackage[text={6.5in,9in},centering,letterpaper,dvips]{geometry}

\usepackage{amsrefs}
\usepackage[colorlinks=true]{hyperref}
\usepackage{cleveref}
\usepackage{amscd}
\usepackage{amsthm}
\usepackage{mathtools}
\usepackage{tikz-cd}
\usepackage{bm}
\allowdisplaybreaks

\newtheorem{theorem}{Theorem}[section]
\newtheorem{lemma}[theorem]{Lemma}
\newtheorem{proposition}[theorem]{Proposition}
\newtheorem{corollary}[theorem]{Corollary}

\newtheorem{conjecture}[theorem]{Conjecture}
\newtheorem{fact}[theorem]{Fact}

\theoremstyle{definition}
\newtheorem{definition}[theorem]{Definition}
\newtheorem{example}[theorem]{Example}

\newtheorem{innercustomassumption}{Assumption}
\newenvironment{customassumption}[1]
  {\renewcommand\theinnercustomassumption{#1}\innercustomassumption}
  {\endinnercustomassumption}

\theoremstyle{remark}
\newtheorem{remark}[theorem]{Remark}

\theoremstyle{question}
\newtheorem{question}[theorem]{Question}

\numberwithin{equation}{section}

\newcommand{\R}{\mathbb{R}}

\newcommand{\Z}{\mathbb{Z}}

\newcommand{\lt}{\left}
\newcommand{\rt}{\right}
\newcommand{\tms}{\times}

\newcommand{\rmk}{\begin{remark}}
\newcommand{\ermk}{\end{remark}}
\newcommand{\cor}{\begin{corollary}}
\newcommand{\ecor}{\end{corollary}}
\newcommand{\eq}{\begin{equation}}
\newcommand{\eeq}{\end{equation}}
\newcommand{\eqs}{\begin{equation*}}
\newcommand{\eeqs}{\end{equation*}}
\newcommand{\prop}{\begin{proposition}}
\newcommand{\eprop}{\end{proposition}}
\newcommand{\thm}{\begin{theorem}}
\newcommand{\ethm}{\end{theorem}}
\newcommand{\conj}{\begin{conjecture}}
\newcommand{\econj}{\end{conjecture}}
\newcommand{\lem}{\begin{lemma}}
\newcommand{\elem}{\end{lemma}}
\newcommand{\defi}{\begin{definition}}
\newcommand{\edefi}{\end{definition}}
\newcommand{\ex}{\begin{example}}
\newcommand{\eex}{\end{example}}
\newcommand{\alis}{\begin{align*}}
\newcommand{\ealis}{\end{align*}}
\newcommand{\pf}{\begin{proof}}
\newcommand{\epf}{\end{proof}}
\newcommand{\ali}{\begin{align}}
\newcommand{\eali}{\end{align}}
\newcommand{\qus}{\begin{question}}
\newcommand{\equs}{\end{question}}

\newcommand{\mc}{\mathcal}
\renewcommand{\bf}{\textbf}

\newcommand{\C}{\mathbb{C}}

\newcommand{\sub}{\subset}

\newcommand{\ov}{\overline}

\newcommand{\bb}{\mathbb}

\newcommand{\op}{\operatorname}

\renewcommand{\a}{\alpha}
\renewcommand{\b}{\beta}
\renewcommand{\d}{\partial}
\newcommand{\e}{\epsilon}

\newcommand{\g}{\gamma}

\newcommand{\s}{\sigma}

\renewcommand{\l}{l}
\renewcommand{\o}{\omega}

\newcommand{\fk}{\frak}

\newcommand{\G}{\Gamma}

\renewcommand{\O}{\mathcal{O}}

\renewcommand{\S}{\Sigma}

\renewcommand{\ov}{\overline}

\newcommand{\tU}{\tilde{U}}

\newcommand{\SLC}{\operatorname{SL}(2, \mathbb{C})}
\newcommand{\SL}{\op{SL}_2}
\newcommand{\su}{\op{SU}(2)}
\newcommand{\HP}{\mathit{HP}}
\newcommand{\Perv}{\op{\bf{Perv}}}

\begin{document}
\title{A sheaf-theoretic $\SLC$ Floer homology for knots}

\author[Laurent C\^ot\'e]{Laurent C\^ot\'e}
\thanks {LC was supported by a Stanford University Benchmark Graduate Fellowship.}
\address{Department of Mathematics, Stanford University, 450 Serra Mall, Stanford, CA 94305}
\email{lcote@stanford.edu}

\author[Ciprian Manolescu]{Ciprian Manolescu}
\thanks {CM was supported by NSF grant DMS-1708320.}
\address {Department of Mathematics, UCLA, 520 Portola Plaza, Los Angeles, CA 90095}
\email {cm@math.ucla.edu}

\begin{abstract}
Using the theory of perverse sheaves of vanishing cycles, we define a homological invariant of knots in three-manifolds, similar to the three-manifold invariant constructed by Abouzaid and the second author. We use spaces of $\SLC$ flat connections with fixed holonomy around the meridian of the knot. Thus, our invariant is a sheaf-theoretic $\SLC$ analogue of the singular knot instanton homology of Kronheimer and Mrowka. We prove that for two-bridge and torus knots, the $\SLC$ invariant is determined by the $l$-degree of the $\widehat{A}$-polynomial. However, this is not true in general, as can be shown by considering connected sums of knots.
\end {abstract}

\maketitle

\section{Introduction}
Floer's instanton homology \cite{Floer} is an important gauge-theoretic invariant of homology $3$-spheres, defined using $\su$ connections in the trivial bundle. It was later extended to other bundles over $3$-manifolds; see \cite{FloerSurgery, DonaldsonBook}. Moreover, Kronheimer and Mrowka \cite{KMknots} developed a similar invariant for knots in three-manifolds, which they called {\em singular instanton homology}. Their construction uses $\su$ connections on the three-manifold that are singular along the knot, having traceless holonomy around the meridian. (The trace zero condition is needed to ensure monotonicity, so that an index bound on the moduli spaces of Floer trajectories gives an energy bound.) The singular instanton homology of knots can be used to derive lower bounds for the slice genus of knots \cite{KMs}, and was notably used in the proof that Khovanov homology detects the unknot \cite{KMUnknot}.

In \cite{abou-man}, motivated by Witten's work on Khovanov homology \cite{witten}, Abouzaid and the second author defined a homological invariant of three-manifolds using $\SLC$ instead of $\su$ connections. When working with complex gauge groups, defining Floer homology in the usual way is difficult because of noncompactness issues. However, one also expects no trajectories between different components of the space of $\SLC$ connections, which should make things easier. Indeed, the construction in \cite{abou-man} uses only sheaf theory, and none of the analysis characteristic of gauge theory. The resulting invariant, denoted $\HP^*(Y)$, is called the {\em sheaf-theoretic $\SLC$ Floer cohomology} of the three-manifold $Y$.

The purpose of this paper is to construct an invariant similar to $\HP^*$ for knots in three-manifolds. In the spirit of Kronheimer and Mrowka's definition of singular instanton homology, we use $\SLC$ connections on the knot complement for which the trace of the holonomy around the meridian of the knot is fixed to be some value $\tau \in (-2,2)$. Note that, since we do not work with moduli spaces of trajectories, we do not have to worry about monotonicity. Also, for simplicity, we will restrict our attention to irreducible connections (unlike in the work of Kronheimer and Mrowka).

Here is a sketch of the construction. Consider a doubly-pointed knot $K \subset Y$, that is, a knot with two distinct points fixed on it. We choose a Heegaard decomposition of $Y$ of genus $g \geq 6$, such that the Heegaard surface $\Sigma$ intersects $K$ transversely in the two points. We consider the relative character variety $X^{\tau}_{irr}(\Sigma)$, made of irreducible $\SLC$ connections on $\Sigma$ with holonomy trace $\tau$ around the two points. This is a complex symplectic manifold, which can be identified with a space of parabolic Higgs bundles on the surface. The two handlebodies yield complex Lagrangians $L_0$, $L_1 \subset X^{\tau}_{irr}(\Sigma)$, which can be equipped with spin structures. The intersection $L_0 \cap L_1$ is an oriented d-critical locus in the sense of Joyce \cite{joyce}. Applying the work of Bussi \cite{bussi}, we obtain from here a perverse sheaf of vanishing cycles, $P^{\bullet}_{L_0, L_1}$, on $L_0 \cap L_1$. Observe that $L_0 \cap L_1$ is the relative character variety $X^{\tau}_{irr}(Y)$ associated to $Y$, defined by asking the holonomy around the knot meridian to have trace $\tau$.

\begin{theorem}
\label{thm:main}
Let $Y$ be a closed, connected, oriented three-manifold, and $K \subset Y$ a doubly-pointed knot. Then, the object $P^{\bullet}_{\tau}(K):=P^{\bullet}_{L_0, L_1}$ is an invariant of $Y$ and $K$, up to canonical isomorphism in a category  $\textup{\bf{Perv}}'(X^{\tau}_{irr}(Y))$ of perverse sheaves on $X^{\tau}_{irr}(Y)$.

As a consequence, its hypercohomology
$$  \HP^*_{\tau}(K) := \bb{H}^*(P^{\bullet}_{\tau}(K))$$
is also an invariant of $Y$ and $K$, well-defined up to canonical isomorphism in the category of $\Z$-graded Abelian groups. 
\end{theorem}

We refer to Section~\ref{sec:def} for an exact definition of the category $\Perv'(X^{\tau}_{irr}(Y))$.

The proof of Theorem~\ref{thm:main} goes along the same lines as the proof of the corresponding result for closed three-manifolds in \cite{abou-man}. It involves checking invariance under stabilization, and naturality. In the process we have to establish certain properties of relative character varieties (e.g. simple connectivity) that were not immediately available in the literature.

We will refer to $\HP^*_{\tau}(K)$ as the 
\emph{$\tau$-weighted sheaf-theoretic $\op{SL}(2, \C)$-Floer cohomology} of the knot $K \subset Y$.
When computing $\HP^*_{\tau}$ for various knots, we will write $M_{(k)}$ for an Abelian group $M$  supported in degree $k$.

By analogy with the signed count of flat connections with fixed meridian trace in the $\su$ case (cf. \cite{lin-casson} and \cite{herald}), we make the following definition.

\begin{definition}
We let the \emph{($\tau$-weighted, sheaf-theoretic) $\op{SL}(2, \C)$ Casson-Lin invariant} of $K$ be  
the Euler characteristic of $\HP^*_{\tau}(K)$, denoted $ \chi_{\tau}(K) \in \Z.$
\end{definition}

Let us now assume that $Y$ is an integral homology sphere. For a knot $K \subset Y$, consider the character variety of the complement $Y - \op{nbhd}(K)$ and its image in the character variety of the boundary torus, $X(T^2) \cong (\C^* \times \C^*)/\Z_2$. The one-dimensional components of this image are (roughly) the zero set of a polynomial in two variables, $m$ and $l$, corresponding to the meridian and the longitude of the knot. This is the $A$-polynomial of the knot, introduced in \cite{planecurves}. Boyer and Zhang \cite{boyer-zhang} have a variant called the $\widehat{A}$-polynomial, $\widehat{A}(m, l)$, which also keeps track of the degrees of the maps between the one-dimensional components of $X(Y - \op{nbhd}(K))$ and $X(T^2)$. 

In simple cases, the relative character variety $X^{\tau}_{irr}(Y)$ is a finite number of points, obtained by intersecting $X(Y - \op{nbhd}(K))$ with the preimage of the hyperplane
$$ (\{\tau\} \times \C^*)/\Z_2 \subset (\C^* \times \C^*)/\Z_2 \cong X(T^2).$$
If so, we expect $ \HP^*_{\tau}(K)$ to be isomorphic to several copies of $\Z$, all in degree zero, and the number of copies to be given by the $l$-degree of the $\widehat{A}$-polynomial. We prove that this is the case in the following situation. 

\begin{theorem}
\label{thm:degd}
Suppose that $Y$ is an integral homology sphere and $K \subset Y$ is a knot such that the character scheme $\mathscr{X}^{\tau}(Y - K)$ is reduced and one-dimensional. Then, for all but finitely many $\tau \in (-2,2)$, we have that
\eq
\label{eq:degd}
 \HP^*_{\tau}(K) =  \Z^d_{(0)},
 \eeq
where $d=\op{deg}_{l} \widehat{A}(m, l)$. \end{theorem}

In fact, in Section~\ref{section:constantdim1} we derive precise conditions on the values of $\tau$ for which we can guarantee that \eqref{eq:degd} holds. 

To be able to use Theorem~\ref{thm:degd}, we need to identify classes of knots $K$ that satisfy the hypotheses. There are geometric conditions that ensure that the character scheme is one-dimensional. For example, it suffices for the knot $K$ to be either
\begin{itemize}
\item {\em small}, i.e., such that $Y-K$ does not contain a closed, orientable, essential surface; or
\item {\em slim}, i.e., hyperbolic and such that every component of ${X}^{\tau}(Y - K)$ which contains an irreducible representation also contains a discrete, faithful representation.
\end{itemize}

When one of these conditions is satisfied, one still needs to check by hand whether the character scheme is reduced. This can be done in practice, for example, for two-bridge knots, torus knots, and some pretzel knots. 

\begin{theorem}
\label{thm:concrete}
Suppose that $K \sub S^3$ is a two-bridge knot, a torus knot, or a $(-2,3, 2n+1)$ pretzel knot where $n \neq 0,1,2$ and $2n+1$ is not divisible by $3$. Then, for all but finitely many $\tau \in (-2,2)$, the equality \eqref{eq:degd} is satisfied. 
\end{theorem}

For specific knots, we can study the character scheme in more detail and calculate the sheaf-theoretic Floer cohomology  $\HP^*_{\tau}(K)$ even for non-generic $\tau$. For example, when $K=3_1$ is the trefoil, we find that
$$\HP^*_{\tau}(3_1)=\begin{cases}
\Z_{(0)} & \text{if } \tau \in (-2,2) \setminus \{\sqrt{3}, - \sqrt{3}\},\\
0 &\text{if } \tau \in \{\sqrt{3}, -\sqrt{3}\}.
\end{cases}$$

On the other hand, for the figure-eight knot $K=4_1$, we have
$$\HP^*_{\tau}(4_1)= \Z^2_{(0)}, \ \ \text{for all } \tau \in (-2,2).$$

In general, the invariant  $\HP^*_{\tau}(K)$ does not always have to be supported in degree zero. An example of this is given by considering connected sums of knots in $S^3$, which may have higher-dimensional character varieties. 

\begin{theorem}
\label{thm:finalconnectedsum} For $i=1,2$, suppose that $K_i \sub S^3$ is a two-bridge knot, a torus knot, or a $(-2,3, 2n+1)$ pretzel knot where $n \neq 0,1,2$ and $2n+1$ is not divisible by $3$.  Let $K= K_1 \# K_2$. Then, for all but finitely many $\tau \in (-2,2)$, we have that $\HP^*_{\tau}(K)$ is supported in degrees $-1$ and $0$ and, in fact, 
$$\HP^*_{\tau}(K) = \Z_{(-1)}^{ (\op{deg}_{\l} \widehat{A}(K_1)) \cdot (\deg_{\l} \widehat{A}(K_2))} \oplus \Z_{(0)}^{(\deg_{\l} \widehat{A}(K_1)) +(\deg_{\l} \widehat{A}(K_2))+ (\op{deg}_{\l} \widehat{A}(K_1)) \cdot (\deg_{\l} \widehat{A}(K_2))}.$$ 
\end{theorem}

We remark that in this paper we only considered the case $\tau \in (-2,2)$, so that the holonomy around the knot meridian is an elliptic element of $\SLC$. However, we expect that similar invariants exist for any $\tau \in \C^* \setminus \{-2,2\}$, and that they have similar properties. The only difficulty in carrying out the same constructions is that the topology of the spaces $X^{\tau}_{irr}(\Sigma)$ is less understood for $\tau \not \in (-2,2)$. Specifically, one would need to extend the results in Appendix I by showing that certain moduli spaces of stable $K(D)$-pairs (in the terminology of \cite{boden-yok}) are connected and simply connected.

We also expect that one can define {\em framed} versions of $P^{\bullet}_{\tau}(K)$ and $\HP^*_{\tau}(K)$, similar to the framed invariants $P^{\bullet}_{\#}(Y)$ and $\HP^*_{\#}(Y)$ of closed three-manifolds from \cite{abou-man}. These would take into account the reducible connections in addition to the irreducibles. Such a construction would be in fact closer to Kronheimer and Mrowka's singular knot instanton homology.

We end by raising a few questions for further investigation.

\qus
For an arbitrary knot $K \subset Y$, is $\HP^*_{\tau}(K)$ independent of $\tau$ (up to isomorphism), for generic $\tau \in (-2,2)$? 
\equs

\qus
Is the $\SLC$ Casson-Lin invariant additive, i.e., do we have $\chi_{\tau}(K_1 \# K_2) =\chi_{\tau}(K_1) + \chi_{\tau}(K_2)$ for any $K_1, K_2 \subset S^3$ and $\tau \in (-2,2)$? (See Theorem~\ref{theorem:euler} for a partial result in this direction.) 
\equs

\qus
If a knot $K \subset S^3$ satisfies $\HP^*_{\tau}(K)=0$ for all $\tau \in (-2,2)$, is $K$ the unknot?
\equs

The paper is organized as follows. Section~\ref{section:relativerepschemes} contains some background about relative character schemes and parabolic group cohomology. Section~\ref{section:definitionofinvariant} contains the definition of the knot invariants $P^{\bullet}_{\tau}(K)$ and $\HP^*_{\tau}(K)$, and the proof of Theorem~\ref{thm:main}. In Section~\ref{sec:tools} we develop some general tools for computing perverse sheaves of vanishing cycles. In Section~\ref{section:constantdim1} we apply these tools to relate our knot invariants to the $\widehat{A}$-polynomial; in particular, we prove Theorem~\ref{thm:degd}. Section~\ref{section:1dcomputations} contains concrete calculations for various small and slim knots; it is here where we prove Theorem~\ref{thm:concrete}. In Section~\ref{section:connectedsums} we study the invariants for connected sums of knots, and prove Theorem~\ref{thm:finalconnectedsum}. Finally, in the appendix (Section~\ref{section:parabolichiggs}), we establish a few facts about the topology of relative character varieties, which are used in Section~\ref{section:definitionofinvariant}.

\medskip

 \bf{Acknowledgements.}  We would like to thank Hans Boden, Julian Chaidez, Brian Conrad, Yasha Eliashberg, Tony Feng, Maxim Jeffs, Michael Kapovich, Nikolas Kuhn, Aaron Landesman, Ben Lim, Ikshu Neithalath, Dat Nguyen, Jacob Rasmussen, and Semon Rezchikov for helpful conversations during the preparation of this paper. We also thank the referee for comments on a previous version.

\section{Preparatory material}  \label{section:relativerepschemes}

This preparatory section introduces some material from algebraic geometry that will be needed throughout the paper. In particular, we introduce \emph{relative} versions of representation and character schemes. We also discuss \emph{parabolic group cohomology}, which is a variant of ordinary group cohomology. All of these objects have previously appeared in the literature (see e.g. \cite{kapovich}), although they seem to have been mostly considered implicitly. 

The material in this section is technical and the proofs will not be used in the remainder of the paper. The reader may therefore wish to treat this section as a reference once she has acquainted herself with the main definitions. 

\subsection{Representation varieties and their relative counterparts}  \label{subsection:relrepgeneralities} Let $\G= \langle g_1,\dots,g_k \mid r_1,\dots,r_l \rangle$ be a finitely-presented group. The $\op{SL}(2,\C)$-\emph{representation variety} of $\G$ is the set of group homomorphisms $$R(\G)= \op{Hom}(\G, \op{SL}(2,\C)).$$ 

If we view $\op{SL}(2,\C)$ as an algebraic subset of $\C^4$, then $R(\G)$ can naturally be viewed as an algebraic subset of $\C^{4k}$. Indeed, the condition that each $g_i \in \G$ maps to a matrix of determinant $1$ is described by a set of polynomial equations $f_1,\dots,f_k$. The relations $r_1,\dots,r_l$ then impose additional equations $f_{k+1},\dots,f_{k+l}$. 

We can also consider the \emph{representation scheme} \eq \label{equation:repscheme} \mathscr{R}(\G)= \op{Spec} \lt( \C[x_1,\dots,x_{4k}]/ (f_1,\dots,f_{k+l}) \rt).\eeq In the language of scheme theory, the representation variety $R(\G) \sub \mathscr{R}(\G)$ can be viewed as the reduced subscheme associated to $\mathscr{R}(\G)$.  

Representation varieties play an important role in the work of Abouzaid and the second author \cite{abou-man}. In the present paper, we need to consider certain generalizations which we call \emph{relative representation varieties}. Relative representation varieties are just subvarieties of ordinary representation varieties which parametrize representations $\G \to \op{SL}(2,\C)$ with fixed trace on certain conjugacy classes of $\G$. In order to define these objects precisely, however, it is useful to take a more abstract perspective. 

We begin by considering the following enlargement of the category of groups. 

\defi \label{definition: gp+} Let $\bf{Gp}^+$ be the category whose objects consist of a finitely-presented group $\G$ along with a set of distinct conjugacy classes $\fk{c}_1,\dots,\fk{c}_n$ (where $n \geq 0$ depends on the particular object). If $n=0$ the set of conjugacy classes is empty. An arrow $(\G; \fk{c}_1, \dots, \fk{c}_n) \to (\G'; \fk{c}_1', \dots, \fk{c}_m')$ is simply a morphism of groups which sends $\bigcup_i \fk{c}_i$ into $\bigcup_j \fk{c}_j'.$  \edefi

Observe that $\bf{Gp}^+$ contains the ordinary category of groups as a subcategory. 

Let $\C \bf{-alg}$ denote the category of commutative algebras over the complex numbers. Given an object $(\G; \fk{c}_1, \dots, \fk{c}_n) \in \bf{Gp}^+$ and a parameter $\tau \in \C$, we can consider the functor $$\mathscr{R}^{\tau}(\G; \fk{c}_1,\dots,\fk{c}_n): \C \bf{-alg} \to \bf{Sets}$$ taking $A \mapsto \{\rho: \G \to \op{SL}(2, A) \mid \op{Tr}(\rho(h)) = \tau \;  \text{for all } h \in \bigcup_i \fk{c}_i\}.$ 

To lighten the notation, we will usually write $\mathscr{R}^{\tau}(\G)$ in place of $\mathscr{R}^{\tau}(\G; \fk{c}_1,\dots,\fk{c}_n)$ unless we want to keep track of the conjugacy classes. If $n=0$, the choice of $\tau$ is irrelevant. In this case, we write $\mathscr{R}(\G)$ instead of $\mathscr{R}^{\tau}(\G)$. This notation will be clarified by following proposition; cf.\ \Cref{remark:schemenotation}. 

\prop \label{proposition:representable} \label{proposition:representablefunctor} The functor $\mathscr{R}^{\tau}(\G): \C \textup{\bf{-alg}} \to \textup{\bf{Sets}}$ is representable, i.e., there exists a  $\C$-algebra $\mc{A}^{\tau}(\Gamma)$ such that
$$ \mathscr{R}^{\tau}(\G)= \operatorname{Hom}_{\C\textup{\bf{-alg}}}(\mc{A}^{\tau}(\Gamma), \textendash).$$
 \eprop  
\pf  We first consider the case $n=0$. It is not hard to verify that $\mathscr{R}(\G)$ is in fact represented by the $\C$-algebra $\C[x_1,\dots,x_{4k}]/ (f_1,\dots,f_{k+l})$ introduced in \eqref{equation:repscheme}. Let us call this $\C$-algebra $\mc{A}(\G)$. It follows from general facts of category theory that $\mc{A}(\G)$ is the unique representative, up to canonical isomorphism. In particular, it is independent of the presentation of $\G$.  We refer the reader to \cite[Prop.\ 1.2]{lub-mag} for a detailed exposition of these arguments. 

If $n>0$, we choose a group element $c_i \in \fk{c}_i$ for $i=1,\dots,n$. The condition $\op{Tr}(\rho(c_i))=\tau$ now imposes additional polynomial equations $f_{k+l+1},\dots,f_{k+l+n}$. It can then be shown as in the $n=0$ case that \eqs \mc{A}^{\tau}(\G):= \mc{A}(\G)/ (f_{k+l+1},\dots,f_{k+l+n}) \eeqs represents the functor $\mathscr{R}^{\tau}(\G)$. The representative $\mc{A}^{\tau}(\G)$ is again unique up to canonical isomorphism, and in particular independent of the choice of $c_i \in \fk{c}_i$. \epf

Generalizing \eqref{equation:repscheme}, we define \eqs \mathscr{R}^{\tau}(\G):= \op{Spec} \mc{A}^{\tau}(\G). \eeqs We say that $\mathscr{R}^{\tau}(\G)$ is the \emph{relative representation scheme} associated to the parameter $\tau \in \C$ and to the object $(\G; \fk{c}_1,\dots,\fk{c}_n) \in \bf{Gp}^+$. In general $\mathscr{R}^{\tau}(\G)$ may be singular and non-reduced. The \emph{relative representation variety} $R^{\tau}(\G) \sub \mathscr{R}^{\tau}(\G)$ is the reduced subscheme associated to $\mathscr{R}^{\tau}(\G)$. 

\rmk \label{remark:schemenotation} If $A$ is a $\C$-algebra, then its image under the functor $\mathscr{R}^{\tau}(\G)$ is the set $\mathscr{R}^{\tau}(\G)(A)$. It follows from \Cref{proposition:representable} that this set coincides with the set of $A$-valued points of the scheme $\mathscr{R}^{\tau}(\G)$. This explains why we use the same notation to refer to a functor and to a scheme. \ermk

\ex \label{example: relrepexample} Suppose that $\G= \langle a_1, \dots, a_g, b_1, \dots, b_g, c_1, c_2 \mid r_1, \dots, r_m \rangle.$ Let $\fk{c}_1= \op{Conj}(c_1)$ and $\fk{c}_2= \op{Conj}(c_2).$ Then $\mathcal{A}(\G)$ is the $\C$-algebra formed from the polynomial algebra $$\C[x^{a_1}_{11}, x^{a_1}_{12}, 
x^{a_1}_{21}, x^{a_1}_{22}, \dots, x^{c_2}_{11}, x^{c_2}_{12},x^{c_2}_{21}, x^{c_2}_{22}]$$ by first modding out by the relations $\{ x^{\a}_{11}x^{\a}_{22}-x^{\a}_{12}x^{\a}_{21} -1 =0\}$ where $\a$ ranges over the generators of $\G$, and then modding out by the relations coming from $r_1,\dots,r_m$.  To obtain $\mathcal{A}^{\tau}(\G)$, one simply adds the relations $x^{c_1}_{11}+x^{c_1}_{22}-\tau=0$ and $x^{c_2}_{11}+x^{c_2}_{22}-\tau=0$. This example will be important in the sequel.
\eex

The algebraic group $\op{SL}_2$ acts by conjugation on $\mathscr{R}^{\tau}(\G)$. We define the \emph{relative character scheme} $\mathscr{X}^{\tau}(\G)$ to be the GIT quotient of this action. The \emph{relative character variety} $X^{\tau}(\G) \sub \mathscr{X}^{\tau}(\G)$ is the reduced subscheme associated to $\mathscr{X}^{\tau}(\G)$. We let $\mathscr{R}^{\tau}_{irr}(\G) \subset \mathscr{R}^{\tau}(\G)$ be the open subscheme corresponding to irreducible representations, and we define ${R}^{\tau}_{irr}(\G)$, $\mathscr{X}^{\tau}_{irr}(\G)$ and $X^{\tau}_{irr}(\G)$ similarly.

As in the case of ordinary representation and character schemes, it can be shown that a morphism $(\G; \fk{c}_1,\dots,\fk{c}_n) \to (\G'; \fk{c}_1',\dots,\fk{c}_m')$ in $\bf{Gp}^+$ induces morphisms $\mathscr{R}^{\tau}(\G') \to \mathscr{R}^{\tau}(\G)$ and $\mathscr{X}^{\tau}(\G') \to \mathscr{X}^{\tau}(\G)$. This is explained in \cite[p.\ 6]{lub-mag} for ordinary character schemes, and the arguments generalize to the relative case.

\subsection{Fiber products} \label{subsection: fiberproducts}

Suppose that $\mc{G}, \mc{H}_0, \mc{H}_1, \mc{J}$ are objects of $\bf{Gp}^+$ whose underlying groups are $\G, \Pi_0, \Pi_1$ and $\Pi_0 *_{\G} \Pi_1$ respectively. Suppose moreover that the following pushout diagram of groups is induced by arrows in $\bf{Gp}^+$:
\eq \label{equation:diagram0}
\begin{tikzcd}
\G \arrow[r] \arrow[d] & \Pi_1 \arrow{d} \\
\Pi_0 \arrow[r] & \Pi_0 *_{\G} \Pi_1.
\end{tikzcd}
\eeq

\lem \label{lemma:fiberrep} There is a canonical isomorphism of schemes $\mathscr{R}^{\tau}(\Pi_0 *_{\G} \Pi_1) \xrightarrow{\sim} \mathscr{R}^{\tau}(\Pi_0) \tms_{\mathscr{R}^{\tau}(\G)} \mathscr{R}^{\tau}(\Pi_1).$ \elem

\pf We will show that both schemes are spectra of $\C$-algebras which represent the same same functor. Indeed, one has: 
\begin{align*}
\mathscr{R}^{\tau}(\Pi_0) \tms_{\mathscr{R}^{\tau}(\G)} \mathscr{R}^{\tau}(\Pi_1) &=
\op{Hom}_{\C\textbf{-alg}}(\mathcal{A}^{\tau}(\Pi_0) \otimes_{\mathcal{A}^{\tau}(\G)} \mathcal{A}^{\tau}(\Pi_1), - ) \\
&= \op{Hom}_{\C\textbf{-alg}}(\mathcal{A}^{\tau}(\Pi_0), -) \tms_{\op{Hom}_{\C\textbf{-alg}}(\mathcal{A}^{\tau}(\G),-)} \op{Hom}_{\C\textbf{-alg}}(\mathcal{A}(\Pi_1)^{\tau}, -)  \\
&= \op{Hom}_{\textup{\bf{Gp}}^+}(\Pi_0, \op{SL}_2(-)) \tms_{\op{Hom}_{\textup{\bf{Gp}}^+}(\G, \op{SL}_2(-))} \op{Hom}_{\textup{\bf{Gp}}^+}(\Pi_1, \op{SL}_2(-)) \\
&= \op{Hom}_{\textup{\bf{Gp}}^+}(\Pi_0 *_{\G} \Pi_1, \op{SL}_2(-)) \\
&= \mathscr{R}^{\tau}(\Pi_0 *_{\G} \Pi_1).
\end{align*}
It now follows by Yoneda's lemma that there is a unique isomorphism $\mathscr{R}^{\tau}(\Pi_0 *_{\G} \Pi_1) \xrightarrow{\sim} \mathscr{R}^{\tau}(\Pi_0) \tms_{\mathscr{R}^{\tau}(\G)} \mathscr{R}^{\tau}(\Pi_1).$ \epf 
 
\prop[cf.\ Prop 2.10 in \cite{marche}] \label{proposition:surjective} Suppose that the group homomorphisms $\G \to \Pi_0$ and $\G \to \Pi_1$ in \eqref{equation:diagram0} are surjective. Then there is a unique isomorphism of schemes 
\eq \label{equation: characterisom} \mathscr{X}^{\tau}(\Pi_0 *_{\G} \Pi_1) \xrightarrow{\sim} \mathscr{X}^{\tau}(\Pi_0) \tms_{\mathscr{X}^{\tau}(\G)} \mathscr{X}^{\tau}(\Pi_1). \eeq
 \eprop
 
 \pf 
 
The surjectivity of the morphisms $ \G \to \Pi_i$ implies that $\mathscr{R}^{\tau}(\Pi_i) \to \mathscr{R}^{\tau}(\G)$ is a closed embedding of affine schemes.  Hence there is a surjective morphism $\mathcal{A}^{\tau}(\G) \to \mathcal{A}^{\tau}(\Pi_i).$ We let $I_i \sub \mathcal{A}^{\tau}(\G)$ be the kernel of this morphism. 

According to \Cref{lemma:fiberrep} we have a canonical isomorphism 
\begin{align*} \mathcal{A}^{\tau}(\Pi_0 *_{\G} \Pi_1) \xrightarrow{\sim} \mathcal{A}^{\tau}(\Pi_0) \otimes_{\mathcal{A}^{\tau}(\G)} \mathcal{A}^{\tau}(\Pi_1) =\mathcal{A}^{\tau}(\G)/ I_0 \otimes_{\mathcal{A}^{\tau}(\G)} \mathcal{A}^{\tau}(\G)/I_1. 
\end{align*}

We now take invariants under the action of $\op{SL}_2$ by conjugation.  By standard algebraic manipulations, we have $\lt( \mathcal{A}^{\tau}(\G)/ I_0 \otimes_{\mathcal{A}^{\tau}(\G)} \mathcal{A}^{\tau}(\G)/I_1 \rt)^{\op{SL}_2} = \lt( \mathcal{A}^{\tau}(\G) / (I_0 + I_1) \rt)^{\op{SL}_2}  = ( \mathcal{A}^{\tau}(\G) )^{\op{SL}_2} / (I_0 + I_1)^{\op{SL}_2}.$ For the last equality, we used the fact that $\op{SL}_2$ is linearly reductive, which implies that the functor of invariants $M \to M^{\op{SL}_2}$ on $\op{SL}_2$-modules is exact; see for example \cite[Proposition 4.37]{mukai}.

Next, we will need the following fact.

\begin{fact} Let $A$ be a $\C$-algebra with an action of a linearly reductive group $G$ and suppose that $I, J$ are ideals of $A$. Then $(I+J)^{G} = I^{G} + J^{G}.$  \end{fact}

\pf
Consider the short exact sequence
$$ 0 \longrightarrow I \cap J  \longrightarrow I \oplus J  \longrightarrow I+J  \longrightarrow 0.$$
Using again the exactness of the functor of invariants, we get a short exact sequence
$$0 \longrightarrow (I \cap J)^G  \longrightarrow (I \oplus J)^G  \longrightarrow (I+J)^G  \longrightarrow 0.$$
On the other hand, we clearly have $(I \cap J)^G = I^G \cap J^G$ and $(I \oplus J)^G = I^G \oplus J^G$, and the cokernel of the map between them can be identified with $I^G + J^G$.
\epf 

Applying this fact for $G=\SL$, we can now write: 
\begin{align*}
( \mathcal{A}^{\tau}(\G) )^{\op{SL}_2} / (I_0 + I_1)^{\op{SL}_2}  &= ( \mathcal{A}^{\tau}(\G) )^{\op{SL}_2} / ( I_0^{\op{SL}_2} + I_1^{\op{SL}_2}  ) \\
&= ( \mathcal{A}^{\tau}(\G) )^{\op{SL}_2}/ I_0^{\op{SL}_2}  \otimes_{( \mathcal{A}^{\tau}(\G) )^{\op{SL}_2}} ( \mathcal{A}^{\tau}(\G) )^{\op{SL}_2}/ I_1^{\op{SL}_2} \\
&= ( \mathcal{A}^{\tau}(\G) / I_0 )^{\op{SL}_2}  \otimes_{( \mathcal{A}^{\tau}(\G) )^{\op{SL}_2}}  ( \mathcal{A}^{\tau}(\G) / I_1 )^{\op{SL}_2}.  
\end{align*}

It follows that $ \lt( \mathcal{A}^{\tau}(\G)/ I_0 \otimes_{\mathcal{A}^{\tau}(\G)} \mathcal{A}^{\tau}(G)/I_1 \rt)^{\op{SL}_2} =  ( \mathcal{A}^{\tau}(\G) / I_0 )^{\op{SL}_2}  \otimes_{( \mathcal{A}^{\tau}(\G) )^{\op{SL}_2}}  ( \mathcal{A}^{\tau}(\G) / I_1 )^{\op{SL}_2},$ which is equivalent to \eqref{equation: characterisom}.  \epf

\prop \label{proposition: pushoutisom}  Under the hypotheses of \Cref{proposition:surjective}, there is a unique isomorphism of schemes 
\eq \label{equation: irrcharacterisom} \mathscr{X}_{irr}^{\tau}(\Pi_0 *_{\G} \Pi_1) \xrightarrow{\sim} \mathscr{X}_{irr}^{\tau}(\Pi_0) \tms_{\mathscr{X}_{irr}^{\tau}(\G)} \mathscr{X}_{irr}^{\tau}(\Pi_1). \eeq
 \eprop

\pf It follows from general properties of the fiber product of schemes that there is an open embedding $\mathscr{X}_{irr}^{\tau}(\Pi_0) \tms_{\mathscr{X}_{irr}^{\tau}(\G)} \mathscr{X}_{irr}^{\tau}(\Pi_1) \to \mathscr{X}^{\tau}(\Pi_0) \tms_{\mathscr{X}^{\tau}(\G)} \mathscr{X}^{\tau}(\Pi_1) \simeq \mathscr{X}_{irr}^{\tau}(\Pi_0 *_{\G} \Pi_1).$  Since $\mathscr{X}_{irr}^{\tau}(\Pi_0 *_{\G} \Pi_1)$ is also an open subscheme of $\mathscr{X}_{irr}^{\tau}(\Pi_0 *_{\G} \Pi_1)$, it suffices to show that it has the same closed points as $\mathscr{X}_{irr}^{\tau}(\Pi_0) \tms_{\mathscr{X}_{irr}^{\tau}(\G)} \mathscr{X}_{irr}^{\tau}(\Pi_1)$. This can readily be checked by observing that, since $\G \to \Pi_i$ is surjective, a representation $\rho_i: \Pi_i \to \op{SL}(2, \C)$ is irreducible if and only if it pulls back to an irreducible representation of $\G \to \op{SL}(2, \C).$ 
\epf

\subsection{Parabolic group cohomology}  We make a brief digression from our discussion of relative representation and character schemes. As above, let $\G$ be a finitely-presented group. Let $G$ be a complex-algebraic group and let $\fk{g}$ be its Lie algebra. In the sequel, we only need to consider $G=\op{SL}(2,\C)$ and $\fk{g}= \fk{sl}(2, \C)$, but there is no reason to restrict the following discussion to this particular case.


Let us briefly recall the co-cycle construction of the group cohomology of $\G$ with coefficients in the representation $\op{Ad} \circ \rho: \G \to  \op{Aut}(\fk{g})$. Here $\op{Ad}: G \to \op{Aut}(\fk{g})$ is the adjoint representation, given by  $\op{Ad}(g): \xi \mapsto g\xi g^{-1}$.

Consider the abelian groups 
$$C^n(\G; \op{Ad} \rho):= \{ \text{functions}: \G \to \fk{g}\},$$ where the group structure is given by addition of functions. One can define chain maps $$d^{n+1}: C^n(\G, \op{Ad} \rho) \to C^{n+1}(\G, \op{Ad} \rho)$$ by a standard formula; see \cite[p.\ 59]{brown}.  In the case $n=1$ which is the only one which we will use, we have that 
\eqs d^2(\phi)(g_1, g_2) = \op{Ad}\rho(g_1) \phi(g_2)- \phi(g_1g_2) + \phi(g_1), \eeqs
for $g_1, g_2 \in \G$. 

The subgroup of elements $Z^n(\G; \op{Ad} \rho)= \op{ker} d^{n+1}$ are said to be $n$-cocycles, while the subgroup of elements $B^n(\G; \op{Ad} \rho)=\op{im} d^n$ are $n$-coboundaries. The group $$H^n(\G; \op{Ad} \rho)= Z^n(\G; \op{Ad} \rho)/ B^n(\G; \op{Ad} \rho)$$ is the $n$-th cohomology group of $\G$ with coefficients in the representation $\op{Ad}  \rho: \G \to \fk{g}.$  

We now consider the data of an object $(\G; \fk{c}_1, \dots, \fk{c}_n)$ of $\bf{Gp}^+$. 

We let $$Z_{par}^1((\G; \fk{c}_1, \dots, \fk{c}_n); \op{Ad} \rho) \sub Z^1(\G; \op{Ad} \rho)$$ to be the set of $1$-cocycles whose restriction to any element of $\bigcup_i \fk{c}_i$ is a boundary. Said differently, an element $\phi \in Z^1(\G; \op{Ad} \rho)$ is contained in the subset $ Z_{par}^1((\G; \fk{c}_1, \dots, \fk{c}_n); \op{Ad} \rho)$ if, for all $g \in \bigcup_i \fk{c}_i$, there exists $\mu \in \fk{g}$ such that $\phi(g)= \mu- \op{Ad}_{\rho}(g) \mu$. The elements of $Z_{par}^1((\G; \fk{c}_1, \dots, \fk{c}_n); \op{Ad} \rho)$ are said to be \emph{parabolic} 1-cocycles. 

The quotient $$H^1_{par}((\G; \fk{c}_1, \dots, \fk{c}_n); \op{Ad} \rho ):= Z_{par}^1((\G; \fk{c}_1, \dots, \fk{c}_n); \op{Ad} \rho) / B^1(\G; \op{Ad} \rho)$$ is called the \emph{first parabolic cohomology group} of $\G$ with coefficients in $\op{Ad} \rho: \G \to \fk{g}.$  If the set of conjugacy classes is empty, this recovers the ordinary notion of group cohomology. There is also a notion of parabolic group cohomology in higher degrees. Since we will not use it in this paper, we refer the interested reader to \cite[p.\ 347]{hida}.

For future reference, note that an arrow $\phi: (\G; \fk{c}_1, \dots, \fk{c}_n) \to (\G'; \fk{c}_1', \dots, \fk{c}_m') $ induces a map of abelian groups $H^1_{par}((\G'; \fk{c}_1', \dots, \fk{c}_n'); \op{Ad} \rho) \to H^1_{par}((\G; \fk{c}_1, \dots, \fk{c}_n); \op{Ad} \rho \circ \phi)$. 

\rmk \label{remark:parabolicnotation} To lighten the notation, we will often write $H^1_{par}(\G; \op{Ad} \rho )$ and $Z^1_{par}(\G; \op{Ad} \rho)$ when the conjugacy classes are understood from the context. \ermk

\subsection{Tangent spaces} The relevance of parabolic group cohomology to our discussion of relative representation and character schemes is evidenced by the following proposition. 

\prop \label{proposition:tangentspace} Let $(\G; \fk{c}_1,\dots,\fk{c_n}) \in \bf{Gp}^+$ and suppose that $\rho \in R^{\tau}(\G)$ for some $\tau \neq \pm 2$. We then have:
\begin{enumerate}
\item $T_{\rho} \mathscr{R}^{\tau}_{}(\G) \simeq Z^1_{par}(\G; \op{Ad} \rho).$
\item $T_{[\rho]} \mathscr{X}^{\tau}_{irr}(\G) \simeq H^1_{par}(\G; \op{Ad} \rho).$ 
\end{enumerate}
\eprop

In order to prove the first statement, it is convenient to introduce the set $\mc{T}_{\rho} \mathscr{R}_{}^{\tau}(\G)$. This is defined as the set of all maps $\rho: \G \to \op{SL}(2, \C[\e]/(\e^2))$ satisfying the conditions:
\begin{itemize}
\item[(i)] $\op{Tr}(\rho(h))= \tau$ for all $h \in \bigcup_i \mathfrak{c}_i$
\item[(ii)] The composition $\G \to \op{SL}(2, \C[\e]/(\e^2)) \to \op{SL}(2, \C)$ agrees with $\rho,$ where $\op{SL}(2, \C[\e])/(\e^2) \to \op{SL}(2, \C)$ is the map induced by the $\C$-algebra morphism $\C[\e]/(\e^2) \to \C$ mapping $\e \to 0.$
\end{itemize}

The set $\mc{T}_{\rho} \mathscr{R}_{}^{\tau}(\G)$ can be canonically identified with the tangent space $T_{\rho} \mathscr{R}^{\tau}_{}(\G)$; see \cite[p.\ 14]{sik}.  This identification is a general consequence of the fact that $\mathscr{R}_{}^{\tau}(\G)$ represents a functor $\C\bf{-alg} \to \bf{Sets}$, and turns out to be convenient for the purpose of proving \Cref{proposition:tangentspace}.

\pf[Proof of \Cref{proposition:tangentspace} (1)] By the previous discussion, it is enough to show that there is a bijection $Z^1_{par}(\G; \op{Ad}(\rho)) \to \mc{T}_{\rho} \mathscr{R}_{}^{\tau}(\G).$ As in \cite{sik}, consider the map 
\begin{align*}
Z^1_{par}(\G; \op{Ad} \rho) &\to \mc{T}_{\rho} \mathscr{R}_{}^{\tau}(\G) \\
\s &\to (\g \in \G \mapsto (I+ \s(\g) \e) \rho(\g)).
\end{align*}

Let us first check that this map is well defined, i.e. that the image satisfies conditions (i) and (ii) above. Condition (ii) is clearly satisfied. To check condition (i), suppose that $h \in \bigcup_i \mathfrak{c}_i.$ Then $\op{Tr}( (I+ \s(h) \e)\rho(h))= \op{Tr}(\rho(h))+ \e \op{Tr}(\s(h) \rho(h)) = \tau+ \e \op{Tr}(\s(h) \rho(h)).$  But since $\s \in Z^1_{par}(\G; \op{Ad} \rho),$ it follows that $\s(h)= \b- \rho(h) \b \rho(h)^{-1}$ for some $\b \in \fk{g}$ depending on $h.$  Hence $\op{Tr}(\s(h) \rho(h))=0$ as desired. 

Injectivity follows from the analogous statement on \cite[page 15]{sik}.  To show surjectivity, we note that \cite[page 15]{sik} shows that every element of $\mc{T}_{\rho} \mathscr{R}^{\tau}_{}(\G)$ can be written in the form $(\g \mapsto (I+ \s(\g) \e) \rho(\g))$ for some $\s \in Z^1(\G; \op{Ad}(\rho)).$   Hence it suffices to check that $\s \in Z^1_{par}(\G; \op{Ad}(\rho)) \sub Z^1(\G; \op{Ad}(\rho)).$

In other words, given $h \in \G$ with $\op{Tr}(\rho(h)) = \tau \neq \pm 2$, we must check that $\op{Tr}((I+ \s(h) \e)(\rho(h)))= \op{\tau} \Leftrightarrow \op{Tr}(\s(h) \rho(h))=0  \Leftrightarrow  \s(h)= \b - \rho(h) \b \rho^{-1}(h),$ for some $\b \in \fk{g}.$ This follows from the following linear algebra fact. \epf

\begin{fact} \label{fact: trace} Let $B \in \op{SL}(2, \C)$ and $A \in \fk{sl}(2, \C)$ such that $\op{Tr}(B) \neq \pm 2$ and $\op{Tr}(AB)=0$. Then there exists $C \in \fk{sl}(2, \C)$ such that $A= C- BCB^{-1}.$ \end{fact}
\pf 
Fixing $B$, evidently we have $\{ C-BCB^{-1} \mid  C \in \fk{sl}(2, \C) \} \sub \{A \in \fk{sl}(2, \C) \mid \op{Tr}(AB)=0\}.$ By direct computation, one can check that both the left hand side and the right hand side are two dimensional, which gives the desired equality. \epf 

The proof of \Cref{proposition:tangentspace} (2) is essentially identical to that of \cite[Theorem 53]{sik}, except that one considers the relative representation and character schemes $\mathscr{R}_{irr}^{\tau}(\G)$ and $\mathscr{X}^{\tau}_{irr}(\G)$ in place of the ordinary representation and character schemes (which are denoted by $\mc{H}om(\G, G)$ and $\mc{X}_{G}(\G)$ in Sikora's notation).  

The key tool from algebraic geometry is the Luna Slice Theorem; see for example \cite{drezet}. If $\rho$ is an irreducible representation, its orbit in $\mathscr{R}^{\tau}(\G)$ can be shown to be closed. Since $\rho$ is irreducible, its stabilizer $S_{\rho}=\{\pm I\}$ is precisely the center of $G=\op{SL}(2,\C)$. The Luna Slice Theorem then implies that there exists a closed subscheme $S \sub \mathscr{R}^{\tau}(\G)$, usually called an \emph{ \'{e}tale slice}, with the following properties. We have $[\rho] \in S$ and there is a natural map $(G/S_{\rho}) \tms S \to \mathscr{X}^{\tau}(\G)$ sending $(g,s) \mapsto gs$ which is \'{e}tale. Similarly, the projection map $S \to \mathscr{X}^{\tau}(\G)$ is \'{e}tale.  

One useful consequence of the existence of \'{e}tale slices which will be used later on is the fact that $R^{\tau}_{irr}(\G)$ is a $G^{ad}$-bundle over $X^{\tau}_{irr}(\G)$; see \cite[Ex.\ 2.1.1.4.(iii)]{schmitt}. Here $G^{ad}: = G/Z(G)=G/S_{\rho}$. 

\pf[Proof of \Cref{proposition:tangentspace} (2)]  The inclusion $\mathscr{X}^{\tau}_{irr}(\G) \to R^{\tau}_{irr}(\G)$ induces a surjective morphism of tangent spaces $$T_{\rho} R^{\tau}_{irr}(\G) \to T_{[\rho]} \mathscr{X}^{\tau}_{irr}(\G).$$ We saw in (1) that $T_{\rho} R^{\tau}_{irr}(\G)= Z(\G; \op{Ad} \rho)$, so it only remains to show that the kernel is $B^1(\G; \op{Ad} \rho)$. This can be done as in the proof of \cite[Theorem 53]{sik} by appealing to the existence of an \'{e}tale slice as discussed above, along with the fact that $S_{\rho}= Z(G)$ necessarily acts trivially on this slice. \epf

\defi A representation $\rho \in R^{\tau}(\G)$ is said to be {\em regular} if the scheme $\mathscr{R}^{\tau}(\G)$ is regular at $[\rho]$. It is said to be reduced if $\mathscr{R}^{\tau}(\G)$ is reduced at $[\rho]$. \edefi

We end this section by recording the following proposition, which is an analog of Proposition 2.3 and Lemma 2.4 in \cite{abou-man}. It can be proved by modifying the arguments provided in \cite{abou-man} and \cite[Corollary 55]{sik}  but we omit the details.

\prop \label{proposition:regularitycriterion} An irreducible representation $\rho \in R^{\tau}(\G)$ is regular (reduced) if and only if the scheme $\mathscr{X}^{\tau}(\G)$ is regular (reduced) at $[\rho]$. \eprop

\section{Definition of the knot invariant} \label{section:definitionofinvariant}

We now pass to the construction of our knot invariant. The general scheme is similar to that of \cite[Sec.\ 7.1]{abou-man}. Starting from a Heegaard splitting $Y= U_0 \cup_{\S} U_1$, we get complex Lagrangians $L_0= X^{\tau}_{irr}(U_0)$ and $L_1= X^{\tau}_{irr}(U_1)$ in the relative character variety $X^{\tau}_{irr}(\S)$. Appealing to a construction of Bussi in \cite{bussi}, we define a perverse sheaf $P^{\bullet}_{L_0, L_1}$ on the intersection $L_0 \cap L_1= X^{\tau}_{irr}(Y)$. We then show that $P^{\bullet}_{L_0, L_1}$ is suitably independent of the choice of Heegaard splitting and thus defines a topological invariant. This will prove Theorem~\ref{thm:main}.

\subsection{Preliminary definitions}  \label{subsection:preliminarydefs}  Although the basic strategy for defining our invariant mirrors that of \cite[Sec.\ 7.1]{abou-man}, some technicalities occur in our setting which were not present in the original construction. These are mainly due to working with manifolds with boundary. In particular, we eventually wish to appeal to work of Juh\'{a}sz, D. Thurston and Zemke in order to prove the naturality of our invariant. This leads us to use the language of sutured manifolds and to introduce certain auxiliary categories. 

\defi A \emph{sutured manifold} $(M, \g)$ is the data of a compact, oriented $3$-manifold $M$ with nonempty boundary, along with a disjoint union of oriented simple closed curves $\g= \bigcup_i \g_i \sub \d M$. One requires that $\g$ separates $\d M$ into two components $R^+(\g)$ and $R^-(\g)$, where $R^+(\g) \cup R^-(\g)= \d M$ and $\d R^+(\g)= \g$, $\d R^-(\g)= - \g$. \edefi

The set of sutured manifolds naturally forms a category whose arrows are diffeomorphisms $(M, \g) \to (M', \g')$ sending $R^{+} (\g)$ to $R^{+}(\g')$ and $R^{-} (\g)$ to $R^{-}(\g')$.

Let us also introduce the category $\bf{Knot}_{**}$ of doubly-pointed knots. An object $(Y, K, p,q)$ of $\bf{Knot}_{**}$ consists of an oriented $3$-manifold $Y$, an oriented knot $K \sub Y$, and an ordered pair of basepoints $p,q \in K \sub Y$. A morphism $(Y, K, p, q) \to (Y', K', p', q')$ is an orientation-preserving diffeomorphism $Y\to Y'$ which sends $K,p,q$ to $K',p', q'$ respectively. 

There is a functor from $\bf{Knot}_{**}$ to the category of sutured manifolds which will be very important to us and which we refer to as the \emph{(spherical) blowup}. This was also considered by Juh\'{a}sz, Thurston and Zemke in \cite[Definition 2.5]{juhasz-thurston}. Given a doubly-pointed knot $(Y,K, p, q)$, for every $x \in K$, let $N_xK= T_xY/T_xK$ be the fiber of the normal bundle of $K$ over $x$, and let $c_x := UN_xK = (N_xK \setminus \{0\})/\R_+$ be the fiber of the unit normal bundle to $K$ over $x$. The blowup $(E_K, \g)$ of $(Y,K,p,q)$ is obtained by replacing $x$ with $c_x$ for all $x \in K$ and letting $\g=c_p \cup c_q$. Note that the interior of $E_K$ is diffeomorphic to the knot exterior of $K \sub Y$. 

The orientations on $Y$ and on $K$ induce an orientation on $NK$ and hence on $\d E_K$. This orientation is compatible with the boundary orientation of $\d E_K$ inherited from the orientation of $E_K$ (which is itself inherited from the orientation of $Y$).  We orient $c_p$ coherently with respect to the orientation of $N_pK$ and $c_q$ incoherently with respect to the orientation of $N_qK$. This uniquely determines $R^+(\g)$ and $R^-(\g)$. 

A detailed construction of the blowup functor is provided in \cite{aro-kan}. It is shown in particular that a morphism $(Y, K, p, q) \to (Y', K', p', q')$ induces a morphism $(E_K, \g) \to (E_{K'}, \g')$. The image of $\bf{Knot}_{**}$ under the blowup functor forms a subcategory of the category of sutured manifolds. We call this subcategory $\bf{SutKnot}$. In the sequel, we will view the blowup as a functor $\op{Bl}: \bf{Knot}_{**} \to \bf{SutKnot}$ taking $(Y,K, p, q) \mapsto (E_K, \g)$.

Given a surface $\S$ with nonempty boundary, an attaching set is the data of a disjoint union $\mathbf{\a} = \bigcup_i \a_i$ of pairwise disjoint simple closed curves such that each component of $\S- \bigcup_i \a_i$ contains a component of $\d \S$. An attaching set is said to be maximal if it is not contained in a strictly larger attaching set, i.e. if the collection of curves is maximal. 

\defi Given $(E_K, \g) \in \bf{SutKnot}$, a \emph{Heegaard splitting} $\mc{H}$  consists in a decomposition $E_K= U_0 \cup U_1$ into handlebodies, where $U_0 \cap U_1= \S$ is a surface with boundary $\d \S= c_p \cup c_q$.  We require moreover that there exist maximal attaching sets $\mathbf{\a}, \mathbf{\b} \sub \S$, where the $\a_i$ and $\b_i$ bound disks in $U_0$ and $U_1$ respectively.  \edefi

It is shown in \cite[Lem.\ 2.15]{juhasz-thurston} that any sutured manifold, so in particular every object of $\bf{SutKnot}$, admits a Heegaard splitting.

\subsection{Geometric setup} \label{subsection:symplecticgeom}

Let us now fix an object $(Y, K, p,q)$ of $\bf{Knot}_{**}$ and let $(E_K, \g)$ be its blowup. We fix a Heegaard splitting $\mc{H}= (\S, U_0, U_1)$ for $(E_K, \g)$ and choose a basepoint $x_0 \in c_p$.  

Such data determines the following commutative diagram in the category $\bf{Gp}^+$: 
\eq \label{equation: diagram1}
\begin{tikzcd}
\lt(\pi_1({\S}, x_0); \op{Conj}([c_p]), \op{Conj}([c_q])  \rt) \arrow[r] \arrow[d] & (\pi_1({U}_1, x_0); \op{Conj}([c_p])) \arrow{d} \\
(\pi_1({U}_0, x_0); \op{Conj}([c_p])) \arrow[r] & (\pi_1(E_K, x_0); \op{Conj}([c_p])).
\end{tikzcd}
\eeq

It follows from the existence of two maximal attaching sets bounding disks that the inclusions $\iota_i: {\S} \to {U}_i$ induce surjective maps $(\iota_i)_*: \pi_1({\S}, x_0) \to \pi_1( {U}_i, x_0).$  By van Kampen's theorem, the underlying diagram of groups is isomorphic to a pushout diagram of the form \eqref{equation:diagram0}. 

It now follows by \Cref{proposition: pushoutisom} that there is a unique isomorphism 
\eq \label{equation:fiberprod} \mathscr{X}_{irr}^{\tau}(\pi_1(E_K,x_0)) \simeq \mathscr{X}^{\tau}_{irr}(\pi_1({U}_0,x_0)) \tms_{\mathscr{X}^{\tau}_{irr}(\pi_1(\S, x_0))} \mathscr{X}^{\tau}_{irr}(\pi_1(U_1,x_0)). \eeq
In other words, the scheme-theoretic intersection of $\mathscr{X}^{\tau}_{irr}(\pi_1({U}_0,x_0))$ and $\mathscr{X}^{\tau}_{irr}(\pi_1({U}_1,x_0))$ can be identified with $\mathscr{X}_{irr}^{\tau}(\pi_1(E_K,x_0))$.

\prop \label{proposition: smoothcharacterscheme}For $\tau \neq \pm 2$, the relative character scheme $\mathscr{X}^{\tau}_{irr}(\pi_1(\S, x_0))$ is a smooth scheme of dimension $6g-2$. It follows that $\mathscr{X}^{\tau}_{irr}(\pi_1(\S, x_0))= X^{\tau}_{irr}(\pi_1(\S,x_0))$ and that the set of closed points of $X^{\tau}_{irr}(\pi_1(\S, x_0))$ forms a smooth complex manifold of dimension $6g-2$. \eprop

\pf  We will show in Appendix I (see \Cref{mainequivalence} and \Cref{theorem: simplyconnected}) that the character variety $X^{\tau}_{irr}(\pi_1(\S,x_0))$ is smooth and of dimension $6g-2.$  According to \cite[p.\ 156]{weil}, the group $H^1_{par}(\pi_1(\S, x_0); \op{Ad} \rho)$ also has dimension $6g-2.$ Since $T_{[\rho]} \mathscr{X}^{\tau}_{irr}(\S) \simeq H^1_{par}(\pi_1(\S, x_0); \op{Ad} \rho)$ by \Cref{proposition:tangentspace}, this implies that $\mathscr{X}^{\tau}_{irr}(\pi_1(\S,x_0))$ is a smooth scheme. \epf

\prop \label{proposition:charvarietysmooth} For $\tau \neq \pm 2$, the character scheme $\mathscr{X}_{irr}^{\tau}(\pi_1(U_i,x_0))$ is smooth of dimension $3g-1.$ 
\eprop
\pf Note that $\pi_1({U}_i, x_0)=F_{g+1}$. If we fix a generating set $a_0, a_1,\dots, a_g$, then the relative representation scheme is isomorphic to $g$ copies of the scheme $\op{SL}_2$ and one copy of the scheme $\op{Spec} A^{\tau}$ for $A^{\tau}= k[x_1, x_2, x_3, x_4]/ (x_1x_4-x_2x_3-1, x_1+x_4-\tau)$.  One can check that $\op{Spec} A^{\tau}$ is a smooth scheme provided that $\tau \neq \pm 2$. It follows that $\mathscr{R}^{\tau}(\pi_1(U_i,x_0))$ is a smooth scheme.

\Cref{proposition:regularitycriterion} now implies that $\mathscr{X}_{irr}^{\tau}(\pi_1(U_i,x_0))$ is also a smooth scheme, which therefore has the same dimension as $X_{irr}^{\tau}(\pi_1(U_i,x_0))$. To compute this dimension, note that $\op{SL}(2, \C)$ has dimension $3$ while the conjugacy class of elements of trace $\tau$ has dimension $2$ when $\tau \neq \pm 2$.  Hence $\op{dim} R^{\tau}_{irr}(\pi_1({U}_i,x_0))= 3g+2$ and so $\op{dim} X^{\tau}_{irr}(\pi_1(U_i,x_0))= 3g+2-3=3g-1.$ \epf

The relative character varieties which arise from the Heegaard splitting $(\S, U_0, U_1)$ can be identified with certain moduli spaces of flat connections. Let us first describe this identification in the case of $X^{\tau}_{irr}(E_K)$. We let $G= \op{SL}(2,\C)$ and consider the trivial $G$-bundle $E_K \tms G$. A flat connection $A$ on $E_K \tms G$ gives rise to a holonomy representation $\pi_1(E_K, x_0) \to G$ by parallel transporting the fiber at $x_0$. One can check that the action of the gauge group corresponds precisely to conjugation in $G$. If $A$ has the property that its holonomy along $c_p$ has trace $\tau$, then the associated representation defines a point of $X^{\tau}(E_K)$. 

It can be shown that this map (or rather its inverse) defines a bijection between $X^{\tau}(\pi_1(E_K,x_0))$ and the moduli space of flat connections on $E_K \tms G$ whose holonomy has trace $\tau$ along $c_p$. This moduli space can be given the structure of a complex-analytic space, with respect to which the above bijection becomes an isomorphism of complex-analytic spaces. An analogous identification works if we replace $E_K$ with the $U_i$ or $\S$. In the latter case, we need to consider representations whose holonomy has trace $\tau$ along both $c_p$ and $c_q$. 

We will refer to the moduli space of flat connections described above as providing an \emph{analytic model} for $X^{\tau}(\pi_1(E_K,x_0))$. In contrast, we think of the relative character variety as defined in \Cref{subsection:relrepgeneralities} as the algebraic model. Observe that the analytic model is defined independently of any choice of basepoint, while the algebraic model uses the basepoint $x_0$. (The identification between the algebraic and analytic model, described above, commutes with the natural isomorphism of algebraic models induced by changing the basepoint.) We will therefore write $X^{\tau}(E_K)$ in place of $X^{\tau}(\pi_1(E_K, x_0))$ when we wish to consider the analytic model, and similarly for the $U_i$ and $\S$. 

The equivalence between analytic and algebraic models also induces an identification of tangent spaces. This is the content of \Cref{proposition:firstgeometric}, as we now explain. For notational simplicity, let $M$ be either $E_K, U_i, \S$. We define $H^k_{A_{\rho}}(M; \fk{g})$ to be the $k$-th de Rham cohomology group of $M$ with twisted coefficients in the (restriction of the) vector bundle $E_K \tms G \tms_{Ad} \fk{g}$ with respect to the irreducible connection $A_{\rho}$ on $E_K \tms G$. Let $\rho: \pi_1(M, x_0) \to G$ be the holonomy representation induced by $A_{\rho}$. We now consider the following diagram, where the horizontal arrows come from the Mayer-Vietoris sequence for cohomology with local coefficients and the vertical arrows are induced by the inclusion maps. 

\eq  \label{equation: diagram3}
\begin{tikzcd}
H^0_{A_{\rho}}({\S}; \fk{g}) \arrow[r] & H^1_{A_{\rho}}(E_K; \fk{g}) \arrow[d, "a"] \arrow[r] & H^1_{A_{\rho}}({U}_0; \fk{g}) \oplus H^1_{A_{\rho}}({U}_1; \fk{g}) \arrow{d}[left]{b_0}[right]{b_1} \arrow[r] & H^1_{A_{\rho}}({\S}; \fk{g}) \arrow[d, "c"]\\
& H^1_{A_{\rho}}(c_p; \fk{g}) & H^1_{A_{\rho}}(c_p; \fk{g}) \oplus  H^1_{A_{\rho}}(c_p; \fk{g}) & H^1_{A_{\rho}}(c_p; \fk{g}) \oplus H^1_{A_{\rho}}(c_q; \fk{g})
\end{tikzcd}
\eeq

We have the following proposition, which is proved in Section~\ref{sec:tgspaces}.

\prop \label{proposition:firstgeometric} The following groups are canonically isomorphic:  
\begin{itemize}
\item $\op{ker}(a) \simeq H^1_{par}( (\pi_1(E_K, x_0); \op{Conj}(c_p)); \op{Ad} \rho),$ 
\item $\op{ker}(b_0) \simeq H^1_{par}( (\pi_1({U}_0, x_0); \op{Conj}(c_p)); \op{Ad} \rho),$ 
\item $\op{ker}(b_1) \simeq H^1_{par}( (\pi_1({U}_1, x_0); \op{Conj}(c_p)); \op{Ad} \rho),$ 
\item $\op{ker}(c) \simeq H^1_{par}( (\pi_1({\S},x_0); \op{Conj}(c_p), \op{Conj}(c_q)); \op{Ad} \rho).$ 
\end{itemize}
\eprop

\Cref{proposition:firstgeometric} expresses the isomorphism of tangent space between analytic and algebraic models. Indeed, again letting $M$ be $E_K, U_i, \S$, the reader may verify that the left hand side is the tangent space at the irreducible flat connection $A_{\rho}$ in the analytic model of $X^{\tau}_{irr}(M)$; cf.\ the discussion following \cite[Prop.\ 2.11]{mondello}. The right hand side is the algebraic tangent space at the representation $\rho$, as we saw in \Cref{proposition:tangentspace}.  Observe that the right hand side of course depends on the basepoint (up to canonical isomorphism), while the left hand side does not. 

As a consequence of the above discussion, observe that $T_{\rho} X_{irr}^{\tau}({\S})$ can be viewed naturally as a subspace of $H_{A_{\rho}}^1({\S}; \fk{g})$. As shown for instance in \cite{biswas-guru}, the variety $X_{irr}^{\tau}({\S})$ carries a natural holomorphic symplectic form (i.e. a global section of the bundle $\Omega^{2,0}(X_{irr}^{\tau}(\S))$ of $(2,0)$-forms).  If $[\a], [\b] \in  T_{\rho} X_{irr}^{\tau}({\S}) \sub  H_{A_{\rho}}^1({\S}; \fk{g})$ for $\a_1, \a_2 \in \Omega^1({\S}; \fk{g})$, then 
\eq \o_{\C}(\a, \b)=   \int_{{\S}} \op{Tr}( \a_1 \wedge \a_2 ). \eeq

We have the following important fact.

\prop \label{proposition:lagrangian} For $i=1,2$, the natural embeddings $X^{\tau}_{irr}({U}_i) \to X^{\tau}_{irr}({\S})$ are Lagrangian with respect to $\o_{\C}$. \eprop

\pf It follows by combining \Cref{proposition: smoothcharacterscheme} and \Cref{proposition:charvarietysmooth} that $X_{irr}^{\tau}({U_i})$ is half-dimensional. The inclusion of tangent spaces $T_{\rho} X_{irr}^{\tau}({U}_i) \to T_{\rho} X_{irr}^{\tau}({\S})$ corresponds to the restriction map $\iota_i^*: \Omega_{A_{\rho}}^1({U}_i) \to \Omega_{A_{\rho}}^1({\S})$. 

Since $\a, \b$ represent classes in $H^1_{A_{\rho}}({U}_i; \fk{g})$, it follows that they are exact in a neighborhood of $\d U_i$. 

By Stokes' theorem (for manifolds with corners; see \cite[p.\ 415]{lee}), we have \begin{align*}
\o_{\C}(\a, \b)= \int_{{\S}}\op{Tr}(\a \wedge \b) &= \int_{\d U_i} d \op{Tr}(\a \wedge \b) = \int_{U_i} \op{Tr}(d\a \wedge \b) + \op{Tr}( a \wedge d \b) =0.
\end{align*}
This concludes the proof.
 \epf
 
Following \Cref{proposition:lagrangian}, it will be convenient to write $L_i^{\tau}= X_{irr}^{\tau}({U}_i) \sub X_{irr}^{\tau}({\S})$. 


\subsection{Identification of the tangent spaces} \label{sec:tgspaces} We now prove \Cref{proposition:firstgeometric}. We will treat only the first three isomorphisms since the last one appears already in \cite{biswas-guru}. However, our argument will easily generalize to also cover the last case. For the remainder of this section, we drop the conjugacy classes from our notation for the parabolic cohomology groups; cf. \Cref{remark:parabolicnotation}.

In the remainder of this section, let $M$ be either of the three manifolds $E_K, {U}_0$ or ${U}_1.$  We set $\G= \pi_1(M,x_0).$  The inclusion $i: c_p \hookrightarrow M$ induces a map $i_*: \pi_1(c_p,x_0) =\Z \to \G.$  This in turn induces maps $i^*:  Z^1(\G; \op{Ad} \rho) \to Z^1(\pi_1(c_p); \op{Ad} \rho) \to H^1(\pi_1(c_p); \op{Ad} \rho).$  

The following technical lemma is an important ingredient in the proof of \Cref{proposition:firstgeometric}. To fix some terminology, we will say that a 1-cocyle $\phi: \G \to \fk{g}$ restricts to a boundary on an element $g \in G$ if $\phi(g)= \mu - \op{Ad}_g \mu,$ for some $\mu \in \fk{g}.$

\prop \label{proposition: conjugation} Let $\phi \in Z^1(\G; \op{Ad} \rho)$ be a 1-cocycle.  Then $\phi$ restricts to a boundary on $g \in \G$ if and only if $\phi$ restricts to a boundary on all $g_i \in \op{Conj}(g).$ \eprop

\pf Suppose that $\phi(g)= \mu - \op{Ad}_g \mu,$ for some $\mu \in \fk{g}.$   Given $h \in G= \op{SL}(2, \C),$ we wish to show that there exists some $\tilde{\mu} \in \fk{g}$ such that $\phi(h g h^{-1}) = \tilde{\mu} - \op{Ad}_{hgh^{-1}} \tilde{\mu}.$

By \Cref{fact: trace}, it is enough to show that $\phi(hgh^{-1}) hgh^{-1}$ is trace-free.  Since $\phi \in Z^1(\G; \op{Ad} \rho),$ we have $\phi(hgh^{-1})= \phi(h) + \op{Ad}_{h} (\phi(g) + \op{Ad}_g \phi(h^{-1})) = \phi(h) + \op{Ad}_h \phi(g)+ \op{Ad}_{hg} \phi(h^{-1}).$    Hence $\phi(hgh^{-1} )hgh^{-1} = \phi(h) h gh^{-1}+ h \phi(g) gh^{-1} + hg \phi(h^{-1}) h^{-1}.$

Observe first of all that $\op{Tr} ( h \phi(g) g h^{-1}) = \op{Tr} (\phi(g) g)= \op{Tr} ( \mu g - g \mu)= 0.$  Hence it is enough to show that $ \op{Tr} \lt( \phi(h) h gh^{-1} + hg \phi(h^{-1}) h^{-1} \rt)=0.$

Observe now that we have $\phi( h)=  \phi( h e) = \phi(h) + \op{Ad}_h \phi(e) .$ Hence $\op{Ad}_h \phi(e)=0,$ which implies $\phi(e)=0.$   It follows that $\phi(h^{-1}) = - h^{-1} \phi(h) h.$  Hence $\op{Tr}( hg \phi (h^{-1}) h^{-1} )= - \op{Tr} ( hg h^{-1} \phi(h)) = - \op{Tr} ( \phi(h) h gh^{-1}).$  This completes the proof. \epf

\cor We have $\op{ker} i^* = Z^1_{par}(\G; \op{Ad} \rho).$  It follows that the kernel of the natural map $$H^1(\G; \op{Ad} \rho) \to H^1( \pi_1(c_p) ; \op{Ad} \rho)$$ is precisely the group $H^1_{par}(\G; \op{Ad} \rho).$ \ecor

\pf If $\psi \in \op{ker} i^*,$ then $\phi([c_p]) = \mu - \op{Ad}_{[c_p]} \mu$ for some $\mu \in \fk{g}.$  It now follows from \Cref{proposition: conjugation} that $\phi \in Z^1_{par}(\G; \op{Ad} \rho).$ The reverse direction is obvious.  \epf

Let $H^*_{sing}$ denote singular cohomology. The first three cases of \Cref{proposition:firstgeometric} can now be seen to follow from existence of the following commutative diagram:
\eq 
\begin{tikzcd}
H^1_{A_{\rho}}(M; \fk{g}) \arrow[d] \arrow[r, "\sim"] & H^1_{sing}(M; \op{Ad} \rho)  \arrow{d} \arrow[r, "\sim"] &  H^1(\pi_1(M,x_0); \op{Ad} \rho) \arrow[d]\\
H^1_{A_{\rho}}(c_p; \fk{g}) \arrow[r, "\sim"] & H^1_{sing}(c_p; \op{Ad} \rho) \arrow[r, "\sim"] & H^1(\Z; \op{Ad} \rho)
\end{tikzcd}
\eeq

The two leftmost isomorphisms come from the fact that the complexes of abelian groups $\G(\Omega^*(M; \fk{g}))$ and $C^*(M) \otimes_{\Z[\pi_1(M,x_0)]} \fk{g}$ both arise as the global sections of flasque resolutions of the same local system. The local system in question consists in the sections of $M \tms G \tms_{Ad} \fk{g}$ which are flat with respect to $A_{\rho}.$ See \cite[chapter 5]{warner} for details.

The two rightmost isomorphisms can be taken as a definition if $M$ is an Eilenberg-MacLane space. In general, the fibration $\tilde{M} \to M \to B(\pi_1(M,x_0))$ gives rise to a spectral sequence 
\eqs H^p(\pi_1(M,x_0); H^1(\tilde{M}; \fk{g})) = E_2^{p,q} \Rightarrow H^{p+q}(M; \op{Ad} \rho). \eeqs 

Consider the induced exact sequence in low-degrees $0 \to E_2^{1,0} \to H^1(M; \op{Ad} \rho) \to E_2^{0,1}.$ In our case, $E_2^{0,1}=0$ since $H^1(\tilde{M}; \fk{g})=0.$ This implies $H^1(\pi_1(M); \fk{g}) \to H^1(M; \op{Ad} \rho)$ is an isomorphism. 


\subsection{Definition of the invariant}  \label{sec:def} Like Heegaard Floer homology and the sheaf-theoretic $\op{SL}(2,\C)$-homology of \cite{abou-man}, our knot invariant is built from a Heegaard splitting. As in \cite{abou-man}, the class of objects which we associate to Heegaard splittings are perverse sheaves. It will be important to define precisely the category in which they live.

Let us therefore introduce the category $\bf{PSh}$. The objects of this category are triples
$$ (E_K, \g, \mc{F})$$
where $(E_K, \g)$ is an object of $\bf{SutKnot}$ and $\mc{F}$ is a perverse sheaf on $X^{\tau}_{irr}(E_K)$. For simplicity, when the underlying sutured manifold is understood from the context, we will  write $\mc{F}$ for the whole triple.

Let $(E_K, \g, \mc{F})$ and $(E_{K'}, \g', \mc{G})$ be two objects of $\bf{PSh}$. A morphism $f$ between them is defined to be a pair $(f^*, \phi)$ where $f^*$ is the pullback induced by a diffeomorphism $f: (E_K, \g) \to (E_{K'}, \g')$ and $\phi: f^* \mc{G} \to \mc{F}$ is a morphism of perverse sheaves on $X^{\tau}_{irr}(E_K)$. Morphisms compose according to the rule $(g^*, \psi) \circ (f^*, \phi)= (f^* g^*, \phi \circ f^*(\psi))$, and one can check that this rule satisfies the required axioms. 

Finally, we note that the category $\Perv(X^{\tau}_{irr}(E_K))$ of perverse sheaves on $X^{\tau}_{irr}(E_K)$ is a subcategory of $\bf{PSh}$. However, it is not a full subcategory, as it only includes morphisms for which $f$ is the identity diffeomorphism of $(E_K, \g)$. We denote by $\Perv'(X^{\tau}_{irr}(E_K))$ the full subcategory of $\bf{PSh}$ where the objects are $(E_K, \g, \mc{F})$ with fixed $(E_K, \g)$. The morphisms in this subcategory may involve non-trivial self-diffeomorphisms of $(E_K, \g)$.

For a doubly-pointed knot $(Y, K, p, q)$, we will consider Heegaard splittings of the blowup $(E_K, \g)$. To each such splitting $\mc{H}$, we will associate an object $P(\mc{H}) \in \bf{PSh}$. The construction goes as follows; cf.\ \cite[Sec.\ 7.1]{abou-man}. 

Recall from the previous section that $(X_{irr}^{\tau}({\S}), \o_{\C})$ is a complex symplectic manifold (i.e. $\o_{\C}$ is a holomorphic symplectic form). The submanifolds $L_j= X_{irr}^{\tau}(U_j)$ are complex Lagrangians (i.e. they are complex, and the restriction of $\o_{\C}$ to $L_j$ vanishes). Their intersection $L_0 \cap L_1$ can be identified as a complex analytic space with $X_{irr}^{\tau}(E_K)$. We will show in the next section (see \Cref{proposition: spinstructures}) that the $L_j$ carry unique spin structures.  It then follows from \cite[Thm.\ 2.1]{bussi} that we can associate to the above data a perverse sheaf $P^{\bullet}_{L_0, L_1}$ on $X_{irr}^{\tau}(E_K)$, which is unique up to canonical isomorphism in the category of perverse sheaves on $X_{irr}^{\tau}(E_K)$.  We write $P(\mc{H}) = P^{\bullet}_{L_0, L_1}$ and view $P(\mc{H})$ as an object of $\bf{PSh}$. 

We now wish to consider certain standard operations on Heegaard splittings. The three operations are \emph{diffeomorphism}, \emph{stabilization} and \emph{destabilization}, and we refer to these as Heegaard moves. We will now show that Heegaard moves induce isomorphisms of the associated objects in $\bf{PSh}$.

Given $(E_K,\g), (E_{K'}, \g') \in \bf{SutKnot}$, a diffeomorphism between two Heegaard splittings $\mc{H}= (\S, U_0, U_1)$ and $\mc{H}'= (\S', U_0', U_1')$ is a diffeomorphism $f: (E_K, \g) \to (E_{K'}, \g')$ such that $\S \to \S'$ and $U_i \to U_i'$. Let $\bb{Z}_{\S}$ and $\bb{Z}_{\S'}$ be the constant sheaf on $X^{\tau}_{irr}(\S)$ and $X^{\tau}_{irr}(\S')$ respectively, and note that there is a canonical isomorphism $\phi: f^*\bb{Z}_{X'} \to \bb{Z}_X$. The perverse sheaf $P^{\bullet}_{L_0, L_1}$ is built locally by applying the nearby cycle functor to $\bb{Z}_{X}$ and gluing the resulting perverse sheaves (suitably twisted by a bundle parametrizing spin strucutres). It can then be shown that $\phi$ induces an isomorphism $f^*P_{L_0', L_1'} \to P^{\bullet}_{L_0, L_1}$, which we also denote by $\phi$ by abuse of notation.

Thus, if $\mc{H}$ and $\mc{H}'$ are Heegaard splittings of $E_K$ which are related by a diffeomorphism $f: (E_K, \g) \to (E_K, \g)$, then there is an induced isomorphism $P(\mc{H}) \to P(\mc{H}')$ given by the pair $(f^*, \phi)$. 

Next, we consider the operation of stabilization. Since it is described in \cite[Sec.\ 7.2]{abou-man}, we will only give a brief review. Given a Heegaard splitting $\mc{H}= (\S, U_0, U_1)$, one drills out a solid torus $S$ from $U_1$, in such a way that the boundary of $S$ intersects $\S$ in a disk. This gives rise to a new Heegaard splitting $\mc{H}'= (\S', U_0', U_1')$ where $\S'= \S \cup \d S, U_0' = U_0 \cup S$ and $U_1'= U_1- S$. The splitting $\mc{H}'$ is said to be obtained from $\mc{H}$ by stabillization, and we write $\mc{H} \to \mc{H}'$.   

Following \cite[Prop.\ 7.3]{abou-man}, we claim that a stabilization $\mc{H} \to \mc{H}'$ induces an isomorphism $P^{\bullet}_{L_0, L_1} \to P^{\bullet}_{L_0', L_0'}$ of the associated perverse sheaves in $\bf{PSh}$. This is the content of the following proposition. 

\prop \label{proposition:stabilization} Suppose that $\mc{H}'$ is obtained from $\mc{H}$ by a stabilization and that $L_0, L_1$ and $L_0', L_1'$ are the complex Lagrangians arising from $\mc{H}, \mc{H}'$ respectively. Then there is an isomorphism of perverse sheaves $P^{\bullet}_{L_0', L_1'} \to P^{\bullet}_{L_0, L_1}$. \eprop

\pf One can essentially repeat the proof of Proposition 7.3 in \cite{abou-man}, replacing $Y$ with $E_K$ and replacing all character varieties with their relative counterparts. Two ingredients of the original argument in \cite{abou-man} need some care. First of all, the original argument uses the fact $X_{irr}(U_i)$ admits a unique spin structure.  We will establish in \Cref{proposition: spinstructures} that this remains true for $X_{irr}^{\tau}(U_i)$.  Secondly, the argument of \cite{abou-man} uses the fact $X_{irr}(\S)$ is connected and simply-connected.  This also remains true for $X_{irr}^{\tau}(\S)$ and will be established in Appendix I. \epf  

If $\mc{H}'$ is obtained from $\mc{H}$ by a stabilization, then we say that $\mc{H}$ is obtained from $\mc{H}'$ by a destabilization. We may thus associate to each destabilization an isomorphism of associated perverse sheaves, which is just the inverse of the isomorphism induced by the opposite stabilization. 

In order to extract a $3$-manifold invariant from our construction, we need the following version of the Reidemeister-Singer theorem. 

\prop \label{proposition:reid-sing} Suppose that $\mc{H}$ and $\mc{H}'$ are Heegaard splittings for $E_K$ of genus at least 6. Then they can be related by a sequence of stabilizations and destabilizations. We may assume that the genus of the Heegaard splitting never drops below six in this sequence. \eprop 

\pf First of all, note by \cite[Remark 7.2]{abou-man} that one can replace an isotopy by a sequence of stabilizations and destabilizations. Since each stabilization is followed by a destabilization, this does not cause the genus to drop. 

We now appeal to \cite[Prop.\ 2.15]{juhasz}. Up to isotopy, the Heegaard splittings $\mc{H}_0$ and $\mc{H}_1$ are induced by self-indexing Morse functions $f_0$ and $f_1$. The basic idea is to argue that there is 1-parameter family $f_t$ which is Morse for all $t$ aside from $0< t_1<\dots<t_n <1$ where the family has a birth-death critical point.  Translating from the language of functions to the language of handlebody decompositions, this means precisely that the Heegaard splittings can related by stabilizations, destabilizations and isotopies.  To ensure that the genus never drops below 6, one can appeal to standard arguments of Cerf theory \cite[Lem.\ 2.5(1)]{laudenbach} to ensure that the birth times happen before the death times.  \epf

As a corollary to Propositions \ref{proposition:stabilization} and \ref{proposition:reid-sing}, we find that the isomorphism class of $P(\mc{H})$ depends only on $(E_K, \g)$. Since every object $(E_K, \g) \in \bf{SutKnot}$ is the blowup of  a doubly-pointed knot $(Y, K, p, q) \in \bf{Knot}_{**}$, we conclude that the isomorphism class of $P(\mc{H})$ depends only on the underlying doubly-pointed knot. This will be strengthened in the next section. 

\subsection{Naturality} If $\mc{H}$ and $\mc{H}'$ are Heegaard splittings for $(E_K, \g) \in \bf{SutKnot}$ which are related a sequence of stabilizations and destabilizations, the induced morphism $P(\mc{H}) \to P(\mc{H}')$ which we introduced in the previous section could in principle depend on the choice of sequence. We will now argue that this morphism is in fact independent of the choice of sequence. We refer to this property as \emph{naturality}. 

The naturality of the original sheaf-theoretic invariant of Abouzaid and the second author was established in \cite[Sec.\ 7.3]{abou-man}. Their arguments apply in our setting with no essential modifications, so we give only a brief summary of the main steps. 

The naturality of $3$-manifold invariants built from Heegaard-type decompositions was studied in a general context by J\'{u}hasz, Thurston and Zemke \cite{juhasz-thurston}. Applied to our setting, their work implies that our invariant is indeed natural provided that the morphisms associated to the three Heegaard moves satisfy a list of axioms. (The most interesting of these axioms is invariance under a move called handleswap.) Thus, checking naturality of the invariant boils down to checking that these axioms are satisfied. This was done in \cite[Sec.\ 7.3]{abou-man} for the original sheaf-theoretic invariant on which our construction is based. In our setting, the axioms and  their verification are essentially the same; cf.\ \cite[Thm.\ 7.5 and 7.8]{abou-man}. 

\rmk The results of J\'{u}hasz, Thurston and Zemke are phrased in the language of sutured manifolds. This explains why sutured manifolds have also appeared in our construction. In contrast, there was no need for sutured manifolds in the original work of \cite{abou-man} since the manifolds under consideration were closed. \ermk

As a consequence of naturality, it now makes sense to define $$P^{\bullet}_{\tau}(K):= P(\mc{H})= P^{\bullet}_{L_0,L_1},$$ where $\mc{H}$ is any Heegaard splitting of $(E_K, \g)$ and $L_0, L_1$ are the Lagrangians associated to $\mc{H}$. Observe that $P^{\bullet}_{\tau}(K)$ is well-defined in the usual category-theoretic sense: if we had chosen a different Heegaard splitting, we could relate $P(\mc{H})$ and $P(\mc{H})$ by a unique isomorphism in $\bf{PSh}$.  

A morphism $(E_K, \g) \to (E_{K'}, \g')$ in $\bf{SutKnot}$ induces a diffeomorphism of Heegaard splittings $\mc{H}= (\S, U_0, U_1) \to \mc{H}'= (\S', U_0', U_1')$.  One can check that the induced map $P(\mc{H}) \to P(\mc{H}')$ commutes with stabilization -- in the case where $(E_K, \g) = (E_{K'}, \g')$, this is one of the axioms of \cite{juhasz-thurston}; see \cite[Thm.\ 7.5, 2(ii)]{abou-man}. Hence, we get an isomorphism $P^{\bullet}_{\tau}(K) \to P^{\bullet}_{\tau}(K')$ in $\bf{PSh}$. 

In summary, we have assigned to every object $(E_K, \g) \in \bf{SutKnot}$ a perverse sheaf $$P^{\bullet}_{\tau}(K) \in \Perv'(X^{\tau}_{irr}(E_K)) \subset \bf{PSh}$$ and to each morphism $(E_K, \g) \to (E_{K'}, \g')$ a morphism $P^{\bullet}_{\tau}(K) \to P^{\bullet}_{\tau}(K')$ in $\bf{PSh}$. One can check that this assignment satisfies the axioms of a functor. We can then obtain an invariant of doubly-pointed knots simply by precomposing this functor $\bf{SutKnot} \to \bf{PSh}$ with the blowup functor $\bf{Knot}_{**} \to \bf{SutKnot}$. The resulting composition associates to a doubly-pointed knot $(Y,K, p, q)$ the perverse sheaf $P^{\bullet}_{\tau}(K)$. Of course, this depends on the full data $(Y, K, p, q)$, as well as on the choice of $\tau \in (-2,2)$. 

One can also get a functor $\bf{Knot}_{**} \to \bf{Gp}$ by passing to hypercohomology, i.e. by post-composing the above functor $\bf{Knot}_{**} \to \bf{PSh}$ with the hypercohomology functor $\bf{PSh} \to \bf{Gp}$. As explained in \cite[p.\ 2]{abou-man}, if we fix a Heegaard splitting $\mc{H}=(\S, U_0, U_1)$ so that $P^{\bullet}_{\tau}(K)= P(\mc{H})= P^{\bullet}_{L_0, L_1}$, then $\HP^*_{\tau}(K)$ should correspond to the ordinary Lagrangian Floer cohomology of $L_0, L_1$ -- provided that the Floer cohomology can be defined. We therefore call 
$$\HP^*_{\tau}(K) := \bb{H}^*(P^{\bullet}_{\tau}(K))$$ the \emph{$\tau$-weighted sheaf-theoretic $\op{SL}(2, \C)$-Floer cohomology} of the doubly-pointed knot $(Y, K, p,q)$. 

This completes the proof of Theorem~\ref{thm:main}.


Observe that any pair of doubly-pointed knots $(Y, K, p, q)$ and $(Y, K, p', q')$ which share the same underlying knot are of course isomorphic objects in $\bf{Knot}_{**}$. It follows that the associated $\tau$-weighted sheaf-theoretic $\op{SL}(2, \C)$-Floer cohomology groups are isomorphic. It therefore makes sense to think of these groups as knot invariants, provided that one only considers them up to isomorphism (not canonical isomorphism).

Let us mention a couple of basic properties of the invariants we just defined. 
\prop
The perverse sheaf $P^{\bullet}_{\tau}(K)$ is Verdier self-dual. 
\eprop

\prop
If $-Y$ denotes $Y$ with the opposite orientation, then the perverse sheaves $P^{\bullet}_{\tau}$ associated to $K \subset Y$ and $K \subset -Y$ are isomorphic. In particular, for $Y=S^3$, if we let $m(K)$ denote the mirror of $K \subset S^3$, then $P^{\bullet}_{\tau}(K) \cong P^{\bullet}_{\tau}(m(K))$.
\eprop

The proofs of these propositions are entirely similar to those of the corresponding results in the closed case; cf. Propositions 8.1 and 8.2 in \cite{abou-man}.

\subsection{Spin structures} \label{subsection:spinstructures}We will now establish \Cref{proposition: spinstructures}, which was postponed from the previous section. This proposition explains why we need to assume that our Heegaard splitting has genus at least 6.  The arguments of this section will not be used subsequently, so the reader may freely pass to the next section. 

\prop \label{proposition: spinstructures} Suppose that $\S$ has genus $g \geq 6.$ Then $X_{irr}^{\tau}({U}_i)$ admits a unique spin structure. \eprop 

We begin with some preparatory remarks before proving \Cref{proposition: spinstructures}. 

Let $\tU_i$ be the handlebody of genus $g$ obatined from $U_i$ by blowing down around the knot $K$, i.e. filling in a cylinder with boundaries $\operatorname{bl}(p)$ and $\operatorname{bl}(q)$.

Let $\fk{C}^{\tau} \sub \op{SL}(2, \C)$ be the conjugacy class of elements having trace $\tau.$  By \cite[Sec.\ 2.1(d)]{abou-man}, $\fk{C}^{\tau}$ is diffeomorphic to $TS^2.$ Let $R_{irr}(\tU_i) \sub R(\tU_i)$ denote the open subvariety of irreducible representations of $\pi_1(\tU_i, x_0)$. Let $X_{irr}(\tU_i) \sub X_{irr}(\tU_i)$ be defined analogously. 

Observe that $\pi_1(\tU_i, x_0)=F_g,$ where $F_g$ is the free group on $g$ generators. By \cite[Lem.\ 2.6]{abou-man}, we have that $\pi_1(X_{irr}(\tU_i))= 0$ while $\pi_2(X_{irr}(\tU_i))=\Z/2.$  It follows that $\pi_1(X_{irr}(\tU_i) \tms \fk{C}^{\tau})=0$ and $\pi_2(X_{irr}(\tU_i) \tms \fk{C}^{\tau})= \pi_2(X_{irr}(\tU_i)) \oplus \pi_2(\fk{C}^{\tau})= \Z/2 \oplus \Z.$  

Let us now consider the space $(R_{irr}(\tU_i) \tms \fk{C}^{\tau} ) /G^{ad},$ where $G^{ad} = \op{SL}(2,\C)/ Z(\op{SL}(2, \C))= \op{PSL}(2, \C).$  

\lem 
\label{lem:Qspin}
The space  $Q:= (R_{irr}(\tU_i) \tms \fk{C}^{\tau} ) /G^{ad}$ admits a spin structure. \elem

\pf Since $G^{ad}$ acts freely on $R_{irr}(\tU_i),$ we see that $Q$ is a $\fk{C}^{\tau}$-bundle over
\eq X_{irr}(\tU_i)=R_{irr}(\tU_i)/ G^{ad}. \eeq

Since $\pi_1(X_{irr}(\tU_i))=0$ and $\pi_2(X_{irr}(\tU_i))= \Z/2$, it follows by the Hurewicz and universal coefficients theorem that $H^1(X_{irr}(\tU_i), \Z)=H^2(X_{irr}(\tU_i), \Z)=0$. Let $i: \fk{C}^{\tau} \hookrightarrow Q$ be the inclusion of a fiber.  The Leray-Serre spectral sequence implies that there is an isomorphism 
\eq \label{equation:fiberinclusion}  i^*: H^2(Q; \Z) \xrightarrow{\cong} H^2(\fk{C}^{\tau}; \Z). \eeq
By naturality of Chern classes, we also have $$i^* c_1(Q)= c_1( \fk{C}^{\tau}).$$

We now argue that $w_2(\fk{C}^{\tau})=0$. Indeed, recall that $\fk{C}^{\tau} \cong TS^2$. We have $T(TS^2)= \pi^*(TS^2) \oplus \pi^*(TS^2)$. Since $S^2$ admits a spin structure and $H^1(S^2, \Z/2)=0,$ it follows from the Whitney product formula that $0= w_1(TS^2)=w_2(TS^2)$.

Next, we consider the short exact sequence of coefficients $$0 \to \Z \to \Z \to \Z/2 \to 0,$$ which induces a long exact sequence in cohomology
\eq \dots\to H^2(\fk{C}^{\tau}; \Z) \to H^2(\fk{C}^{\tau}; \Z) \to H^2(\fk{C}^{\tau}; \Z) \to H^3(\fk{C}^{\tau}; \Z)=0.\eeq
Since $w_2(\fk{C}^{\tau})$ is the mod $2$ reduction of $c_1(\fk{C}^{\tau})$, it follows that $c_1(\fk{C}^{\tau})$ is divisible by $2$ in $H^2(\fk{C}^{\tau}; \Z)$. Hence $c_1(Q)$ is also divisible by $2$ in $H^2(Q; \Z)$ by \eqref{equation:fiberinclusion}. Hence $w_2(Q)=0$. 
Since $Q$ is a complex manifold, it is orientable. It follows from the vanishing of $w_2(Q)$ that $Q$ admits a spin structure. (Indeed, an orientable manifold admits a spin structure if and only if the classifying map to $\op{BSO}(n)$ can be lifted to $\op{BSpin}(n)$. The second Stiefel-Whitney class is precisely the obstruction to such a lift; see \cite[Thm.\ 4.38]{cohen}.)
 \epf

\pf[Proof of \Cref{proposition: spinstructures}] Observe that there is a natural inclusion $R_{irr}(\tU_i) \tms \fk{C}^{\tau} \to R^{\tau}_{irr}(U_i)$. Taking quotients by the $G^{ad}$ action gives the inclusion $Q=(R_{irr}(\tU_i) \tms \fk{C}^{\tau} ) /G^{ad} \hookrightarrow X_{irr}^{\tau}({U}_i).$  The residual set $X_{irr}^{\tau}({U}_i)  - Q$ consists of irreducible representations and has dimension at most $ \op{dim} R_{red}(\tU_i)+ \op{dim} \fk{C}^{\tau}$, where $R_{red}(\tU_i):= R(\tU_i)- R_{irr}(\tU_i).$ But $R_{red}(\tU_i) $ has complex dimension $2g+1$ \cite[proof of Lem.\ 2.6]{abou-man}.  Hence $X_{irr}^{\tau}({U}_i)  - Q$ has complex dimension at most $ 2g+3.$ 

But $X_{irr}^{\tau}({U}_i)$ has complex dimension $3g-1.$ Hence we find that $X_{irr}^{\tau}({U}_i)  - Q$ has complex codimension $3g-1 - (2g+3)= g -4.$ 

Recall that the second Stiefel-Whitney class of an $n$-dimensional manifold is the obstruction to finding $n-1$ linearly independent sections of the tangent bundle on the 2-skeleton; see \cite[p.\ 140]{milnorstasheff}. This obstruction vanishes if it vanishes in the complement of a subset of complex codimension $\geq 2.$ We conclude that $w_2(X_{irr}^{\tau}(U_i))=0$ provided that $g \geq 6$.  It follows that $X_{irr}^{\tau}(U_i)=L_i^{\tau}$ admits a spin structure. 

It remains to verify that this spin structure is unique. To this end, recall from the proof of Lemma~\ref{lem:Qspin} that $Q$ is a $TS^2$-bundle over a simply connected space, hence it is simply connected itself. Assuming still that $g\geq 6,$ it follows from an analogous dimension count that $X_{irr}^{\tau}(U_i)$ is also simply connected. Because spin structures form an affine space over the first cohomology with $\Z/2$ coefficients, we conclude that the spin structure of $X_{irr}^{\tau}(U_i)=L_i^{\tau}$ is unique.\epf

\section{Computational tools} 
\label{sec:tools}
This section is intended to collect some auxiliary results which will be useful in computing $P^{\bullet}_{\tau}(K)$ for various examples later on. Fix a doubly-pointed knot $(Y,K, p,q)$ and a parameter $\tau \in (-2,2)$. Let $E_K$ be the blowup and fix a Heegaard splitting $\mc{H}= (\S, U_0, U_1)$ for $E_K$. 

We will consider two situations. First, we will show that $P^{\bullet}_{\tau}(K)= P^{\bullet}_{L_0, L_1}$ is a local system if $\mathscr{X}^{\tau}_{irr}(E_K)$ is a smooth scheme. This fact relies crucially on our description of the tangent space to relative representation schemes in \Cref{proposition:tangentspace}.  Next, we will describe $P^{\bullet}_{\tau}(K)$ in the case where $\mathscr{X}^{\tau}_{irr}(E_K)= \op{Spec} R,$ for $R = \frac{\C[\e_1]}{\e^{n_1}} \tms\dots\tms \frac{\C[\e_k]}{\e_k^{n_k}}$.  Such a scheme of course fails to be smooth unless all $n_i=1$. 

\subsection{The smooth case} The following proposition is an analog of \cite[Lem.\ 3.4]{abou-man} for parabolic group cohomology. 

\prop \label{proposition:mayer-vietoris} Suppose that $\rho: \pi_1(E_K, x_0) \to \op{SL}(2,\C)$ is an irreducible representation.  Then there is an exact sequence
\begin{align*} 0 \to H^1_{par}( \pi_1(E_K,x_0); \op{Ad} \rho) \to& H^1_{par}( \pi_1({U}_0, x_0); \op{Ad} \rho) \oplus H^1_{par}( \pi_1({U}_1,x_0); \op{Ad} \rho)\\
 &\to  H^1_{par}( \pi_1({\S},x_0); \op{Ad} \rho). 
\end{align*}
\eprop

Note that we have dropped the conjugacy classes from our notation for the parabolic cohomology groups; cf. \Cref{remark:parabolicnotation}.

\pf  We refer the reader to the diagram \eqref{equation: diagram3}. Since $\rho$ is irreducible, it follows that $H^0_{A_{\rho}}({\S}; \fk{g})=0$. It can then be verified by a straightforward diagram chase that there is an exact sequence
\eq 0 \to \op{ker} a \to \op{ker} b_0 \oplus \op{ker} b_1 \to \op{ker} c. \eeq

The claim now follows from \Cref{proposition:firstgeometric}. \epf

\lem \label{lemma:cleanintersection} The Lagrangian submanifolds $L_0 \sub X^{\tau}_{irr}({\S})$ and $L_1 \sub X^{\tau}_{irr}({\S})$ intersect cleanly at $\rho$ if and only if $\rho$ is a regular point of $\mathscr{X}^{\tau}(E_K).$ \elem

\pf By combining \Cref{proposition:mayer-vietoris} with \Cref{proposition:tangentspace}, we can apply exactly the same argument as in the proof of Lemma 3.4 in \cite{abou-man}.  \epf

The following proposition is our main tool for computing $P_{L_0,L_1}^{\bullet}$. 

\prop  \label{proposition:localsystem} If $\mathscr{X}_{irr}^{\tau}(E_K)$ is a smooth scheme, then $P_{L_0,L_1}^{\bullet}$ is a local system on ${X}_{irr}^{\tau}(E_K)$. The stalks of $P_{L_0,L_1}^{\bullet}$ on each component of ${X}_{irr}^{\tau}(E_K)$ of complex dimension $k$ are isomorphic to $\bb{Z}$, supported in degree $-k$. \eprop

\pf If $L_0$ and $L_1$ intersect cleanly, then it was shown in \cite[Prop 6.2]{abou-man} that  $P_{L_0,L_1}^{\bullet}$ is a local system with stalks isomorphic to $\bb{Z}$ in degree $-k$ on $k$-dimensional components. But it follows from \Cref{lemma:cleanintersection} that $L_0$ and $L_1$ intersect cleanly if and only if $\mathscr{X}_{irr}^{\tau}(E_K)$ is smooth as a scheme. \epf

\subsection{A non-reduced setting} 

Another situation where we are able to compute  $P_{L_0,L_1}^{\bullet}$ is the following. 

\prop \label{proposition:nonreducedcomputation} Suppose that $\mathscr{X}_{irr}^{\tau}(E_K)$ consists of $k$ points of multiplicities $n_1, \dots, n_k$; i.e., it equals $\op{Spec} R$, where 
$$R = \frac{\C[\e_1]}{\e^{n_1}} \tms\dots\tms \frac{\C[\e_k]}{\e_k^{n_k}}.$$ Then  $P_{L_0,L_1}^{\bullet}$ is a sheaf concentrated in degree $0$, with stalk $\Z^{n_i}$ over the $i^{\text{th}}$ point.  \eprop

\Cref{proposition:nonreducedcomputation} is a consequence of the following general lemma, which appeals to theory of d-critical loci introduced by Joyce in \cite{joyce}. A (complex-analytic) d-critical locus is a complex-analytic space $X$ along with a section $s$ of a certain sheaf $\mc{S}_X^0$ which locally parametrizes different ways of writing $X$ as the critical locus of a holomorphic function; see \cite[Def.\ 2.5]{joyce}. 

\begin{lemma}  \label{lemma:nonreducedcomputation} Consider two complex spin Lagrangians $L_0, L_1$ in a complex symplectic manifold $M$. Let $X=\{x\} \in L_0 \cap L_1$ be an isolated intersection point, such that, as a complex analytic space, $X$ is nonreduced of order $n \geq 2$. Then the stalk of the perverse sheaf $P^{\bullet}_{L_0, L_1}$ over $x$ is $\Z^{n}$ in degree zero.
\end{lemma}

\begin{proof}
Let $\O= \O_{\C}$ be the space of holomorphic functions in one variable $z$. By hypothesis, the analytic space $X$ under consideration is isomorphic to $\O/(z^n).$ 

We know that locally, the Lagrangian $L_1$ is the graph of $df$, for some holomorphic function $f: U \to \C$, where $U$ is a neighborhood of $x$ in $L_0$. This turns $X$ into a $d$-critical locus in the sense of \cite[Def.\ 2.5]{joyce}. In our case, the sheaf $\mc{S}^0_X$ is computed in \cite[Example 2.16]{joyce} in the algebraic setting, but the same calculation applies in the complex analytic setting: the sections $s$ of $\mc{S}^0_X$ are elements of the module $(z^{n+1})/(z^{2n})$ over $\O/(z^n)$:
$$ s = a_{n+1} z^{n+1} + a_{n+2}z^{n+2} + \dots + a_{2n-1} z^{2n-1} + (z^{2n}).$$

Such a section $s$ specifies the structure of $X$ as a $d$-critical locus (provided that $a_{n+1} \neq 0$). Once this structure is determined, by Proposition 2.22 in \cite{joyce}, we find that, if $m$ denotes the complex dimension of $L_0$, then there are local holomorphic coordinates $z_1, \dots, z_m$ near $x \in L_0$ in which we can write
$$ f(z_1, \dots, z_m)= s(z_1) + z_2^2 + \dots + z_m^2.$$

The perverse sheaf $P^{\bullet}_{L_0, L_1}|_X$ is $\Z^{\mu-1}$ in degree $0$, where $\mu$ is the Milnor number of $f$. Regardless of the values of $a_{n+1} \neq 0, a_{n+2}, \dots, a_{2n-1}$, the Milnor number is $n+1$, and the claim follows.
\end{proof}

\section{Relation to the $\widehat{A}$-polynomial}  \label{section:constantdim1} In this section, we will relate $P_{\tau}^{\bullet}(K)$ to a knot invariant known as the $\widehat{A}$-polynomial, which is a close cousin of the better-known $A$-polynomial. For knots in integral homology $3$-spheres whose $\op{SL}(2,\C)$-character variety is $1$-dimensional, we will show under certain technical assumptions that $\HP^*_{\tau}(K)$ simply recovers the $\l$-degree of the $\widehat{A}$-polynomial. This relationship breaks down for general knots, as we will see in \Cref{section:connectedsums} when considering certain connected sums of knots. The $\widehat{A}$-polynomial is computable in many situations. In \Cref{section:1dcomputations}, we will leverage some of these computations to determine $\HP^*_{\tau}(K)$ for various classes of prime knots. 

\subsection{Preparations} \label{subsection:preparationsmall} Let us consider a doubly-pointed knot $(Y, K, p, q) \in \bf{Knot}_{**}$ and its blowup $(E_K, \g) \in \bf{SutKnot}$. From now on, we will always assume that the manifold $Y$ is an integral homology $3$-sphere. This assumption will allow us to talk about the Alexander polynomial (in Assumption~\ref{assumption:alexander} below) and the $A$- and $\widehat{A}$-polynomials (in Section~\ref{sec:Apol}).

We fix a basepoint $x_0 \in c_p$ and let \eq \label{equation:generatingset} \G= \pi_1(E_K, x_0)=  \langle g_1, \dots, g_n \mid r_1, \dots, r_l \rangle \eeq 

We fix $\tau \in (-2,2)$ and let $\fk{c}= \op{Conj}([c_p])$. We will also assume that the presentation of $\G$ which we have chosen has the property that $\op{Conj}(g_1)= \fk{c}$. 

We will rely throughout this section on the definitions and notation introduced in \Cref{subsection:relrepgeneralities}. In particular, we remind the reader that $\mc{A}(\G)$ is the coordinate ring of the representation scheme $\mathscr{R}(\G)$. Similarly, $\mc{A}^{\tau}(\G)$ is the coordinate ring of the relative representation scheme $\mathscr{R}^{\tau}(\G)= \mathscr{R}^{\tau}(\G; \fk{c})$. These coordinate rings are well-defined up to unique isomorphism. 

As described in \Cref{example: relrepexample}, one can construct a model for $\mc{A}(\G)$ by associating to each generator $g_i \in \G$ the formal variables $x^{g_i}_{11}, x^{g_i}_{12}, x^{g_i}_{21}, x^{g_i}_{22}$ and modding out the free $k$-algebra $k[x^{g_1}_{11}, x^{g_1}_{12}, \dots, x^{g_n}_{22}]$ by the appropriate relations. One can then construct $\mc{A}^{\tau}(\G)$  by setting 
\eq \label{equation: quotientrecall} \mc{A}^{\tau}(\G)= \mc{A}(\G)/ (x^{g_1}_{11}+ x^{g_1}_{22}- \tau). \eeq

It will be convenient to write $\mc{A}_0(\G):=\mc{A}(\G) / \op{nil}(\mc{A}(\G))$, where $\op{nil}(\mc{A}(\G))$ is the nilradical of $\mc{A}(\G)$. Since $\op{SL}_2$ is linearly reductive, we have 
\eq \label{equation:a0} \mc{A}(\G)^{\op{SL}_2} / \op{nil}(\mc{A}(\G)^{\op{SL}_2}) = \mc{A}(\G)^{\op{SL}_2} / \op{nil}(\mc{A}(\G))^{\op{SL}_2} =  (\mc{A}(\G) / \op{nil}(\mc{A}(\G)) )^{\op{SL}_2} = \mc{A}_0(\G)^{\op{SL}_2}. \eeq

Hence the character variety $X(\G)$ is just the spectrum of $\mc{A}_0(\G)^{\op{SL}_2}$. For $g \in \G$, one can define an element $\tau_g \in \mc{A}_0(\G)^{\op{SL}_2}$ called the \emph{character} of $g$. For a closed point $[\rho] \in X(\G)$, we have $\tau_g(\rho)=\op{Tr}(\rho(g))$. Using the well-known identity $$\tau_g \tau_h= \tau_{gh}+\tau_{g h^{-1}},$$ it can be shown that the $N=2^n-1$ functions $$\tau_{g_{i_1}\dots g_{i_k}}, \; 1\leq k \leq n, \; 1\leq i_1<\dots<i_k \leq n $$ in fact generate $\mc{A}_0(\G)^{\op{SL}_2}$. 

In particular, there is a surjective ring map
\eq \label{equation:charembedding} \phi: k[x_1, \dots, x_N] \to \mc{A}_0(\G)^{\op{SL}_2} \eeq where we can assume that $x_1 \mapsto \tau_{g_1}= x^{g_1}_{11}+ x^{g_1}_{22}.$  Letting $\mc{I}= \op{ker} \phi$, we can use \eqref{equation:charembedding} to identify  \eq k[x_1, \dots, x_N] / \mc{I} =  \mc{A}_0(\G)^{\op{SL}_2}. \eeq It will be convenient later on to write \eq \label{equation:abar} \ov{A}:= k[x_1, \dots, x_N] / \mc{I} =  \mc{A}_0(\G)^{\op{SL}_2}. \eeq 

Finally, we let $$H_{\tau} \hookrightarrow \bb{A}^N$$ be the closed embedding corresponding to the ring map $k[x_1, \dots, x_N] \to k[x_1, \dots, x_N]/(x_1-\tau).$ We let $\pi_n: \bb{A}^N \to \bb{A}$ be the projection map $(x_1,\dots,x_N) \mapsto x_n$.

\subsection{Some assumptions on the character variety} \label{subsection:assumptions} Throughout this section, we limit our attention to knots whose $\op{SL}(2,\C)$-character variety is $1$-dimensional. We usually also need to assume that the character scheme is reduced. We wish to keep track of how the statements proved in this section depend on these assumptions. It will therefore be convenient to label them separately. 

\begin{customassumption}{A.1} \label{assumption:dim1} The character scheme $\mathscr{X}(\G)$ is of dimension at most $1$. \end{customassumption}
\begin{customassumption}{A.2} \label{assumption:reduced} The character scheme $\mathscr{X}(\G)$ is reduced. \end{customassumption} 

We will see in \Cref{section:1dcomputations} that \ref{assumption:dim1} is satisfied for many familiar examples of prime knots, including all two-bridge knots, torus knots and certain hyperbolic pretzel knots. However, it fails for simple examples of composite knots, as we will see in \Cref{section:connectedsums}. \ref{assumption:reduced} has also been verified for two-bridge knots, torus knots and many pretzel knots. Moreover, there are no examples of knots for which \ref{assumption:reduced} is known to fail and it has been conjectured by Le and Tran (see \cite[Conjecture 2]{le-tran}) that \ref{assumption:reduced} is true for all knots.

It will also be useful to consider some assumptions on a complex parameter $\tau \in \C$. Since $P^{\bullet}_{\tau}(K)$ is only defined for $\tau \in (-2,2)$, we will sometimes need to restrict $\tau$ to the real interval $(-2,2) \sub \C$. However, it is in general more convenient to allow $\tau$ to take arbitrary complex values in this section.

\begin{customassumption}{B.1} \label{assumption:components} \label{assumption:first} The analytic space ${H}_{\tau}=\{x_1=\tau\}$ does not contain any component of $\pi^{-1}(-2,2) \cap X(\G)$. \end{customassumption}
\begin{customassumption}{B.2} \label{assumption:smooth} There is an analytic neighborhood of ${H}_{\tau} \cap X(\G) \sub X(\G)$ on which $X(\G)$ is smooth. \end{customassumption}
\begin{customassumption}{B.3} \label{assumption:alexander} If $x \in [0,1]$ has the property that $e^{4 \pi i x}$ is a root of the Alexander polynomial of $K$, then $\tau \neq 2 \cos (2 \pi x)$. \end{customassumption}
\begin{customassumption}{B.4} \label{assumption:proper} \label{assumption:last} There is an analytic neighborhood of $H_{\tau} \cap X(\G) \sub X(\G)$ on which the projection $\pi_1: X(\G) \to \C$ is proper.  \end{customassumption}

In practice, these assumptions may be difficult to verify for a given knot $K \sub Y$ without explicit knowledge of its character variety. However, the following proposition shows that they are satisfied on a cofinite subset of $\C$ whenever \ref{assumption:dim1} holds.

\prop \label{proposition:generic1} Suppose that the character variety of $K \sub Y$ satisfies \ref{assumption:dim1}. Then \ref{assumption:first}--\ref{assumption:last} are satisfied for all but finitely many choices of $\tau \in \C$. \eprop

\pf Since $\G$ is a finitely-generated group, $\mc{A}(\G)$ is a Noetherian ring. This means that $X(\G)$ is a \emph{Noetherian scheme}. It is a general fact that Noetherian schemes have at most finitely many irreducible components. Since $H_{\tau} \cap H_{\tau'}= \emptyset$ if $\tau \neq \tau'$, it follows that \ref{assumption:components} holds for all but finitely many $\tau \in (-2,2)$. 

The singular locus of a complex algebraic variety is Zariski closed and of complex codimension at least $1$. Since $X(\G)$ is $1$-dimensional by \ref{assumption:dim1}, it follows that $X(\G)$ has a $0$-dimensional set of singularities which is therefore finite. Hence \ref{assumption:smooth} holds on a cofinite subset of $\C$. 

It's clear that \ref{assumption:alexander} holds in a cofinite subset of $\C$. Finally, the fact that \ref{assumption:proper} holds for all but finitely many $\tau \in \C$ is the content of \Cref{lemma:post}, whose proof is posponed to \Cref{subsection:localbehavior}. \epf

\subsection{Relating the fiber product to the quotient} In \eqref{equation:charembedding}, we considered an embedding of $X(\G)$ into $\C^N$. The coordinate functions on $\C^N$ pull back under this embedding to the characters of certain products of generators of $\G$, and we assumed in particular that $x_1$ pulls back to $x^{g_1}_{11}+x^{g_1}_{22}$. This suggests that the relative character scheme $\mathscr{X}^{\tau}(\G)$ should be obtained by intersecting the image of the ordinary character variety $X(\G)$ with the hyperplane $H_{\tau}=\{x_1=\tau\}$. The purpose of this section is to make this precise. 

Recall from \eqref{equation:abar} that we have $$\ov{A}:= k[x_1, \dots, x_N] / \mc{I} =  \mc{A}_0(\G)^{\op{SL}_2}.$$ We let $f_{\tau} = (x_1- \tau) \in k[x_1,\dots,x_N]$. By abuse of notation, we will also view $f_{\tau}$ as an element of $\mc{A}_0(\G)^{\op{SL}_2} \sub \mc{A}_0(\G)$ via \eqref{equation:charembedding}, as well as an element of $\ov{A}$ via the quotient projection. 

Let us consider the quotient map
\eqs \mc{A}_0(\G) \to \mc{A}_0(\G)/(f_{\tau} \mc{A}_0(\G)). \eeqs
The $\op{SL}_2$ action on $\mc{A}_0(\G)$ induces a surjective map on the ring of invariants 
\eqs \label{equation: secondlastmap} \ov{A} \to (\mc{A}_0(\G)/(f_{\tau} \mc{A}_0(\G)))^{\op{SL}_2} = \mc{A}_0(\G)^{\op{SL}_2} / (f_{\tau} \mc{A}_0(\G))^{\op{SL}_2} . \eeqs
Observe that the ideal $(f_{\tau}) \sub \ov{A}$ is in the kernel of this map. Thus we obtain a surjection 
\eq \label{equation: lastmap} \ov{A}/(f_{\tau}) \to \mc{A}_0(\G)^{\op{SL}_2} / (f_{\tau} \mc{A}_0(\G))^{\op{SL}_2}. \eeq

\lem \label{lemma:a0a} Suppose that $K \sub Y$ satisfies \ref{assumption:reduced}. It then follows that \eq \mc{A}_0(\G)^{\op{SL}_2}/(f_{\tau} \mc{A}_0(\G))^{\op{SL}_2}= \mc{A}(\G)^{\op{SL}_2} / (f_{\tau} \mc{A}(\G))^{\op{SL}_2}= \lt( \mc{A}(\G)/ (f_{\tau} \mc{A}(\G)) \rt)^{\op{SL}_2}. \eeq \elem 
\pf According to \ref{assumption:reduced}, we have $\op{nil}(\mc{A}(\G)^{\op{SL}_2})= 0$. It then follows from \eqref{equation:a0} that $\mc{A}(\G)^{\op{SL}_2} = \mc{A}_0(\G)^{\op{SL}_2}$. Note next that \eqs (f_{\tau} \mc{A}_0(\G))^{\op{SL}_2}= (f_{\tau} \mc{A}(\G))^{\op{SL}_2} / (f_{\tau} \mc{A}(\G) \cap \op{nil} \mc{A}(\G))^{\op{SL}_2}. \eeqs Since $(f_{\tau} \mc{A}(\G) \cap \op{nil}(\mc{A}))^{\op{SL}_2} \sub \op{nil}(\mc{A}(\G)^{\op{SL}_2}=\emptyset$, we conclude that $(f_{\tau} \mc{A}_0(\G))^{\op{SL}_2} = (f_{\tau} \mc{A}(\G))^{\op{SL}_2}$. The conclusion follows. \epf

Observe that $\lt( \mc{A}(\G)/ (f_{\tau} \mc{A}(\G)) \rt)^{\op{SL}_2}$ is the coordinate ring of $\mathscr{X}^{\tau}(\G)$. Under the hypotheses of \Cref{lemma:a0a}, it follows that the map \eqref{equation: lastmap} induces a closed embedding of schemes \eq \label{equation:schemeembedding} \mathscr{X}^{\tau}(\G) \to ({H}_{\tau} \cap X(\G)). \eeq
It is easy to see that the underlying embedding of varieties is an isomorphism. We would like to show that this is also true for schemes.

\prop Suppose that $K \sub Y$ satisfied \ref{assumption:dim1}--\ref{assumption:reduced} and that $\tau \in \C$ satisfies \ref{assumption:components}. Then the closed embedding \eqref{equation:schemeembedding} is an isomorphism. \eprop

\pf  Let $R$ be the coordinate ring of the group scheme $\op{SL}_2$. Let $\mu: \mc{A}_0(\G) \to \mc{A}_0(\G) \otimes R$ be the map of $\C$-algebras which induces the $\op{SL}_2$-action on $\op{Spec} \mc{A}_0(\G)= R(\G)$. Note that $\mu(f) = f \otimes 1$ for $f \in \mc{A}_0(\G)^{\SL}$. 

It is equivalent for \eqref{equation:schemeembedding} and for \eqref{equation: lastmap} to be isomorphisms.  Since we know that \eqref{equation: lastmap} is surjective, we will show that it is injective. 

To this end, we need to show that $(f_{\tau} \mc{A}_0(\G))^{\op{SL}_2} \sub f_{\tau} \mc{A}_0(\G)^{\op{SL}_2}$.  So suppose that $(f_{\tau} g) \in (f_{\tau} \mc{A}_0(\G))^{\op{SL}_2}$ while $g \notin \mc{A}_0(\G)^{\op{SL}_2}$. This implies that
\begin{align} \label{equation:reduced}
0= \mu(f_{\tau} g)- \mu(f_{\tau})\mu(g) \nonumber 
= (f_{\tau} g) \otimes 1 - (f_{\tau} \otimes 1)\mu(g) \nonumber 
&= (f_{\tau} \otimes 1)(g \otimes 1) - (f_{\tau} \otimes 1) \mu(g) \nonumber \\
&=(f_{\tau} \otimes 1)(g \otimes 1- \mu(g)).
\end{align}

Given a closed point $x \in \op{SpecMax}R,$ there is an evaluation map $\op{ev}_x: \mc{A}_0(\G) \otimes R \to \mc{A}_0(\G)$.  Since $\mc{A}_0(\G)$ and $R$ are reduced $\C$-algebras (i.e. the associated schemes are reduced), it can be shown that $g\otimes 1$ and $\mu(g)$ are equal if and only if they are equal at all closed points.\footnote[1]{At the cost of slightly complicating the argument, we could avoid appealing to the fact that $\mc{A}_0(\G)$ is reduced; see \cite[Sec.\ 3, Rmk.\ 2.4]{gulbrandsen}.} Since $g \not \in \mc{A}_0(\G)^{\SL}$, we see that there exists $x_0 \in \op{SpecMax} R$ such that $\op{ev}_{x_0} (g \otimes 1) \neq \op{ev}_{x_0}(\mu(g))$. 

Let us now apply $\op{ev}_{x_0}$ to \eqref{equation:reduced} to obtain  \eq 0 = \op{ev}_{x_0} ( f_{\tau} \otimes 1) \op{ev}_{x_0}(g \otimes 1- \mu(g)) = f_{\tau} \op{ev}_{x_0}(g \otimes 1- \mu(g)).\eeq

We conclude that $f_{\tau} \in \mc{A}_0(\G)$ is a zero divisor. But since $\mc{A}_0(\G)$ is reduced, it follows that $f_{\tau}$ vanishes on an irreducible component $C \sub \op{Spec} \mc{A}_0(\G) = R(\G)$ (see \cite[1.4.8(2)]{manin}). Let $\pi: R(\G) \to X(\G)$ be the natural projection. Observe that $\pi(C) \subset H_{\tau}$. We wish to show that $\ov{\pi(C)}$ is a $1$-dimensional component of $X(\G)$. 

Let's first show that $\ov{\pi(C)}$ is not zero-dimensional. If it were, then it would consist of a single point $[\rho]$. Since we can always find an abelian representation with trace $\tau$, it follows that $[\rho]$ must be abelian. But one can then directly verify that $\pi^{-1}(\pi(C))= \pi^{-1}([\rho])$ does not contain any irreducible components of $X(\G)$. Indeed, if $\rho' \in \pi^{-1}(\pi(C))$, then $\rho'$ is abelian and hence can be deformed through abelian representations. Hence $\ov{\pi(C)}$ is at least one-dimensional.  Therefore it is exactly one-dimensional by \ref{assumption:dim1}.  This is a contradiction in view of \ref{assumption:components}. \epf

\defi Let ${X}_0(\G) \sub X(\G)$ be the open subscheme obtained from $X(\G)$ by removing all components which do not contain an irreducible representation. \edefi

\prop \label{proposition:intersectionhyperplane} Suppose that $K \subset Y$ satisfies \ref{assumption:dim1}--\ref{assumption:reduced} and $\tau \in \C$ satisfies \ref{assumption:components} and \ref{assumption:alexander}. Then there is an isomorphism $\mathscr{X}^{\tau}_{irr}(\G) \to {H}_{\tau} \cap {X}_0(\G)$. \eprop 

\pf Since $\mathscr{X}^{\tau}_{irr}(\G) \hookrightarrow \mathscr{X}^{\tau}(\G)$ is an open subscheme, it follows from the previous proposition that there is an open embedding $\mathscr{X}^{\tau}_{irr}(\G) \hookrightarrow \mathscr{X}^{\tau}(\G) \to {H}_{\tau} \cap X(\G)$. Observe that the image of this map lands inside ${H}_{\tau} \cap {X}_0(\G)$. It suffices to verify that $\mathscr{X}^{\tau}_{irr}(\G)  \to {X}_0(\G)$ is surjective on closed points to obtain the desired claim.

If $\rho$ is irreducible, then $[\rho] \in {H}_{\tau} \cap {X}_0(\G)$ is in the image of the above map.  If $\rho$ is reducible, then note by the definition of $X_0(\G)$ that $[\rho] \in X(\G)$ cannot belong to the component of abelian representations. It follows from \cite[Prop.\ 6.2]{planecurves} that $\tau = 2 \cos (2 \pi x)$ where $e^{4 \pi i x}$ is a root of the Alexander polynomial of $K$. This contradicts \ref{assumption:alexander}. \epf

We now analyze the structure of $H_{\tau} \cap X_0(\G)$. 

\prop \label{proposition: subfinal} Suppose that $K \sub Y$ satisfies \ref{assumption:dim1}--\ref{assumption:reduced} and that $\tau \in \C$ satisfies \ref{assumption:components}--\ref{assumption:smooth}. Then we can write ${H}_{\tau} \cap {X}_0(\G) = \op{Spec} R_{\tau}$, where $$R_{\tau}=  \frac{\C [\e_1]}{\e_1^{n_1}} \tms \dots \tms \frac{\C [\e_k]}{\e_k^{n_k}}$$ for some values $n_1, \dots, n_k \geq 1$. \eprop 

\pf It follows from \ref{assumption:dim1} and \ref{assumption:components} that ${H}_{\tau} \cap X(\G)$ and hence ${H}_{\tau} \cap {X}_0(\G)$ is zero-dimensional.  Using \ref{assumption:dim1} and \ref{assumption:smooth}, we find that ${H}_{\tau} \cap {X}_0(\G)$ embeds into a smooth finite-type $\C$-scheme of dimension $1.$ The claim then results from \Cref{lemma:zerodim} below. \epf

\lem \label{lemma:zerodim} Let $Z \sub S$ be an embedding of a zero-dimensional $\C$-scheme of finite type into a smooth, one-dimensional $\C$-scheme of finite type. Then $Z= \op{Spec}B$ for $B=  \frac{\C[\e_1]}{\e_1^{n_1}} \tms \dots \tms \frac{\C[\e_k]}{\e_k^{n_k}}.$ \elem

\pf We can easily reduce to the case where $Z= \op{Spec} B$ is connected and hence $l=1.$ Then the closed embedding $Z \subset S$ is induced by a surjective ring map $A \to B.$ Since $\op{Spec}B$ is connected and of dimension $0,$ it follows that $B$ is a local ring with maximal ideal $\frak{m}.$ Let $\frak{p}$ be the kernel of the composition $A \to B \to B/\frak{m}.$ Then we have a surjection $A_{\frak{p}} \to B_{\frak{m}}= B.$ 

Note that $B$ is a Noetherian ring of dimension $0.$ It follows that $B$ is an Artinian ring, and hence complete. Hence we have a surjection $\widehat{A}_{\frak{p}} \to \widehat{B} = B.$ But $A_{\frak{p}}$ is a discrete valuation ring since $S$ is smooth. Hence $\widehat{A}_{\frak{p}} \simeq \C[[t]].$ Since the ideals of $\C[[t]]$ are all of the form $(t^n)$ for $n \geq 0$, we conclude that $B \simeq \C[[t]](t^n)= \C[t]/(t^n).$ \epf

\cor \label{corollary:locallyinvariant} Suppose that $K \sub Y$ satisfies \ref{assumption:dim1}--\ref{assumption:reduced} and that $\tau \in  (-2,2)$ satisfies  \ref{assumption:components}, \ref{assumption:smooth} and \ref{assumption:alexander}. Let $R_{\tau}$ be as in \Cref{proposition: subfinal}. Then $P^{\bullet}_{\tau}(K)$ is supported in degree $0$ and we have $\HP^0_{\tau}(K)= \Z^d$, where $d=\sum_1^l n_i$.  \ecor

\pf Apply Propositions~\ref{proposition:intersectionhyperplane} and \ref{proposition: subfinal}. If $n_i=1$ for all $i$, then $\op{Spec} R_{\tau}$ is a smooth scheme of dimension $0$. It follows from \Cref{proposition:localsystem} that $P^{\bullet}_{\tau}(K)$ consists of a copy of $\mathbb{Z}$ supported in degree $0$ over each point. 

If $n_j>1$ for some $j$, then $\op{Spec} R_{\tau}$ is not smooth. However, we can apply \Cref{proposition:nonreducedcomputation}, which implies that $P^{\bullet}_{\tau}(K)$ is still supported in degree $0$, with stalks isomorphic to $\Z^{n_j}$ over each point of multiplicity $n_j$. \epf

\subsection{Independence of parameters} \label{subsection:localbehavior} We now move to analyzing the dependence of $\HP^*_{\tau}(K)$ on the parameter $\tau$. If $K \sub Y$ satisfies \ref{assumption:dim1}--\ref{assumption:reduced}, we will show that $\HP^*_{\tau}(K)$ is constant over all values of $\tau$ which satisfy \ref{assumption:first}--\ref{assumption:last}. 

Recall that in \eqref{equation:generatingset} we chose a set of generators for $\G$. This gives embeddings of $X_0(\G)$ and $X(\G)$ in some affine space $\bb{C}^n$. Let $\ov{X_0}(\G) \sub \bb{CP}^n$ and $\ov{X}(\G) \sub \bb{CP}^n$ denote the projective closures of $X_0(\G)$ and $X(\G)$ respectively. We endow $\ov{X_0}(\G)$ and $\ov{X}(\G)$ with the unique reduced scheme structure, which ensures that they are well-defined as (possibly singular) schemes.  We also let $D_{\tau} \sub \bb{CP}^n$ be the projective closure of $H_{\tau}$, i.e. $D_{\tau} = \bb{CP}^{n-1} \sub \bb{CP}^n$ is a hyperplane.

If $X \sub \bb{CP}^n$ is a projective scheme of dimension $n$, we let $p_X(m) \in \Z[m]$ be the Hilbert polynomial of $X.$ Recall that the \emph{degree} of $X$ is the leading coefficient of $p_X$ times $n!$.

We then have the following version of B\'{e}zout's theorem. 

\thm[18.6.K \cite{vakil}] \label{theorem: bezout} Let $X \sub \bb{CP}^n$ be a reduced projective subscheme of dimension at least $1$ and let $D$ be a hypersurface. Suppose that $D$ does not contain any components of $X.$ Then $\op{deg}(D \cap X)= \op{deg}(D) \op{deg}(X) .$ \ethm

\cor \label{corollary:deginvariance} If $\tau \in \C$ satisfies \ref{assumption:components}, then we have $\op{deg}(D_{\tau} \cap \ov{{X}_0}(\G))= \op{deg}(D_{\tau}) \op{deg}(\ov{X_0}(\G)) = \op{deg}(\ov{X_0}(\G)).$ \ecor

The following technical lemma will be useful. 

\lem \label{lemma:continuouspoints} Assume that $K \sub Y$ satisfies \ref{assumption:dim1}--\ref{assumption:reduced}. Let $U_1\subset \C$ be the Zariski open subset of $\tau \in \C$ that satisfy \ref{assumption:components}. Then, over $\tau \in U_1$, the subset of intersection points $\{D_{\tau} \cap \ov{X_0}(\G)\} \sub \bb{CP}^n$ varies continuously in $\tau$ with respect to the Euclidean topology. \elem

\pf Let $\op{Hilb}_d(n)$ be the Hilbert scheme parametrizing all subschemes of $\bb{CP}^n$ whose Hilbert polynomial is the positive integer $d$. Topologically, we can think of $\op{Hilb}_d(n)$ as the set of $d$-tuples of unordered, not necessarily distinct points of $\bb{CP}^n$. 

By definition, for $\tau \in U_1$, the hypersurface $H_{\tau}$ does not contain any component of $\ov{X}(\G)$. Since the intersection $D_{\tau} \cap \ov{X_0}(\G)$ has degree $d$, there is an induced map $h: U_1 \to \op{Hilb}_d(n)$ given by $h(\tau) = \ov{X}_0(\G) \cap D_\tau.$ Since $h$ is in particular continuous with respect to the Euclidean topology, the lemma follows.  \epf

\Cref{lemma:continuouspoints} leads us to a useful reformulation of \ref{assumption:proper}. Let us define $L_{\infty} \sub \bb{CP}^n$ to be the hyperplane at infinity associated to the projective closure $X_0(\G) \sub \ov{X_0}(\G)  \sub \bb{CP}^n= \bb{A}^n \cup L_{\infty}$. 

Under the hypothesis that $\tau \in \C$ satisfies \ref{assumption:reduced} and \ref{assumption:components}, we have
\eq \label{equation:degreeadd} \op{deg}(D_{\tau} \cap X_0(\G))+ \op{deg}(D_{\tau} \cap (\ov{X_0}(\G) \cap L_{\infty})) = \op{deg}( D_{\tau} \cap \ov{X_0}(\G)). \eeq
Now \Cref{corollary:deginvariance} implies that the right hand side of \eqref{equation:degreeadd} is constant. It then follows by \Cref{lemma:continuouspoints} that \ref{assumption:proper} fails precisely at those points where $\op{deg}(D_{\tau} \cap (\ov{X_0}(\G) \cap L_{\infty})) $ jumps. 

We can now prove a lemma which was already promised in the proof of \Cref{proposition:generic1}.

\lem \label{lemma:post} Assume that $K \sub Y$ satisfies \ref{assumption:dim1}--\ref{assumption:reduced}. Then \ref{assumption:proper} holds for all but finitely many choices of $\tau \in \C$. \elem

\pf As in \Cref{lemma:continuouspoints}, let $U_1 \sub \C$ be the cofinite set on which \ref{assumption:components} is satisfied. By our previous discussion, it is enough to check that the degree of $D_{\tau} \cap (\ov{X_0}(\G) \cap L_{\infty})$ jumps at finitely many points of $U_1$. 

For $\tau \in U_1$, note that $(D_{\tau} \cap \ov{X_0}(\G) \cap L_{\infty})= (D_{\tau} \cap L_{\infty}) \cap ( \ov{ X_0}(\G)  \cap L_{\infty}),$ where equality holds both as topological spaces and as schemes. Now  $( \ov{ X_0}(\G)  \cap L_{\infty})$ is a zero-dimensional scheme, i.e. topologically a disjoint union of points $p_1,\dots,p_l$. We can think of $(D_{\tau} \cap L_{\infty})= \bb{CP}^{n-2} $ as a family of hyperplanes in $\bb{CP}^{n-1}$.  

The key algebro-geometric fact is that, for any positive integer $m \geq 1$, the condition for $(D_{\tau} \cap L_{\infty})$ to intersect a point $p_i$ with multiplicity $m$ is Zariski closed. This implies that the jumps occur in a Zariski closed subset of $U_1 \sub \C$, i.e. at a finite set of points. Since $U_1 \sub \C$ is cofinite, this proves the lemma. \epf

\cor \label{corollary:locallyconstant} Suppose that $K \sub Y$ satisfies \ref{assumption:dim1}--\ref{assumption:reduced} and let $U_{14}\subset \C$ be the Zariski open subset of $\tau \in \C$ that satisfy \ref{assumption:components} and \ref{assumption:proper}. Then $\op{deg}(H_{\tau} \cap X_0(\G))$ is independent of $\tau \in U_{14}$. \ecor

\pf We know that $\op{deg}( D_{\tau} \cap \ov{X_0}(\G))$ and $\op{deg}(D_{\tau} \cap (\ov{X_0}(\G) \cap L_{\infty}))$ are constant on $U_{14}$ by \Cref{corollary:deginvariance} and the discussion following \Cref{lemma:continuouspoints}. The conclusion now follows from \eqref{equation:degreeadd}. \epf

\cor \label{corollary:locallyconstant1} Suppose that $K \sub Y$ satisfies \ref{assumption:dim1}--\ref{assumption:reduced}. Let $U  \subset \C$ be the Zariski open subset of $\tau \in \C$ that satisfy   the four assumptions \ref{assumption:first}--\ref{assumption:last}; cf. \Cref{proposition:generic1}. Then, for $\tau \in (-2,2) \cap U$, the hypercohomology $\HP^*_{\tau}(K)$ is supported in degree $0$ and is independent of $\tau$. \ecor

\pf  We just saw in \Cref{corollary:locallyconstant} that $\op{deg}(H_{\tau} \cap X_0(\G))$ is constant for $\tau \in U$. Let us then set $d= \op{deg}(H_{\tau} \cap X_0(\G))$ for $\tau \in U$. 

According to \Cref{proposition: subfinal}, we have $H_{\tau} \cap X_0(\G)= \op{Spec}R_{\tau}$ where $R_{\tau}= \frac{\C [\e_1]}{\e_1^{n_1}} \tms \dots \tms \frac{\C [\e_k]}{\e_l^{n_k}}$. One can verify using the definition of degree that $\op{deg}(\op{Spec} R_{\tau})= \sum_{i=1}^k n_i$. Hence $d=\op{deg}(H_{\tau} \cap X_0(\G))= \sum_{i=1}^k n_i$. It now follows from \Cref{corollary:locallyinvariant} that $\HP^*_{\tau}(K)$ is supported in degree $0$ with $\HP^*_{\tau}(K)= \Z^d$.  \epf

\subsection{Relation to the $\widehat{A}$-polynomial} \label{sec:Apol} Let us choose a pair of loops $\mu, \lambda \sub \d E_K$ so that $(\mu, \lambda)$ is a basis for $\pi_1(\d E_K, x_0)= \Z \oplus \Z$ and $\lambda$ is nullhomologous in $E_K$. These loops are unique up to orientation and isotopy in $E_K$. 

Let us now outline the definition of the $A$- and $\widehat{A}$-polynomials. These are polynomials in two complex variables $m, l$ which are built from the $\op{SL}(2,\C)$-character variety of a knot in a homology $3$-sphere. The $A$-polynomial was introduced in \cite{planecurves} and has been extensively studied since then.  The $\widehat{A}$-polynomial is a close relative; it was defined by Boyer--Zhang \cite{boyer-zhang} and it also appears in the work of Boden and Curtis \cite{boden-cur} from which we draw our exposition. 

To define these polynomial invariants, we consider the inclusion $\d E_K \to E_K$ which induces a map $i_*: \pi_1(\d E_K, x_0) \to \pi_1(E_K, x_0)= \G$. Let $r: X(E_K) \to X( \d E_K)$ be the induced map on character varieties. Referring to \eqref{equation:generatingset}, we may assume that $i_*(\mu)= g_1$.

It is a general fact that points on the character variety $X(\d E_K)$ are in bijective correspondence with completely reducible representations $\pi_1(\d E_K, x_0) \to \op{SL}(2,\C)$. Since $\pi_1(\d E_K, x_0)= \Z \oplus \Z$ is abelian, the set of completely reducible representations coincides with the set of diagonal representations. 

We may thus define a map $t: \C^* \tms \C^* \to X(\d M)$ by sending a pair $(x,y) \in \C^* \tms \C^*$ to the unique diagonal representation $\rho$ such that $x$ is an eigenvalue of $\rho(\mu)$ and $y$ is an eigenvalue of $\rho(\lambda)$. One can check that $t$ is $2:1$ away from the points $(1,1), (1,-1), (-1,1), (-1,1)$. 

To define the $A$-polynomial, consider the affine variety $\ov{t^{-1}(r(X(E_K)))} \sub \C^* \tms \C^* \sub \C^2$.  It is a fact of algebraic geometry that any codimension $1$ subvariety in $\C^n$ is the vanishing locus of a principal ideal. It therefore makes sense to consider the following definition.

\defi (\cite{planecurves}) Let $A(K)= A(m, l) \in \C[m, l]$ be the generator of the vanishing ideal of the $1$-dimensional components of $\ov{t^{-1}(r(X(E_K)))} \sub \C^2$. We call $A(m, l)$ the \emph{$A$-polynomial} of $K \sub Y$. This polynomial is well-defined up to multiplication by nonzero scalars.
 \edefi 

To define the $\widehat{A}$-polynomial, let $\{ X_j\}$ be an enumeration of the one-dimensional components of $X(E_K)$ which contain an irreducible character and have the property that the Zariski closure $\ov{r(X_j)}$ is $1$-dimensional.   Let $\a_j$ be the degree of $r: X_j \to \ov{r(X_j)} \sub X( \d M)$. Finally, let $\widehat{A}_j(m, l)$ be the defining polynomial (of smallest degree) of $\ov{t^{-1}(r(X_j))} \sub \C^* \tms \C^* \sub \C^2$.

\defi \label{definition:ahatpoly} (\cite{boyer-zhang}) Let $\widehat{A}(K):= \prod_j \widehat{A}_j(m, l)^{\a_j}$. We say that $\widehat{A}(K)$ is the $\widehat{A}$-polynomial of $K \sub Y$. This polynomial is well-defined up to multiplication by nonzero scalars. \edefi

The relation between the $\widehat{A}$-polynomial and our knot invariant is expressed in the following statement. 

\thm \label{theorem:apolylink} Let $Y$ be an integral homology sphere. Suppose  that $K \sub Y$ satisfies \ref{assumption:dim1}--\ref{assumption:reduced} and that $\tau \in (-2,2)$ satisfies \ref{assumption:components}--\ref{assumption:proper}. It then follows that $\HP^*_{\tau}(K)= \mathbb{Z}^d$ in degree $0$, where $d= \op{deg}_{l} \widehat{A}(m, l)$. \ethm

Before proving \Cref{theorem:apolylink}, we need the following technical lemma. 

\lem \label{lemma:reducedtechnical} Assume that $K \sub Y$ satisfies \ref{assumption:dim1}--\ref{assumption:reduced}. Then for all but finitely many choices of $\tau \in \C$, the scheme theoretic intersection $H_{\tau} \cap X_0(\G) \sub \C^n$ is reduced and is topologically a finite set of points. \elem

\pf As in  \Cref{corollary:locallyconstant}, let $U_{14} \sub \C$ be the cofinite set on which \ref{assumption:first} and \ref{assumption:last} are satisfied. We saw there that $H_{\tau} \cap X_0(\G)$ has degree $d$ for all $\tau \in U_{14}$. It follows as in the proof of \Cref{lemma:continuouspoints} that there is a map $U_{14} \to \op{Hilb}_{d}(\C^n)$ sending $\tau \in U_{14}$ to $H_{\tau} \cap X_0(\G)$. Non-reducedness of the scheme-theoretic intersection $H_{\tau} \cap X_0(\G)$ is equivalent to two points colliding in the Hilbert scheme. This is a Zariski closed condition. Hence, if we can show that this condition is not always satisfied, it will follow that it is satisfied for at most finitely many points. 

To this end, referring to the group presentation \eqref{equation:generatingset}, observe that the $\{H_{\tau}\}_{\tau \in \C}$ are set theoretically just affine coordinate hyperplanes in $\C^n$, i.e. $H_{\tau}= \{ (x_1,\dots,x_n) \in \C^n \mid x_1= \tau\}$. Let us now consider the map $\pi: X_0(\G) \to \C$ given by projecting onto the first coordinate. If $H_{\tau} \cap X_0(\G)$ has a non-reduced point, then $\tau$ is a critical value of $\pi$. Hence it is enough to show that the set of critical values of $\pi$ is not all of $\C$. But since $X_0(\G)$ is a smooth manifold away from a finite set of singularities, this is a direct corollary of Sard's theorem. \epf

We can now give the proof of \Cref{theorem:apolylink}. 

\pf[Proof of \Cref{theorem:apolylink}] Let $C= \bigcup_j X_j \sub \C^n$ be the union of the $1$-dimensional components of $X(\G) \sub \C^n$ which contain an irreducible representation. Note that $C$ is precisely the closure of $X_0(\G)$ in $\C^n$. It follows from \Cref{corollary:locallyinvariant} and \Cref{lemma:reducedtechnical} that there is a cofinite set $V \sub (-2,2)$ such that $\sum_j \# \{ H_{\s} \cap X_j\} = d$ for $\s \in V$, where $\HP^*_{\s}(K)= \Z^d$ in degree $0$. By \Cref{corollary:locallyconstant1}, we also have $\HP^*_{\tau}(K)= \Z^d$ in degree $0$ for the given value of $\tau$.

In the notation of \Cref{definition:ahatpoly}, the restriction of $r$ to $X_j$ is a degree $\a_j$ map. This means that there is a Zariski open subset of $\ov{r(X_j)}$ in which every point has $\a_j$ preimages. Since $r(X_j)$ has dimension 1, this open subset has finite complement. Hence there is a cofinite subset $\tilde{V} \sub V \sub (-2,2)$ such that every point in $r(H_{\s} \cap X_j)$ has precisely $\a_j$ preimages for $\s \in \tilde{V}$. It follows that $r(H_{\s} \cap X_j)$ has cardinality $\# \{ H_{\s} \cap X_j \} / \a_j$ for $\s \in \tilde{V}$. 

We now consider the preimage of $r(H_{\s} \cap X_j)$ under $t$ for $\s \in \tilde{V}$. Since $\s \neq \pm 2$, this has cardinality $2 (\# \{ H_{\s} \cap X_j \} /\a_j)$. Fixing $m_0$ so that $2 \cos (m_0)= \tau'$, we observe that half of the points in the set $t^{-1}(r(H_{\s} \cap C_j))$ are solutions to $\widehat{A}_j(K)(m_0, \l)=0$ and the other half to $\widehat{A}_j(K)(1/m_0, \l)=0$.  In other words, the hyperplane $\{m_0=0\}$ intersects the zero set of $\widehat{A}$ in $\# \{ H_{\s} \cap X_j \} /\a_j$ points. 

Since this holds for all $\s \in \tilde{V}$, we conclude that $\op{deg}_{\l} \widehat{A}_j(K)= \# \{ H_{\s} \cap X_j \}/\a_j$.  But $\widehat{A}(K)= \prod_j \widehat{A}_j(K)^{\a_j}$ so $\op{deg}_{\l} \widehat{A}(K)= \sum_j \a_j (\# \{ H_{\s} \cap X_j \}/\a_j)= \sum_j \# \{ H_{\s} \cap X_j \}= d$. \epf

\pf[Proof of \Cref{thm:degd}] This is an immediate consequence of Proposition~\ref{proposition:generic1} and Theorem~\ref{theorem:apolylink}. 
\epf

\section{Some computations} \label{section:1dcomputations} In the previous section, we studied $P^{\bullet}_{\tau}(K)$ for knots satisfying \ref{assumption:dim1} and \ref{assumption:reduced}. We showed that $\HP^*_{\tau}(K)$ is generically independent of $\tau$ and expressible in terms of the $\widehat{A}$-polynomial of $K$; cf. \Cref{proposition:generic1} and \Cref{theorem:apolylink}.  In this section, we will consider examples of knots which satisfy \ref{assumption:dim1}--\ref{assumption:reduced}. We will use the results of the previous section to compute $\HP^*_{\tau}(K)$ for some of these knots. 

\subsection{Some knots satisfying \ref{assumption:dim1} and \ref{assumption:reduced}}  \label{subsection:dim1prelim} The first class of knots which we consider is the following. 

\defi A knot $K \sub S^3$ is said to be \emph{small} if its complement does not contain a closed, orientable, essential surface. Otherwise, $K$ is said to be \emph{large}. \edefi

The class of small knots is known to contain all 2-bridge knots \cite{hatcher-thurston} and all torus knots \cite{tsau}. A well-known result of Culler and Shalen \cite[Sec.\ 2.4]{planecurves} implies that the character variety of a small knot is at most $1$-dimensional. It follows that small knots satisfy \ref{assumption:dim1}. More information on the properties of small knots can be found in the survey of Ozawa \cite[Sec.\ 5]{ozawa}. 

We also wish to consider a distinguished class of hyperbolic knots.  Recall that a hyperbolic knot can be characterized by the property that its complement admits a discrete, faithful representation to $\op{PSL}(2,\C)$. It is a general fact that this distinguished representation can always be lifted to $\op{SL}(2,\C)$; see \cite[Prop.\ 3.1.1]{culler-shalen}.  Note however that the $\op{SL}(2,\C)$-character variety of a hyperbolic knot may in general contain many components, and one does not expect all of them to contain a discrete, faithful representation. 

\defi We say that a knot $K \sub S^3$ is \textit{slim} if it is hyperbolic and if every component of $X(\G)$ which contains an irreducible representation also contains a discrete, faithful representation. \edefi

Given an orientable hyperbolic $3$-manifold of finite volume with $n$ cusps, a well-known result of $3$-manifold topology which is usually attributed to Thurston (see \cite[Sec.\ 4.5]{shalen}) states that any component of its $\op{SL}(2,\C)$-character variety which contains a discrete, faithful representation is $n$-dimensional. In particular, any component of the character variety of a hyperbolic knot which contains a discrete, faithful representation is $1$-dimensional. It follows that all components of the character variety of a slim knot which contain an irreducible representation are $1$-dimensional. The remaining component of the character variety consists of abelian representations and is always $1$-dimensional. It follows that slim knots satisfy \ref{assumption:dim1}. 

It can be shown using \cite{burde} that hyperbolic twist knots are slim, because their character variety has only one component containing an irreducible representation. It has also been shown in \cite{mattman} that the $(-2,3,n)$ pretzel knot is slim provided that it is hyperbolic and that $n$ is not divisible by $3$. In fact, $(-2,3,n)$ is hyperbolic precisely when $n  \notin \{1,3,5\}$; see \cite[p.\ 1834]{macasieb-mattman}. We refer the reader to \cite[p.\ 2]{boyer-luft-zhang} for a discussion of these facts.  

The classes of small knots and slim knots overlap, but neither one is contained in the other. For example, it is a fact that all twist-knots are two-bridge knots, which implies that hyperbolic twist-knots are both small and slim. However, any two-bridge knot which is not hyperbolic (e.g. the trefoil) is automatically small but not slim. 

An example of a slim knot which is not small is the $(-2,3,7)$ pretzel knot, which is also known as the Fintushel-Stern knot. It is easy to see that this knot admits multiple Conway spheres, i.e. $2$-spheres which intersect the knot transversally in $4$ points. It can be shown that any knot which admits a Conway sphere automatically admits a genus $2$ incompressible surface in its complement, and therefore fails to be small. We refer to \cite{budney} for an illuminating exposition of these facts.

We now discuss to what extent the classes of small knots and slim knots satisfy \ref{assumption:reduced}. As we noted in \Cref{subsection:preparationsmall}, \ref{assumption:reduced} has been conjectured to hold for all knots in $S^3$. This has been verified for many examples, including all two-bridge knots and torus knots; see \cite[Sec.\ 2]{boden-cur}. It has also been verified by Le and Tran \cite{le-tran} for the $(-2,3,2n+1)$ pretzel knot for all $n \in \Z$. We can therefore state the following result, which can be viewed as a corollary of \Cref{theorem:apolylink} and the preceding discussion. The first part of this corollary is exactly Theorem~\ref{thm:concrete} from the Introduction.

\cor \label{corollary:knotpolyrelation} Suppose that $K \sub S^3$ is a two-bridge knot, a torus knot, or a $(-2,3, 2n+1)$ pretzel knot where $n \neq 0,1,2$ and $2n+1$ is not divisible by $3$. Then, for all but finitely many $\tau \in (-2,2)$, $\HP^*_{\tau}(K)= \Z^d$ in degree $0$ where $d= \op{deg}_{\l} \widehat{A}(K)$. More generally, this holds for all small and slim knots which satisfy \ref{assumption:reduced}. \ecor 

From a practical standpoint, \Cref{corollary:knotpolyrelation} does not help with computing $\HP^*_{\tau}(K)$ unless one is able to get a handle on the $\widehat{A}$-polynomial of $K$. Unfortunately, there appears to be little in the way of general computations of the $\widehat{A}$-polynomial in the literature (but see \cite{boden-cur} for the $\widehat{A}$-polynomial of certain Whitehead doubles). In contrast, the ordinary $A$-polynomial is known to be effectively computable, and has been computed for wide classes of examples. One can take advantage of these computations in cases where the $A$- and $\widehat{A}$-polynomial coincide. The class of slim knots turns out to have this property, and provides a source of examples which will be discussed in \Cref{subsection:slimcomputations}. 

\subsection{The trefoil and figure-eight} \label{subsection:trefoilfigure8} For the trefoil and figure-eight knots, the $\op{SL}(2,\C)$-character variety is well-understood. We can therefore compute $P^{\bullet}_{\tau}(K)$ explicitly. The arguments of this section are elementary and could potentially be pushed to other two-bridge knots, but they quickly become tedious. 

Let $K =3_1\sub S^3$ be a right-handed or left-handed trefoil. It follows from the above discussion that $K \sub S^3$ satisfies \ref{assumption:dim1}--\ref{assumption:reduced}, for example because it is a two-bridge knot.  According to \cite{porti}, the character variety of $K$ is given by the equations 
\eq \label{equation:trefoilvariety} X(\G)= \{(x,y) \mid (y-2)(x^2-y-1) =0 \}, \eeq where $\{(y-2)=0\}$ is the component of reducible representations, and $x$ is the trace of a meridian.  

One can check by hand that \ref{assumption:components} and \ref{assumption:proper} are always satisfied, and that \ref{assumption:smooth} is satisfied away from the points $(\pm \sqrt{3}, 2)$ which correspond to the intersection of the component of reducible representations $\{(y-2)=0\}$ with the component of irreducibles $\{x^2-y-1=0\}$. Using the fact that the Alexander polynomial of the trefoil is $\Delta(t)=t^2-t+1$, one can check that these are precisely the points ruled out by \ref{assumption:alexander}. 

In summary, the trefoil $K \sub S^3$ satisfies \ref{assumption:dim1}--\ref{assumption:reduced} and $\tau \in (-2,2)$ satisfies \ref{assumption:first}--\ref{assumption:last} provided that $\tau \neq \pm \sqrt{3}$.  It follows from \Cref{theorem:apolylink} that, for $\tau \in (-2,2) \setminus \{ -\sqrt{3}, \sqrt{3}\}$, we have $\HP^*_{\tau}(K)= \mathbb{Z}^d$ in degree $0$, where $d= \op{deg}_{\l} \widehat{A}(m, \l)$.

To compute $d$, let $\chi_l: X(\G) \to \C$ be the trace of the longitude of $K$. Observe that $(x, y) \mapsto (x, \chi_{l})$ restricts on $X(\G) \sub \C^2$ to an injective map which thus has degree $1$. It follows that $A(K)= \widehat{A}(K)$. 

Now $A(K)$ is computed for instance in \cite{cooper-long} and equals $ 1+m^6\l$ or $\l + m^6$, depending on the orientation.  The $\l$-degree is independent of the orientation, and we find $\op{deg}_{\l} A(K)= 1$. It follows from \Cref{theorem:apolylink} that $\HP^*_{\tau}(K)= \Z$ in degree zero, for $\tau \in (-2,2) \setminus \{ -\sqrt{3}, \sqrt{3}\}$.

Let us also consider what happens when $\tau = \pm \sqrt{3}$. We will show that $X^{\sqrt{3}}_{irr}(\G)$ and $X^{-\sqrt{3}}_{irr}(\G)$ are empty. This implies that $\HP^*_{\sqrt{3}}(K)= \HP^*_{-\sqrt{3}}(K)=0$ in all degrees. 

To this end, note that $X^{\tau}_{irr}(\G)$ is an open subvariety of the affine variety $X^{\tau}(\G)$, for any $\tau \in \C$. It is then a fact that $X^{\tau}_{irr}(\G)$ must have a closed point if it is non-empty; see \cite[3.6.J.(a)]{vakil}. It is therefore enough to show that $X^{\sqrt{3}}_{irr}(\G)$ and $X^{-\sqrt{3}}_{irr}(\G)$ have no closed points.

Recall first from \Cref{proposition:representablefunctor} that closed points on $X^{\tau}_{irr}(\G)$ for $\tau \in \C$ correspond to irreducible representations with trace $\tau$ along the meridian. Recall also that the closed points of $X(\G)$ correspond to completely reducible representations $\G \to \op{SL}(2,\C)$. Referring to \eqref{equation:trefoilvariety}, we see that the lines $\{x= \sqrt{3}\}$ and $\{x= - \sqrt{3}\}$ intersect the character variety of the trefoil in a single point. This means that the trefoil admits unique completely reducible representations having trace $\sqrt{3}$ and $-\sqrt{3}$ respectively along the meridian. 

It's easy to construct abelian representations having this property: just send a generator of $\Z= \op{Ab}(\G)$ to $\op{diag}(e^{ \pi i /6},e^{- \pi i/6})$ and $\op{diag}(e^{5\pi i/6}, e^{-5\pi i /6})$ respectively. Since abelian representations are in particular completely reducible, it follows that there are no irreducible representations with trace $\sqrt{3}$ and $-\sqrt{3}$ respectively along the meridian. Hence $X^{\sqrt{3}}_{irr}(\G)$ and $X^{-\sqrt{3}}_{irr}(\G)$ have no closed points.

We conclude that, as advertised in the Introduction, 
$$\HP^*_{\tau}(3_1)=\begin{cases}
\Z_{(0)} & \text{if } \tau \in (-2,2) \setminus \{\sqrt{3}, - \sqrt{3}\},\\
0 &\text{if } \tau \in \{\sqrt{3}, -\sqrt{3}\}.
\end{cases}.$$

The figure-eight knot can be handled similarly. According to \cite{porti}, its character variety is \eq X(\G)= \{(x,y) \mid (y-2)(y^2- (x^2-1)y+ x^2-1) =0 \}, \eeq where $\{(y-2)=0\}$ is the component of reducible representations, and $x$ is the trace of a meridian.  

The figure-eight $K=4_1$ satisfies \ref{assumption:dim1}--\ref{assumption:reduced} for the same reasons as the trefoil. One can again check by hand that \ref{assumption:first}--\ref{assumption:last} are satisfied for all $\tau \in (-2,2)$ ( \ref{assumption:alexander} turns out to be vacuous for the figure-eight knot).  We again have $A(K)= \widehat{A}(K)$, and it is shown in \cite{cooper-long} that $A(K)= \l^2 m^4+ \l(-m^8+m^6+ 2m^4+m^2-1)+ m^4$. We conclude from \Cref{theorem:apolylink} that 
$$\HP^*_{\tau}(4_1)= \Z^2_{(0)},$$
for any $\tau \in (-2,2)$.

\subsection{Computations for slim knots} \label{subsection:slimcomputations}  Let $K \sub S^3$ be a slim knot. Thanks to a theorem of Dunfield \cite[Cor.\ 3.2]{dunfield}, we know that the restriction map $r: X(E_K) \to X( \d E_K)$ which appears in the definition of $\widehat{A}(K)$ has degree $1$. This implies that $A(K)= \widehat{A}(K)$. As was noted previously, the $A$-polynomial of a knot in $S^3$ is effectively computable. It follows that the $\widehat{A}$-polynomial of a slim knot is also effectively computable.

In order to relate $P^{\bullet}_{\tau}(K)$ to $\widehat{A}(K)$, we need the character scheme of $K$ to be reduced, i.e. $K$ must satisfy \ref{assumption:reduced}. As we discussed in the paragraph preceding \Cref{corollary:knotpolyrelation}, this is not known in general for slim knots. However, \ref{assumption:reduced} is known for all twist knots (since they are two-bridge knots) and for all $(-2,3,2n+1)$ pretzel knots where $n \in \Z$ \cite{le-tran}. This provides a reasonably large source of examples for which we can compute $\HP^*_{\tau}(K)$ for generic values of $\tau$. 

\ex \label{example:slimknotscomp} Consider the twist knots $K_m$ with $m \neq 0$ full twists and one clasp, as in \cite{mathews}. In the Rolfsen knot table, the first twist knots are  $K_1= 3_1$ (trefoil), $K_{-1}=4_1$ (figure-eight), $K_2=5_2$, $K_{-2}=6_1$ (stevedore), $K_3=7_2$, $K_{-3}=8_1$. Except for the trefoil, they are all hyperbolic and hence slim. Their $A$-polynomials can all be found in the appendix of \cite{planecurves}. One finds that, for generic $\tau$,
$$\HP^*_{\tau}(5_2)= \Z^3_{(0)}, \ \HP^{*}_{\tau}(6_1)= \Z^4_{(0)}, \ \HP^*_{\tau}(7_2)= \Z^5_{(0)},\  \HP^*_{\tau}(8_1)= \Z^6_{(0)}.$$ (We warn the reader that the definition of the $A$-polynomial in \cite{planecurves} includes a factor of $(\l-2)$ corresponding to the component of irreducibles, which shifts the degree by $1$).  \eex

In fact, the $A$-polynomial of all twists knots has been computed explicitly, and is presented in a closed form in \cite{mathews}. One can verify using \cite{mathews} that the $\l$-degree of the $A$-polynomial is always $2$ less than the crossing number, which explains the pattern in the above examples. (We will see in \Cref{example:pretzel} that this pattern is special to twist knots.) Thus, for general twist knots and generic $\tau$, we have
$$\HP^*_{\tau}(K_n)=\begin{cases}
\Z^{2n-1}_{(0)} & \text{if } n >0,\\
\Z^{-2n}_{(0)} &\text{if } n<0.
\end{cases}$$

A different class of examples for which we can explicitly determine $\HP^*_{\tau}(K)$ are certain hyperbolic pretzel knots.  As was mentioned previously, Mattman \cite{mattman} showed that the $(-2,3,m)$ pretzel knots is slim provided that it is hyperbolic (which happens if and only if $m \notin \{1,3,5\}$) and that $m$ is not divisible by $3$. Moreover, Le and Tran showed that the character variety is reduced for all $(-2,3,2n+1)$ pretzel knots where $n \in \Z$ \cite{le-tran}. Finally, in \cite{pretzelpoly} the authors identify a recursion relation which allows one to compute $A(-2,3,2n+1)$ for $n \in \Z$. 

Putting these results together, it follows that one can in principle compute $\HP^*_{\tau}(P(-2,3, 2n+1))$ provided that $n \notin \{0, 1, 2\}$ and that $2n+1$ is not divisible by $3$. We refer the reader to \cite{pretzelpoly} for the precise recursion relation, and limit ourselves to the following example. 

\ex \label{example:pretzel} The pretzel knot $(-2,3,7)$, also known as the Fintushel-Stern knot, has crossing number $12$. As we remarked in \Cref{subsection:dim1prelim}, $(-2, 3,7)$ is not a small knot, since its complement contains an incompressible genus $2$ surface. It is shown in \cite{pretzelpoly} that $A(-2,3,7)= -1+ \l m^8- 2 \l m^{10}+ 2 \l m^{20} + \l^2 m^{22} - \l^4 m^{4} - 2\l^4 m^{42} - \l^5 m^{50} + 2 \l^5 m^{52}- \l^5 m^{54}  + \l^6 m^{62}$. It follows that 
$$\HP^*_{\tau}(P(-2,3,7))= \Z^6_{(0)},$$ for all but finitely many values of $\tau \in (-2,2)$. \eex

\section{Connected sums of knots} \label{section:connectedsums}

In this section, we consider a class of knots whose $\op{SL}(2, \C)$-character variety has $2$-dimensional components. A knot $K$ in this class can be constructed as the connected sum of two knots satisfying \ref{assumption:dim1} and \ref{assumption:reduced} in \Cref{subsection:assumptions}. Under certain assumptions on $\tau \in (-2,2)$ which are satisfied generically, we will show that $\HP^*_{\tau}(K)$ is supported in degrees $-1$ and $0$ and compute it explicitly for some examples. 

\subsection{Topological preliminaries} For the purpose of fixing some notation, we begin by briefly reviewing how to form the connected sum of two knots.  

Let $K_1 \sub Y_1$ and $K_2 \sub Y_2$ be oriented knots, where $Y_1, Y_2$ are homology $3$-spheres. Choose embedded closed balls $\ov{B}_1 \sub Y_1, \ov{B}_2 \sub Y_2$ such that $\ov{B}_i-K_i$ is diffeomorphic to $\{(x,y, z) \mid \|(x,y,z)\| \leq 1, (x,y,z) \neq (x,0,0)\}$. Let $B_i \sub \ov{B}_i$ be the (open) interior. Let $C_i:= \d (Y_i- B_i)$ and observe that $C_i \cap K_i$ is a disjoint union of two points, which are canonically ordered by the orientation on $K_i$. Let $\phi: C_1 \to C_2$ be a diffeomorphism which maps $C_1 \cap K_1$ to $C_2 \cap K_2 $ and preserves the ordering. 

\defi We say that $K_1\# K_2:= (K_1-B_1) \cup_{\phi} (K_2-B_2)$ is the \emph{connected sum} of $K_1$ and $K_2$. Note that $K_1\#K_2$ is a knot in $Y_1 \# Y_2:= (Y_1-B_1) \cup_{\phi} (Y_2 -B_2)$. It can be shown that the connected sum is well-defined up to equivalence of knots, even though the construction depends a priori on many choices. \edefi

A knot in $S^3$ which is the connected sum of two nontrivial knots is often said to be a \emph{composite} knot. Two well-known examples of composite knots are the granny knot and the square knot, which are respectively the connected sum of two trefoils with the same and opposite orientations. 


The fundamental group of $K_1\#K_2$ can be easily computed using van Kampen's theorem. Writing $K:= K_1 \# K_2$ and $Y= Y_1\#Y_2$, we have
\begin{align} \label{equation:fiberconnect} \pi_1(Y-K) &= \pi_1(Y_1-K_1-B_1) *_{\pi_1(S^2-p_1-p_2)} \pi_1(Y_2-K_2-B_2) \\
&= \pi_1(Y_1-K_1) *_{\pi_1(S^2-p_1-p_2)} \pi_1(Y_2-K_2). \nonumber
\end{align}

Here $\{p_1, p_2\} = C_1 \cap K_1 \equiv_{\phi} C_2 \cap K_2$, where the ordering is inherited from the orientation of $K_1, K_2$. The fundamental groups considered in \eqref{equation:fiberconnect} are assumed to be defined with respect to some fixed basepoint $x_0 \in S^2-p_1-p_2$. We do not keep track of this choice in our notation since it plays no role in this section. 

For future reference, we let $\iota_i: \pi_1(Y_i-K_i) \to \pi_1(Y-K)$ be the natural maps associated to the amalgamated product. Note that the fundamental group of $K=K_1 \# K_2$ is independent of the choice of orientations of $K_1$ and $K_2$.

We will rely throughout this section on definitions and notation introduced in \Cref{subsection:relrepgeneralities}. However, we deviate from the notation of the previous sections in one important way: namely, given a knot $K \sub Y$, we will let $\mathscr{R}(Y-K)$ and $\mathscr{X}(Y-K)$ denote the representation and character schemes of $\pi_1(Y-K)$ with respect to an unspecified basepoint $x_0$. We follow the analogous notational convention for the representation and character varieties, and for their relative counterparts. 

There is an inconsistency in the fact that we are now considering representation and character varieties of knot complements: in Sections \ref{section:constantdim1} and \Cref{section:1dcomputations}, we always considered representation and character varieties associated to the blowup $E_K$ of a doubly-pointed pointed knot $(Y, K, p,q)$. The reason for this change is that it allows us to avoid considering blowups of connected sums. Of course, the fundamental group of a blowup is isomorphic to that of the knot complement, so one can always pass between these choices (in a non-canonical way).

\subsection{Some assumptions} We will now restrict our attention to knots in integral homology spheres satisfying some additional assumptions. As in \Cref{subsection:assumptions}, it will be convenient to label these assumptions separately. 

Let $Y$ be an integral homology sphere and let $K \sub Y$ be a knot. Let $\tau \in (-2,2)$ be a real parameter. We consider the following assumptions on this data. 

\begin{customassumption}{C.1} \label{assumption:smoothchar} The irreducible locus of the relative character scheme $\mathscr{X}^{\tau}_{irr}(Y-K)$ is smooth. \end{customassumption}

\begin{customassumption}{C.2} \label{assumption:0dimchar} The irreducible locus of the relative character scheme $\mathscr{X}^{\tau}_{irr}(Y-K)$ is zero-dimensional. \end{customassumption}

\begin{customassumption}{C.3} \label{assumption:repeatalexander} If $x \in [0,1]$ has the property that $e^{4 \pi i x}$ is a root of the Alexander polynomial of $K \sub Y$, then $\tau \neq 2 \cos (2 \pi x)$. \end{customassumption}

Note that \ref{assumption:smoothchar}--\ref{assumption:repeatalexander} depend both on the topological data $K \sub Y$ and on $\tau \in (-2,2)$. Moreover \ref{assumption:smoothchar}--\ref{assumption:repeatalexander} are closely related to the assumptions introduced in \Cref{subsection:assumptions}. In fact, \ref{assumption:repeatalexander} is essentially identical to \ref{assumption:alexander}, while we will show in \Cref{subsection:connectedcomp} that the knot $K \sub Y$ satisfies \ref{assumption:smoothchar}--\ref{assumption:repeatalexander} for all but finitely many $\tau \in (-2,2)$ provided that it satisfies \ref{assumption:dim1} and \ref{assumption:reduced}.

It will be useful to introduce the following definition. 

\defi 
	Given a finitely-presented group $\G$, a representation $\G \to \op{SL}(2,\C)$ is said to be \emph{diagonal} if it can be conjugated to have image consisting only of diagonal matrices. Diagonal representations are in particular reducible.
\edefi 

Let us now state an important proposition which will be used throughout this section.

\prop \label{proposition:types} For $i=1,2$ let $K_i \sub Y_i$ be a knot in an integral homology sphere. Let $K=K_1 \# K_2 \subset Y=Y_1 \# Y_2$ and suppose that $\tau \in (-2,2)$ satisfies \ref{assumption:repeatalexander} with respect to $K_1 \sub Y_1$ and $K_2 \sub Y_2$. Then any irreducible representation of $\pi_1(Y-K)$ into $\op{SL}(2,\C)$ must pull back to an irreducible representation along $\iota_1$ or $\iota_2$ (or both).   \eprop

This proposition justifies the following definition. 

\defi \label{definition:types}
	For $i=1,2$, let $K_i \sub Y_i$ be as in \Cref{definition:types}. The representations $\pi_1(Y-K) \to \op{SL}(2,\C)$ which pull back to an irreducible representation along $\iota_i$ and to a reducible representation along $\iota_j$ for $i \neq j$ with $i,j \in \{1,2\}$ are said to be of Type I. The representations which pull back to an irreducible representation along both factors are said to be of Type II.  
\edefi

\pf[Proof of \Cref{proposition:types}] We need to check that $\rho$ is reducible if $\rho \circ \iota_i$ is reducible for $i=1,2$. To this end, it is useful to note that if $\rho \circ \iota_i$ is reducible, then it is abelian. This is a consequence of \cite[Prop.\ 6.1]{planecurves} and the fact that $\tau \in (-2,2)$ satisfies \ref{assumption:repeatalexander} for $K_i \sub Y_i$. It then follows from the classification of $\op{SL}(2,\C)$ representations in \cite[Sec.\ 2.1]{abou-man} that $\rho \circ \iota_i$ is in fact diagonal since $\tau \neq \pm 2$. It is therefore enough to check that the fact that $\rho \circ \iota_i$ are diagonal for $i=1,2$ implies that $\rho$ is diagonal. 

Let us then suppose that $x \in \pi_1(S^2-p_1-p_2)$ is a generator. Viewing $x$ as an element of both $\pi_1(Y_1-K_1)$ and $\pi_1(Y_2-K_2)$, note that $\rho \circ \iota_1(x)= \rho \circ \iota_2(x)$.  If we assume that $\rho \circ \iota_1$ is diagonal, then there is an element $g \in \op{SL}(2,\C)$ such that $g \rho \iota_1(x) g^{-1} = \op{diag}(\l, \l^{-1})$ for some $\l \in \C^*$, where $\l \neq \pm 1$. It now follows from \Cref{lemma:diag} below that in fact $\rho \circ \iota_1$ and $\rho \circ \iota_2$ are in fact simultaneously diagonalized by $g$.  Since the image of the $\iota_i$ generates $\pi_1(Y-K)$, we conclude that $g$ diagonalizes $\rho$. \epf

\lem \label{lemma:diag} Let $\G$ be a finitely-generated group and let $\s: \G \to \op{SL}(2,\C)$ be a diagonal representation. Suppose that there exists an element $h \in \G$ with $\s(h) = \op{diag}(\l, \l^{-1})$ for $\l \neq \pm 1$. Then $\op{Im} \s \sub \{ \op{diag}(z, z^{-1}) \mid z \in \C^* \}$, i.e. $\s$ is diagonal. \elem
\pf By hypothesis, there exists $M \in \op{SL}(2,\C)$ such that $M \s M^{-1}$ is diagonal. In particular, this means that $M \s(h) M^{-1}=M \op{diag}(\l, \l^{-1}) M^{-1}$ is diagonal.  One can check that this implies that either $M=D$ for $D \in  \{ \op{diag}(\l, \l^{-1}) \mid \l \in \C^* \}$ or $M = J D$ for $J= \begin{pmatrix} 0& -1 \\
1& 0
\end{pmatrix}$ and $D \in \{ \op{diag}(\l, \l^{-1}) \mid \l \in \C^* \}$. 

The conclusion now follows from the fact -- which is also straightforward to verify -- that conjugation by $D$ or $JD$ cannot send a non-diagonal matrix to a diagonal matrix. \epf

\subsection{Analysis of the character variety}   According to \Cref{lemma:fiberrep}, the isomorphism \eqref{equation:fiberconnect} gives rise to a fiber product of relative representation schemes \eq \label{equation:fiberproductsum} \mathscr{R}^{\tau}(Y-K)= \mathscr{R}^{\tau}(Y_1-K_1) \tms_{\mathscr{R}^{\tau}(S^2-p_1-p_2)} \mathscr{R}^{\tau}(Y_2-K_2). \eeq 

Let $\pi_i: \mathscr{R}^{\tau}(Y_i-K_i) \to \mathscr{R}^{\tau}(S^2-p_1-p_2)$ be the maps associated to the fiber product. Observe that the $\pi_i$ are compatible with the usual conjugation action of the algebraic group $\op{SL}_2$. 

By a slight abuse of notation, we also let $\pi_i: R^{\tau}(Y_i - K_i) \to R^{\tau}(S^2-p_1-p_2)$ be the corresponding maps on character varieties. Observe that $\pi_i$ maps a representation $\rho$ to its restriction to $\pi_1(S^2-p_1-p_2)$. The conjugation action of $\op{SL}(2,\C)$ is compatible with $\pi_i$. Since $\{ \pm \op{Id}\}$ acts trivially, it descends to an action of $\op{PSL}(2,\C)$. 

Given $\rho : \pi_1(Y_i-K_i) \to \op{SL}(2,\C)$ for some $i=1,2$, let $\mathscr{R}^{\tau}(Y_i-K_i)_{\rho}$ be the connected component of $\rho$ in $\mathscr{R}^{\tau}(Y_i-K_i)$ and let ${\pi_i}|_{\rho}$ be the restriction of $\pi_i$ to $\mathscr{R}^{\tau}(Y_i-K_i)_{\rho}$.

The next two propositions give a more precise description of ${\pi_i}|_{\rho}$ in the case where $\rho$ is reducible and irreducible respectively. Eventually, we wish to consider a certain fiber product over $\mathscr{R}^{\tau}(S^2-p_1-p_2)$ and to show that it forms a smooth scheme; see \Cref{proposition:schemesmoothsum}. In ordinary differential topology, one often verifies that a fiber product is smooth by showing that the maps from each factor onto the base are submersions. The scheme theoretic analog of a submersion is a \emph{smooth morphism}. We will need to make use of this notion which unfortunately requires a significant amount of scheme theory to define; see \cite[III.\ Chap.\ 10]{hartshorne}. However, we will mostly use it as a black box by appealing to general results which allow one to promote ordinary submersions on closed points to smooth morphisms of schemes.

\prop \label{proposition:reduciblecomponents} For $i =1,2$, suppose that $K_i \sub Y_i$ and $\tau \in (-2,2)$ satisfy \ref{assumption:repeatalexander}.  If $\rho \in R^{\tau}(Y_i-K_i)$ is a reducible representation, then ${\pi_i}|_{\rho}: \mathscr{R}^{\tau}(Y_i-K_i)_{\rho} \to \mathscr{R}^{\tau}(S^2-p_1-p_2)$ is an isomorphism. Moreover, $\mathscr{R}^{\tau}(Y_i-K_i)_{\rho}$ is the unique component of $\mathscr{R}^{\tau}(Y_i-K_i)$ containing a reducible representation. \eprop

\pf We saw in the proof of \Cref{proposition:types} that all reducible $\op{SL}(2,\C)$ representations of $\pi_1(Y_i-K_i)$ are in fact abelian. Since any abelian representation factors through $H_1(Y_i-K_1)= \Z$, this implies that all such representations are conjugate. Hence they belong to the same component of $\mathscr{R}^{\tau}(Y_i-K_i)$. 

One can check by hand that 
$$\mathscr{R}^{\tau}(S^2-p_1-p_2)= \mathscr{R}^{\tau}(\Z) \simeq \op{PSL}(2,\C)/ \C^* \simeq TS^2$$ is smooth as a scheme provided that $\tau \neq \pm 2$. Since $\mathscr{R}^{\tau}(Y_i-K_i)_{\rho}$ is a component of abelian representations, one can show that $\mathscr{R}^{\tau}(Y_i-K_i)= \mathscr{R}^{\tau}(\Z)$ which implies that $\mathscr{R}^{\tau}(Y_i-K_i)$ is also smooth. 

We previously observed that the map of representation varieties ${\pi_i}|_{\rho}: R^{\tau}(Y_i-K_i)_{\rho} \to R^{\tau}(S^2-p_1-p_2)$ is compatible with the conjugation action of $\op{PSL}(2,\C)$. This action is evidently transitive on $R^{\tau}(S^2-p_1-p_2)$, from which it follows that ${\pi_j}|_{\rho}$ is surjective at the level of closed points and of tangent spaces. This can be shown to imply (see \cite[III, Prop.\ 10.4]{hartshorne}) that $\pi_i$ is a smooth morphism of schemes. \epf

\prop \label{proposition:irreduciblesmooth} For $i=1,2$, suppose that $K_i \sub Y_i$ and $\tau \in (-2,2)$ satisfy \ref{assumption:smoothchar} and \ref{assumption:repeatalexander}.   If $\rho \in R^{\tau}(Y_i-K_i)$ is irreducible, then ${\pi_i}|_{\rho}: \mathscr{R}^{\tau}(Y_i-K_i)_{\rho} \to \mathscr{R}^{\tau}(S^2-p_1-p_2)$ is a smooth morphism of schemes. \eprop

\pf It follows by combining \ref{assumption:smoothchar} with \Cref{proposition:regularitycriterion} that $\mathscr{R}^{\tau}(Y_i-K_i)_{\rho}$ is a smooth scheme. We saw in the proof of the previous proposition that $\mathscr{R}^{\tau}(S^2-p_1-p_2)$ is a smooth scheme. We again observe that the map ${\pi_i}|_{\rho}: R^{\tau}(Y_i-K_i)_{\rho} \to R^{\tau}(S^2-p_1-p_2)$ is compatible with the conjugation action of $\op{PSL}(2,\C)$, and that ${\pi_j}|_{\rho}$ is surjective at the level of closed points and of tangent spaces. Using again \cite[III, Prop.\ 10.4]{hartshorne}, this implies that ${\pi_i}|_{\rho}$ is a smooth morphism of schemes. \epf

\prop \label{proposition:schemesmoothsum} For $i =1,2$, suppose that $K_i \sub Y_i$ and $\tau \in (-2,2)$ satisfy \ref{assumption:smoothchar} and \ref{assumption:repeatalexander}.  Then $\mathscr{X}^{\tau}_{irr}(K_1 \# K_2)$ is smooth as a scheme. Moreover, each connected component consists exclusively or Type I or Type II representations.  \eprop

\pf It follows by combining \Cref{proposition:types} and Equation \eqref{equation:fiberproductsum} that each connected component of $R^{\tau}_{irr}(K_1 \#K_2)$ consists exclusively of Type I or Type II representations. It follows that the same is true for the connected components of $X^{\tau}_{irr}(K_1\# K_2)$, which is the quotient of $R^{\tau}_{irr}(K_1 \# K_2)$ by a free $\op{PSL}(2,\C)$ action. 

We saw in Propositions \ref{proposition:reduciblecomponents} and \ref{proposition:irreduciblesmooth} that $\pi_i$ restricts to a smooth morphism of schemes on $\mathscr{R}^{\tau}(Y_i-K_i)$. It is a general fact that the property of being a smooth morphism is closed under fiber products and composition; see \cite[III, Prop.\ 10.1]{hartshorne}. Appealing again to \eqref{equation:fiberproductsum}, this implies in particular that the composition $\mathscr{R}^{\tau}_{irr}(Y-K) \to \mathscr{R}^{\tau}(S^2-p_1-p_2) \to \op{Spec} \C$ is a smooth morphism of schemes, which means that $\mathscr{R}^{\tau}_{irr}(Y-K)$ is a smooth scheme. Appealing to \Cref{proposition:regularitycriterion}, this implies that $\mathscr{X}^{\tau}_{irr}(Y-K)$ is also a smooth scheme.  \epf

We record the following lemma, which follows immediately from \eqref{equation:fiberproductsum}. 

\lem \label{lemma:type1characterize} Assuming that $K_i \sub Y_i$ and $\tau \in (-2,2)$ satisfy \ref{assumption:smoothchar} and \ref{assumption:repeatalexander} for $i=1,2$, the locus of Type I components of $X^{\tau}_{irr}(K_1 \# K_2)$ is topologically the disjoint union of $X^{\tau}_{irr}(K_1)$ and $X^{\tau}_{irr}(K_2)$.    \elem

We obtain the following corollary by combining \Cref{proposition:schemesmoothsum} with \Cref{proposition:localsystem}. 

\cor \label{corollary:localsytemcomp} Under the hypotheses of \Cref{proposition:schemesmoothsum}, the perverse sheaf $P^{\bullet}_{\tau}(K)$ is a local system. Its stalks over every k-dimensional component are isomorphic to $\Z$, supported in degree $-k$. \ecor

We now specialize to knots which satisfy \ref{assumption:0dimchar} in addition to \ref{assumption:smoothchar} and \ref{assumption:repeatalexander}. Given $\tau \in (-2,2)$, if $K_i \sub Y_i$ is a knot in an integral homology sphere which satisfies \ref{assumption:0dimchar}, it makes sense to let $m_i \in \bb{N}$ be the number of components of $X_{irr}^{\tau}(Y_i-K_i)$.  We now have the useful following proposition.  

\prop \label{} For $i =1,2$, suppose that $K_i \sub Y_i$ and $\tau \in (-2,2)$ satisfy \ref{assumption:smoothchar}--\ref{assumption:repeatalexander}. If $\rho \in R^{\tau}(Y_i-K_i)$ is irreducible, then $R^{\tau}(Y_i-K_i)_{\rho}$ is a $\C^*$-bundle over $R^{\tau}(S^2-p_1-p_2) \simeq \op{PSL}(2,\C)/ \C^*$. \eprop

\pf We showed in \Cref{proposition:schemesmoothsum} that ${\pi_i}|_{\rho}: \mathscr{R}^{\tau}(Y_i-K_i)_{\rho} \to \mathscr{R}^{\tau}(S^2-p_1-p_2)$ is a smooth morphism of schemes if $\rho \in R^{\tau}(Y_i-K_i)$ is irreducible. The discussion preceding the proof of \Cref{proposition:tangentspace} (2) also implies that $R^{\tau}(Y_i-K_i)_{\rho}$ is a $\op{PSL}(2,\C)$-bundle over the point $[\rho] \in X^{\tau}_{irr}(Y_i-K_i)$. 

Let us now show that ${\pi_i}|_{\rho}: R^{\tau}(Y_i-K_i)_{\rho} \to R^{\tau}(S^2-p_1-p_2)$ is a $\C^*$-bundle. Given any $\s \in  R^{\tau}(Y_i-K_i)_{\rho}$, we will show that the fiber over ${\pi_i}|_{\rho}(\s)$ is a copy of $\C^*$. Let $\mu \in \pi_1(S^2-p_1-p_2)$ be a generator. After possibly conjugating by an element of $\op{SL}(2,\C)$, we may as well assume that $\s(\mu)$ is diagonal. Since $\op{SL}(2,\C)$ acts transitively on $R^{\tau}(Y_i-K_i)_{\rho}$, the fiber over ${\pi_i}|_{\rho}(\s)$ can be identified with the orbit of $\s$ under the action of $\op{Stab}(\s(\mu))$. But since $\s(\mu)$ is diagonal and $\op{Tr} \s(\mu) \neq \pm 2$, one can check that $\op{Stab}(\s(\mu))= \{ \op{diag}(\l, \l^{-1}) \mid \l \in \C^*\}$. Finally, since the kernel of the $\op{SL}(2,\C)$-action is precisely $\{\pm \op{Id}\}$, it follows that the orbit of $\s$ under $\op{Stab}(\s(\mu))$ is $\C^* / \{\pm \op{Id}\}= \C^*$. \epf

If we continue to assume that $K_i \sub Y_i$ and $\tau \in (-2,2)$ satisfy \ref{assumption:smoothchar}--\ref{assumption:repeatalexander}, we can precisely describe $X^{\tau}_{irr}(Y-K) \sub X^{\tau}(Y-K)$.   Indeed, \eqref{equation:fiberproductsum} implies that we can think of a representation $\rho: \pi_1(Y-K) \to \op{SL}(2,\C)$ as a pair $(\rho_1, \rho_2)$ where $\rho_i: \pi_1(Y_i-K_i) \to \op{SL}(2,\C)$ and where the $\rho_i$ restrict to the same representation on $\pi_1(S^2-p_1-p_2)$. One then counts that there are precisely $m_1$ Type I components of $\mathscr{R}^{\tau}(Y-K)$ which restrict to an abelian representation on the first factor, and $m_2$ Type I components which restrict to an abelian representation on the second factor. Observe that each such component is topologically a copy of $\op{PSL}(2,\C)$. Recalling that $\op{PSL}(2,\C)$ acts transitively on $R^{\tau}_{irr}(Y-K)$ by conjugation, we find that the locus of Type I components of $X^{\tau}_{irr}(Y-K)$ is a disjoint union of $m_1+m_2$ isolated points. 

Similarly, one counts $m_1m_2$ Type II components of $R^{\tau}_{irr}(Y-K)$. These form $\C^* \tms \C^*$ bundles over ${R}^{\tau}(S^2-p_1-p_2) \simeq \op{PSL}(2,\C)/\C^*$.  The $\op{PSL}(2,\C)$ conjugation action is free and transitive and one finds that the image of each component in the character variety is a copy of $\C^*$. We summarize the previous discussion with the following corollary, which also relies on \Cref{proposition:schemesmoothsum} and \Cref{corollary:localsytemcomp}.

\cor \label{corollary:connectedsumdescription} Suppose that $K_i \sub Y_i$ and $\tau \in (-2,2)$ satisfy \ref{assumption:smoothchar}--\ref{assumption:repeatalexander} for $i=1,2$, and let $m_i \in \mathbb{N}$ be defined as above. Then the irreducible locus $\mathscr{X}^{\tau}_{irr}(Y-K)$ is a smooth scheme. Topologically, it is a disjoint union of $m_1+m_2$ points corresponding to Type I representations, and $m_1m_2$ copies of $\C^*$ corresponding to Type II representations. The restriction of the perverse sheaf $P^{\bullet}_{\tau}(K)$ to any $k$-dimensional component is a local system whose stalks are isomorphic to $\Z$, in degree $-k$. \ecor

\subsection{Local systems}  \label{subsection:localsystems} We now wish to strengthen \Cref{corollary:connectedsumdescription} by showing that the local system $P^{\bullet}_{\tau}(K)$ is in fact trivial.  It is evidently enough to consider the Type II components of $X^{\tau}_{irr}(Y-K)$ which are diffeomorphic to $\C^*$. Let $\mc{C}$ be such a component. 

Recall from \Cref{subsection:preliminarydefs} that a Heegaard splitting $\mc{H}= (\S, U_0, U_1)$ induces inclusions $X^{\tau}_{irr}(U_i) \to X^{\tau}_{irr}(\S)$. These inclusions are Lagrangian with respect to the natural symplectic form on $X^{\tau}_{irr}(\S)$ and we therefore write $X^{\tau}_{irr}(U_i):= L_i$. We showed in \Cref{subsection:spinstructures} that the $L_i$ carry unique spin structures provided that the genus of $\S$ is at least $6$, which we can always assume. We can then identify $\mc{C}$ with a component of $L_0 \cap L_1$. 

In order to show that $P^{\bullet}_{\tau}(K)$ is a trivial local system on $\mc{C}$, we will appeal to Lemma 6.3 of \cite{abou-man} and to the discussion preceding this lemma. The main objects under consideration are a certain local system $|W^+|$ on $\mc{C}$ and a map \eq \label{equation:spinmap} TL_0|_{\mc{C}} \to TL_1|_{\mc{C}}. \eeq Here \eqref{equation:spinmap} corresponds to (29) in \cite{abou-man}, which is described in the paragraphs preceding the statement of Lemma 6.3 in \cite{abou-man}.

In our particular setting, Lemma 6.3 of \cite{abou-man} says that $$P^{\bullet}_{\tau}(K)|_{\mc{C}} \simeq |W^+|[1],$$ provided that the map \eqref{equation:spinmap} preserves spin structures. Our first task in the remainder of this section is to show that \eqref{equation:spinmap} does indeed preserve spin structures. Then, we will describe $|W^+|$ more carefully and argue that it is a trivial local system.

We begin by introducing a Heegaard splitting which is well adapted to our setting. Unlike in \Cref{section:definitionofinvariant}, we will consider Heegaard splittings for \emph{knot complements} rather than knot exteriors. Since these have isomorphic homotopy groups, the distinction is immaterial for our present purposes -- indeed, we only needed to emphasize this distinction in \Cref{section:definitionofinvariant} for the purpose of proving the naturality of our invariant.

Let $(\Sigma, p, q, U_0, U_1)$ be a Heegaard splitting of genus at least $3$ for $Y_1 - K_1$ where $p, q \in K$ are marked points on $\S$. Similarly, let $(\Sigma', p', q', U_0', U_1')$ be a Heegaard splitting for $Y_2 - K_2$ of genus at least $3$. 

We get a Heegaard splitting (of genus at least $6$) for $ (Y_1 \# Y_2) - (K \# K')= Y - K = $ by identifying disks around $q$ and $p'$: 
$$ (\Sigma \# \Sigma', p, q', U_0 \#_b U_0', U_1 \#_b U_1'),$$
where $\#_b$ denotes the boundary connected sum.

Let $G= \op{SL}(2,\C)$ and let $\fk{C}^{\tau} \sub G$ be the conjugacy class of elements of trace $\tau$. The relative character variety of the Heegaard surface used for $K_1$ can be described as
\begin{align*}
 X^{\tau}(\S) &= \{(A_1, \dots, A_g, B_1, \dots, B_g, C_1, C_2) \in G^{2g+2}\mid C_1, C_2 \in \fk{C}^{\tau}, \prod [A_i, B_i] C_1 C_2 = 1 \}/G \\
 &= \{(A_1, \dots, A_g, B_1, \dots, B_g, C_1) \in G^{2g+1}\mid C_1 \in \fk{C}^{\tau}, \prod [A_i, B_i] C_1 = \op{diag}(t, t^{-1}) \}/\C^*,
 \end{align*}
where $t + t^{-1} = \tau$. We have a similar description for $K_2$:
$$  X^{\tau}(\S') = \{(A_1', \dots, A_h', B_1', \dots, B_h', C_2') \in G^{2h+1}, \mid C_2' \in \fk{C}^{\tau}, \prod [A_i', B_i'] C_2' = \op{diag}(t, t^{-1}) \}/\C^*.$$

For the connected sum, we have
\begin{multline*}
 X^{\tau}(\S \# \S') = \{(A_1,\dots, A_g, B_1, \dots,  B_g, C_1, A_1', \dots, A_h', B_1', \dots, B_h', C_2') \in G^{2g+2h+2} \mid C_1, C_2' \in \fk{C}^{\tau}, \\
 \prod [A_i, B_i] C_1 = \prod [A_i', B_i'] C_2' \}/G 
\end{multline*}

Note that $X^{\tau}(\S) \times X^{\tau}(\S')$ lives inside $X^{\tau}(\S \# \S')$ as a complex codimension one subset, given by asking that the holonomy around the curve where we do the connect sum, $\prod [A_i, B_i] C_1 = \prod [A_i', B_i'] C_2'$, has trace $\tau$. 

Consider a neighborhood $V$ of this subset where the same holonomy has trace contained in the interval $(\tau-\e, \tau+\e) \sub (-2,2)$, for some $\e>0$ suitably small. There, we can assume
\eq \label{equation:conjugationact} \prod [A_i, B_i] C_1 = \prod [A_i', B_i'] C_2' = \op{diag}(u, u^{-1}),\eeq
and divide by the residual action of $\C^*$ instead of $G$. Note that \eqref{equation:conjugationact} involves making a choice, since we 
could also have conjugated the above expression to $\op{diag}(u^{-1},u)$. The $\C^*$ action depends on this choice. However, if $\e$ is small enough, then this choice can be made consistently and one can then verify that the resulting $\C^*$ action is holomorphic; cf.\; \cite{goldman}. 

Furthermore, there is a $\C^*$ action on $V$ given by simultaneous conjugation on $A_1, \dots, A_g,$ $B_1, \dots, B_g$, $C_1$, and leaving $A_1', \dots, A_h', B_1', \dots, B_h', C_2'$ unchanged.  Let us call this action $\varphi$. Note that $\varphi$ preserves the Lagrangians $L_0, L_1$, but the action $\varphi$ on the $L_i$ is not free. However, we will exhibit open subsets $L_i^0 \sub L_i$ on which the action is free, and such that $L_0^0 \cap L_1^0$ is precisely the sub-locus of $L_0 \cap L_1$ consisting of Type II representations. 

Without loss of generality, by a suitable choice of the generators for $\pi_1$ of the surfaces under consideration, we can assume that $L_0$ is given by the equations
$$B_1=\dots=B_g=B_1'=\dots=B_h'=I.$$  
Then, the closure of $L_0$ inside $X^\tau(\Sigma)$ is the set
\begin{align*} \ov{L}_0&= \{(A_1,\dots,A_g,A_1',\dots,A_h', C_1, C_2') \in G^{2g+2} \mid C_1=C_2' \in \fk{C}^{\tau} \}/G \\
&= \{(A_1,\dots,A_g,A_1',\dots,A_h', C) \in G^{2g+1} \mid C= \op{diag}(u, u^{-1}) \} / \C^* \\
&= \bigl(\{(A_1,\dots,A_g,A_1',\dots,A_h') \in G^{2g}\}/ \C^*\bigr) \tms \op{diag}(u, u^{-1}),
\end{align*}  where $u+u^{-1}= \tau$ and the $\C^*$ action is by conjugation of the $\{A_i, A_i'\}_i$.  

Consider now the restriction of $\varphi$ to $\ov{L}_0$, which conjugates the $A_i$ while keeping the $A_i'$ fixed. Since both $\C^*$ actions commute, we have 
\begin{align*} \ov{L}_0 / \C^* &= \bigl( (\{(A_1,\dots,A_g,A_1',\dots,A_g') \in G^{2g}\}/ \C^*) \tms \op{diag}(u, u^{-1}) \} \bigr)/\C^* \\
&=  \bigl\{(A_1,\dots,A_g) \in G^g \bigr\}/ \C^* \tms \bigl\{(A_1',\dots,A_g') \in G^g\bigr\}/ \C^* \tms \op{diag}(u, u^{-1})\\
&=\bigl \{(A_1,\dots,A_g) \in G^g \bigr\}/ \C^*  \tms \bigl\{(A_1',\dots,A_g') \in G^g \bigr\}/ \C^*.\end{align*}

Let $X^1_{irr} \sub \{(A_1,\dots,A_g) \in G^g\}/ \C^* $ be the open locus of representations such that $(A_1,\dots,A_g,$ $ \op{diag}(u,u^{-1}))$ is irreducible and let $X^2_{irr}\sub \{ (A_1',\dots,A_g') \in G^g \} / \C^*$ be defined analogously. Let $L_0^0$ be the preimage of $X^1_{irr} \tms X^2_{irr}$ under the quotient map $\ov{L}_0 \to \ov{L}_0/\C^*$. 

\lem \label{lemma:codimensionspin} The inclusion $L_0^0 \sub L_0$ has codimension at least $4$. \elem

\pf Let $R_1 \sub \ov{L}_0$ be the set of representations such that $(A_1,\dots,A_g, \op{diag}(u, u^{-1}))$ is reducible and such that $(A_1',\dots,A_g', \op{diag}(u, u^{-1}))$ is irreducible. Let $R_2 \sub \ov{L}$ be defined analogously. Finally, let $R_{12}$ be the set of representations such that $(A_1,\dots,A_g,  \op{diag}(u, u^{-1}))$ and $(A_1',\dots,A_g', \op{diag}(u, u^{-1}))$ are both reducible. 

By an argument as in \cite[Lem.\ 2.6]{abou-man}, we see that $R_1$ has dimension at most $(2g+1)+3g=5g+1$. Similarly, $R_2$ has dimension at most $5g+1$, while $R_{12}$ has dimension at most $(2g+1)+(2g+1)= 4g+2$. Since $L_0$ has dimension $6g-1$, we conclude that the codimension is at least $g-2$. Since we are working with a Heegaard splitting of genus at least $6$, the lemma follows. \epf

\lem \label{lemma:uniquespinproduct} The manifold $X_{irr}^1 \tms X^2_{irr} \simeq L_0^0/\C^*$ admits a unique spin structure. \elem

\pf Observe that the $X^i_{irr}$ coincide with the varieties $X^{\tau}_{irr}(U_i)$ considered in \Cref{subsection:spinstructures} (note that the $U_i$ occurring in \Cref{subsection:spinstructures} are different from the $U_i$ considered in this section). Having made this observation, the existence of spin structure follows from \Cref{proposition: spinstructures}. The uniqueness follows from the proof of \Cref{proposition: spinstructures}, which shows that the $X^i_{irr}$ are simply-connected.  \epf

The above discussion goes through for $L_1$ in place of $L_0$. Note that, if we assume that $L_0$ is given by setting $B_i$ and $B_i'$ to the identity, we do not have a simple expression for $L_1$ in terms of the chosen matrices. However, by choosing different sets of generators for fundamental groups, we can put $L_1$ in the same form as we did for $L_0$ (at the expense of putting $L_0$ in a complicated form). The actions of $\C^*$ are independent of these choices of generators, so we can construct a subset $L_1^0 \subset L_1$ by the same procedure as above. 

Examining the definition of $L_i^0 \sub L_i$, we see that $L_0^0 \cap L_1^0 \sub L_0 \cap L_1$ is precisely the sub-locus of Type II representations. Hence $\mc{C}$ is a connected component of $L_0^0 \cap L_1^0$ and the action preserves $\mc{C}$ since it is continuous. 

\rmk \label{remark:c2notneeded} Although we have been assuming throughout \Cref{subsection:localsystems} that $\tau \in(-2,2)$ and $K_i \sub Y_i$ satisfy \ref{assumption:smoothchar}--\ref{assumption:repeatalexander}, we did not use \ref{assumption:0dimchar} in constructing the $\C^*$ action $\varphi$. \ermk

The next proposition relies crucially on the fact that the knots $K_i \sub Y_i$ satisfy \ref{assumption:0dimchar}.  

\prop \label{proposition:transitive} The action of $\varphi$ on $\mc{C}$ is transitive. \eprop

\pf It will be helpful to consider an alternative description of $\varphi$ in the spirit of \Cref{corollary:connectedsumdescription}.  Recall that $\mc{C}$ is the quotient of a Type II component $\tilde{\mc{C}} \sub R^{\tau}_{irr}(K-Y)$ by a transitive $\op{SL}(2,\C)$-action. Now $\tilde{\mc{C}}$ is a $\C^* \tms \C^*$ bundle over $R^{\tau}(S^2-q-p') \simeq \op{SL}(2,\C)/ \C^*$. Let $m \in \pi_1(S^2-q-p')$ be a generator and let $\s \in R^{\tau}(S^2-q-p')$ be the unique representation with the property that $\s(\mu)=\op{diag}(u, u^{-1})$. 

The elements of $\tilde{\mc{C}}$ are pairs $(\rho_1, \rho_2)$ of irreducible representations of $\pi_1(Y_1-K_1)$ and $\pi_2(Y_2-K_2)$ respectively, which agree on the generator $\mu \in \pi_1(S^2-q-p')$. By conjugating, we can ensure that $\rho(\mu)=\rho'(\mu)= \s(\mu)= \op{diag}(u, u^{-1})$. This has the effect of identifying all of the fibers with the fiber $F_{\s}$ over $\s \in R^{\tau}(S^2-q-p')$. This is precisely what we did in \eqref{equation:conjugationact}. 

The stabilizer of $ \op{diag}(u, u^{-1})$ is isomorphic to $\C^*$, and the quotient of $F_{\s}$ by the residual $\C^*$ action is then precisely $\mc{C} \sub X^{\tau}(Y-K)$.  However, we can also acts by conjugation on either factor of $F_{\s}= \C^* \tms \C^*$ while leaving the other factor fixed. One of these actions is $\varphi$, while the other corresponds to acting by simultaneous conjugation on  $A_1', \dots, A_h', B_1, \dots, B_h', C_2'$ while leaving $A_1, \dots, A_g, B_1, \dots, B_g, C_1$ unchanged. It follows from this description that $\varphi$ acts transitively on $\mc{C}= F_{\s}/\C^*$.   \epf 

Putting together the previous results, we arrive at the following corollary. 

\cor \label{corollary:finalspin} The map \eqref{equation:spinmap} preserves spin structures.  \ecor
\pf Consider the projections $\pi_i: L_i^0 \to L_i^0 /\C^*$. We note that $ TL_i^0= \pi_i^*(L_i^0/\C^*) \oplus \R^2$, where $\R^2$ is viewed as the Lie algebra of $\C^*$. According to \Cref{lemma:uniquespinproduct} applied to both $L_0^0$ and $L_1^0$, the quotients $L_i^0 / \C^*$ admit a unique spin structure which we call $s_i$ for concreteness. We can then endow $TL_i^0$ with a spin structure $S_i$, which is the direct sum of $\pi_i^*(s_i)$ and the trivial spin structure on the trivial $\R^2$-bundle.  

According to \Cref{lemma:codimensionspin} applied to $L_0$ and $L_1$, the inclusion $L_i^0 \sub L_i$ has complex codimension $g-2$. Since $\S$ has genus at least $6$, the inclusion induces an isomorphism on fundamental groups. It follows that the $L_i^0$ are simply-connected since the $L_i$ are simply-connected. Hence the spin structure $S_i$ coincides with the restriction of the spin structure coming from $L_i$. In particular, it is enough to prove that \eqref{equation:spinmap} takes ${S_0}|_{\mc{C}}$ to ${S_1}|_{\mc{C}}$. 

According to \Cref{proposition:transitive}, the $\C^*$ action is transitive on $\mc{C} \sub L_0^0 \cup L_1^0$. This means that \eqref{equation:spinmap} is the pullback of a map 
\eq \label{equation:spinpoint} (TL_0^0)_{[\rho]} \to (TL_1^0)_{[\rho]}, \eeq  where $[\rho] \in L_0^0/\C^* \cap L_1^0 / \C^*$. But \eqref{equation:spinpoint} necessarily maps ${s_0}_{[\rho]} \to {s_1}_{[\rho]}$ since spin structures are unique over a point. By definition of the $S_i$, it follows that \eqref{equation:spinmap} sends ${S_0}|_{\mc{C}}$ to ${S_1}|_{\mc{C}}$ since both \eqref{equation:spinmap} and the $S_i$ are obtained from a pullback.  \epf

Having shown that \eqref{equation:spinmap} preserves spin structures, it remains to argue that $|W^+|$ is a trivial local system. To prove this, we will appeal to similar considerations as in the proof of \Cref{corollary:finalspin} and that of Lemma 8.3 in \cite{abou-man}.

For $i=1,2$, let $N_i\mc{C} \sub TL_i^0|_{\mc{C}}$ be the normal bundle associated to the inclusion $\mc{C} \sub L_i^0$.  Let $T\mc{C}^{\perp}$ be the symplectic orthogonal complement to $T\mc{C} \sub TX^{\tau}_{irr}(\S)|_{\mc{C}}$. 
As argued in \cite[Sec.\ 6]{abou-man}, there is a natural isomorphism from the direct sum $N_0 \mc{C} \oplus N_1\mc{C}$ to the symplectic normal bundle $T\mc{C}^{\perp}/ T\mc{C}$. In particular, $N_0 \mc{C} \oplus N_1\mc{C}$ is a symplectic vector bundle and the $N_i\mc{C}$ are transverse Lagrangian sub-bundles.  
 
A \emph{polarization} of $N_0 \mc{C} \oplus N_1\mc{C}$ can be thought of as a Lagrangian sub-bundle which is fiberwise transverse to both summands. As explained in \cite[Sec.\ 6]{abou-man}, a choice of such a polarization allows one to view $N_1\mc{C}$ as a sub-bundle of $N_0\mc{C} \oplus N^*_0\mc{C}$, which can be described by the graph of a quadratic form $q$ on $N_0\mc{C}$.  The bundle $|W^+| \sub N_0Q$ is a maximal sub-bundle on which the real part of $q$ is positive, which depends only on $q$ up to isomorphism; see \cite[Sec.\ 6]{abou-man}.

As in the proof of \Cref{corollary:finalspin}, one notes that the $N_i \mc{C}$ are equivariant under the $\C^*$ action $\varphi$, in the sense that they are isomorphic to pullbacks of vector bundles over a point $[\rho] \in L_0^0/\C^* \cap L_1^0/\C^*$. As argued in the proof of Lemma 8.3 in \cite{abou-man}, we can therefore construct a $\C^*$-equivariant polarization of $N_0 \mc{C} \oplus N_1\mc{C}$, simply by pulling back a polarization over $\rho$.  We then find that $q$ has constant coefficients with respect to a $\C^*$-equivariant trivialization of $N_0 \mc{C}$, which implies that $|W^+|$ is trivial. 

We conclude that $P^{\bullet}_{\tau}(K)$ is a trivial local system. Along with \Cref{corollary:localsytemcomp}, this implies the following result. 

\thm \label{theorem:connectedsum} Suppose that \ref{assumption:smoothchar}--\ref{assumption:repeatalexander} are satisfied. Then $P^{\bullet}_{\tau}(K)$ is a trivial local system supported on $m_1+m_2$ isolated points and $m_1m_2$ copies of $\C^*$. It follows that $\HP^{-1}_{\tau}(K) = \Z^{m_1m_2}$, $\HP^0_{\tau}(K)= \Z^{m_1+m_2+m_1m_2}$ and $\HP^k_{\tau}(K)= 0$ for all $k \notin \{-1,0\}$.  \ethm

\subsection{Euler characteristics}
We now discuss the behavior of the Euler characteristic of $\HP^*_{\tau}(K)$, the  $\op{SL}(2, \C)$  Casson-Lin invariant $\chi_{\tau}(K)$ defined in the Introduction. In this section we assume that the knots $K_i$ satisfy \ref{assumption:smoothchar} and \ref{assumption:repeatalexander}, but they do not have to satisfy  \ref{assumption:0dimchar}. The following theorem asserts that the  $\op{SL}(2, \C)$ Casson-Lin invariant is additive for such knots.

\thm \label{theorem:euler} For $i =1,2$, suppose that $K_i \sub Y_i$ and $\tau \in (-2,2)$ satisfy \ref{assumption:smoothchar} and \ref{assumption:repeatalexander}. Then $\chi_{\tau}(K_1\# K_2)= \chi_{\tau}(K_1) + \chi_{\tau}(K_2)$.   \ethm 

\pf According to \Cref{proposition:schemesmoothsum}, $\mathscr{X}^{\tau}_{irr}(K_1 \# K_2)$ is smooth as a scheme.  By \Cref{proposition:localsystem}, it follows that our perverse sheaf is a rank one local system on $X^{\tau}_{irr}(K_1\# K_2)$.  On each component $C$ of $X^{\tau}_{irr}(K_1\# K_2)$ of complex dimension $k$, this local system is supported in degree $-k$. Note that the Euler characteristic of a cochain complex with twisted coefficients in a local system depends only on the rank, and not on the twisting. Thus, the contribution to $\chi_{\tau}(K_1 \# K_2)$ from the component $C$ equals $(-1)^k$ times the topological Euler characteristic of $C$. Since the $K_i$ satisfy \ref{assumption:smoothchar}, the same argument shows that $\chi_{\tau}(K_i)$ is a signed count of the topological Euler characteristics of the components of $X^{\tau}_{irr}(K_i)$.

\Cref{proposition:schemesmoothsum} shows that $X^{\tau}_{irr}(K_1\# K_2)$ consists exclusively of Type I and Type II components. According to \Cref{lemma:type1characterize}, the locus of Type I components is topologically the disjoint union of $X^{\tau}_{irr}(K_1)$ and $X^{\tau}_{irr}(K_2)$.  

The theorem now follows from the fact that the Type II components of $X^{\tau}_{irr}(K_1\# K_2)$ have vanishing topological Euler characteristic. This is a consequence of the fact that the Type II components admit a free $\C^*$ action, and are therefore $\C^*$-bundles which always have trivial Euler characteristic; cf. \Cref{subsection:localsystems} and \Cref{remark:c2notneeded}. \epf 

By repeatedly applying \Cref{theorem:euler}, we can in fact compute the Euler characteristic of arbitrary connected sums of knots. 

\cor \label{corollary:euler2} For $j \in \{1,2,\dots,n\}$, let $K_j \sub Y_j$ be a knot in an integral homology sphere. Suppose that $K_j \sub Y_j$ and $\tau \in (-2,2)$ satisfy \ref{assumption:smoothchar} and \ref{assumption:repeatalexander}. Then $\chi_{\tau}(\#_{j=1}^n K_j) = \sum_{j=1}^n \chi_{\tau}(K_j)$. \ecor

\pf Using \Cref{proposition:schemesmoothsum}, one sees that the connected sum of two knots satisfying \ref{assumption:smoothchar} also satisfies \ref{assumption:smoothchar}.  Similarly, the connected sum of two knots satisfying \ref{assumption:repeatalexander} also satisfies \ref{assumption:repeatalexander}, since the Alexander polynomial is multiplicative under connected sums. The corollary follows. \epf 

\subsection{Some computations} \label{subsection:connectedcomp}

The character varieties of the trefoil and figure-eight were described in \Cref{theorem:connectedsum}. For a trefoil $K'= 3_1 \sub S^3$, one computes that \eq X^{\tau}(S^3-K')= \{ (x, y) \mid (y-2) (\tau^2- y-1) =0\}. \eeq  Away from the point $\tau = \pm \sqrt{3}$, this is a smooth, zero-dimensional scheme which thus satisfies \ref{assumption:smoothchar}. Using the fact that the Alexander polynomial of the trefoil is $\Delta(t)=t^2-t+1$, one can check that these are precisely the points ruled out by \ref{assumption:repeatalexander}. 

Let us write $K= 3_1 \# 3_1$ for a connected sum of two trefoils. It follows from \Cref{corollary:connectedsumdescription} that $X^{\tau}_{irr}(3_1 \# 3_1)$ is a disjoint union of one copy of $\C^*$ and two isolated points. According to \Cref{theorem:connectedsum}, we have for $\tau \in (-2,2) \setminus \{\pm \sqrt{3}\}$ that $\HP^*_{\tau}(3_1 \# 3_1)$ is supported in degrees $-1,0$ with $$\HP^{-1}_{\tau}(3_1 \# 3_1)= \Z, \ \  \ \HP^0_{\tau}(3_1 \# 3_1)= \Z^3,.$$  

For the figure-eight $K''= 4_1 \sub S^3$, we compute similarly that \eq X^{\tau}(S^3-K'')= \{(x,y) \mid (y-2)(y^2- (\tau^2-1)y+ \tau^2-1) =0 \}, \eeq which is smooth and zero-dimensional for all $\tau \in (-2,2)$. Writing $K= 4_1 \# 4_1$ for a connected sum of two figure-eight knots, it again follows from \Cref{corollary:connectedsumdescription} that $X^{\tau}_{irr}(4_1 \# 4_1)$ is a disjoint union of four copies of $\C^*$ and four isolated points. 
We find for $\tau \in (-2,2)$ that $\HP^*_{\tau}(4_1 \# 4_1)$ is supported in degrees $-1,0$, with $$\HP^{-1}_{\tau}(4_1 \# 4_1)= \Z^4, \ \  \ \HP^0_{\tau}(4_1 \# 4_1)= \Z^8.$$ 

We can similarly compute for $\tau \in (-2,2) \setminus \{\pm \sqrt{3}\}$ that $\HP^*_{\tau}(3_1 \# 4_1)$ is supported in degrees $-1,0$, with $\HP^{-1}_{\tau}(3_1 \# 4_1)= \Z^2$ and $\HP^0_{\tau}(4_1 \# 4_1)= \Z^5$. 

Using \Cref{corollary:euler2}, we can compute the Euler characteristic of our invariant for arbitrary connected sums of trefoils and figure-eights. Letting $K= (\#_{n} 3_1) \# (\#_{m} 4_1)$ denote a connected sum of $n$ trefoils and $m$ figure-eights, we find that $$\chi_{\tau}(K)= n+ 2m,$$ provided that $\tau \in (-2,2) \setminus \{\pm \sqrt{3}\}$. 

We would like to extend the above computations to connected sums of the small knots and slim knots which were considered in \Cref{section:1dcomputations}. Thanks to the following lemma, we can show that these knots satisfy \ref{assumption:smoothchar}--\ref{assumption:repeatalexander} for all but finitely many $\tau \in (-2,2)$. 

\lem \label{lemma:moregeneric} Suppose that $K \sub Y$ satisfies \ref{assumption:dim1}--\ref{assumption:reduced} in \Cref{subsection:assumptions}. Then $\mathscr{X}^{\tau}_{irr}(Y-K)$ is smooth and zero-dimensional for all but finitely many $\tau \in (-2,2)$. \elem

\pf Let us first note that \ref{assumption:first}--\ref{assumption:last} hold for all values of $\tau$ contained in some cofinite subset $V \sub (-2,2)$ according to \Cref{proposition:generic1}. It then follows by \Cref{proposition:intersectionhyperplane} that $\mathscr{X}^{\tau}_{irr}(\G) = H_{\tau} \cap X_0(\G)$, where $\G= \pi_1(K, x_0)$. But by \Cref{proposition: subfinal} and \Cref{lemma:zerodim}, we know that $H_{\tau} \cap X_0(\G) = \op{Spec} R_{\tau}$ where \eq R_{\tau}= \frac{\C[\e_1]}{\e_1^{n_1}} \tms\dots\tms \frac{\C[\e_k]}{\e_k^{n_k}}. \eeq 
Finally, it follows from \Cref{lemma:reducedtechnical} that, after restricting to a possibly smaller cofinite subset $V' \sub V \sub (-2,2)$, we have that $n_i=1$ for $i=1,\dots,k$. This proves the claim.  \epf

We saw in \Cref{subsection:dim1prelim} that \ref{assumption:dim1}--\ref{assumption:reduced} hold for generic $\tau$ for a wide class of small knots and slim knots, including all two-bridge knots, torus knots, and $(-2,3, 2n+1)$ pretzel knots where $n \neq 0,1,2$ and $2n+1$ is not divisible by $3$. The above lemma now implies that these knots also satisfy \ref{assumption:smoothchar}--\ref{assumption:repeatalexander} for all but finitely many $\tau \in (-2,2)$.  By combining \Cref{lemma:moregeneric} with \Cref{theorem:apolylink}, \Cref{corollary:knotpolyrelation} and \Cref{theorem:connectedsum}, we obtain Theorem~\ref{thm:finalconnectedsum}, which we restate here for convenience. 

\cor[Theorem~\ref{thm:finalconnectedsum}] \label{corollary:finalconnectedsum} For $i=1,2$, suppose that $K_i \sub S^3$ is a two-bridge knot, a torus knot, or a $(-2,3, 2n+1)$ pretzel knot where $n \neq 0,1,2$ and $2n+1$ is not divisible by $3$.  Let $K= K_1 \# K_2$. Then, for all but finitely many $\tau \in (-2,2)$, we have that $\HP^*_{\tau}(K)$ is supported in degrees $-1,0$ with 
$$\HP^{-1}_{\tau}(K)= \Z^{ (\op{deg}_{\l} \widehat{A}(K_1)) \cdot (\deg_{\l} \widehat{A}(K_2))}$$ and $$\HP^{0}_{\tau}(K)= \Z^{(\deg_{\l} \widehat{A}(K_1)) +(\deg_{\l} \widehat{A}(K_2))+ (\op{deg}_{\l} \widehat{A}(K_1)) \cdot (\deg_{\l} \widehat{A}(K_2))}.$$ \ecor

We can compute $P^{\bullet}_{\tau}(K)$ explicitly when $K$ is a connected sum of the slim knots considered in \Cref{subsection:slimcomputations} whose $\widehat{A}$-polynomial can be determined exactly. We limit ourselves to the following example.

\ex We saw in \Cref{example:slimknotscomp} that $\op{deg}_{\l} \widehat{A}(8_1)= \Z^6$. It follows from \Cref{corollary:finalconnectedsum} that, for all but finitely many $\tau \in (-2,2)$, we have that $\HP^*_{\tau}( 8_1 \# 8_1)$ is supported in degrees $-1,0$ with $\HP^{-1}_{\tau}(8_1\# 8_1)= \Z^{36}$ and $\HP^0_{\tau}(8_1\# 8_1)= \Z^{48}$.  \eex

Finally, \Cref{corollary:euler2} allows us to compute the Euler characteristic of our invariant for arbitrary connected sums of the knots considered in \Cref{corollary:finalconnectedsum}.  More precisely, if we let $K_j \sub Y_j$ be one of the knots considered in \Cref{corollary:finalconnectedsum} for $j =1,2,\dots,n$, we find that 
$$\chi_{\tau}(\#_{j=1}^n K_j) = \sum_{j=1}^n \deg_{\l} \widehat{A}(K_j),$$ for all but finitely many $\tau \in (-2,2)$.

\section{Appendix I: Parabolic Higgs bundles}   \label{section:parabolichiggs}

The purpose of this appendix is to prove that the character variety $X^{\tau}_{irr}(\S)$ considered in \Cref{proposition:stabilization} is connected and simply-connected. We do this by exploiting a version of the non-abelian Hodge theory correspondence due to Simpson, which establishes an isomorphism between $X^{\tau}_{irr}(\S)$ and a certain moduli space of \emph{parabolic} Higgs bundles; see \Cref{mainequivalence}. 

The topology of parabolic Higgs bundles can be probed using Morse theory. In particular, Boden and Yokagawa showed in \cite{boden-yok} that the relevant moduli spaces are connected and simply-connected, under a technical hypothesis involving ``genericity of weights". Unfortunately, this hypothesis is violated in the case of interest to us. Our analysis, which draws significantly from the paper of Thaddeus \cite{thaddeus}, will show that the desired conclusion is nonetheless true, provided that $\S$ has genus at least $2$.

We begin with a brief introduction to the theory of parabolic Higgs bundles. More details can be found in \cite{boden-yok}, \cite{thaddeus} and \cite{mondello}. We particularly recommend the last reference to readers with no prior exposure to this theory. This appendix is self-contained and completely independent from the rest of the paper. 

\subsection{General definitions} This section collects some important definitions from the theory of parabolic Higgs bundles. We warn the reader that there are many variants of these definitions in the literature. 

Let $\mc{S}$ be a Riemann surface of genus $g.$  Fix a set of distinct points $p_1, \dots, p_n \in \mc{S}$ and define the effective divisor $D= p_1 + \dots + p_n.$ 

\defi \label{definition: parabolicvector}A \emph{parabolic vector bundle} $E_{*}$ of rank $r$ on $(\mc{S}, D)$ consists of the following data:
\begin{enumerate}
\item[(i)] a holomorphic vector bundle $E$ of rank $r$,
\item[(ii)] for each point $p_i,$ a complete flag $E_{p_i}= E_{1}(p_i) \supset E_{2}(p_i) \supset \dots \supset E_{r}(p_i) \supset 0,$ 
\item[(iii)] for each point $p_i,$ a choice of weights $0 \leq \a_{1}(p_i) < \a_{2}(p_i)< \dots < \a_{r}(p_i)<1.$
\end{enumerate}
\edefi

There are many conventions in the literature regarding the indexing of the flags and weights. \Cref{definition: parabolicvector} is identical to \cite[Definition 2.1]{boden-yok}, except that we require the flags to be complete for notational simplicity since this is the only case relevant to us.

Let $E_*$ and $F_*$ be parabolic vector bundles on $(\mc{S}, D)$ with weights $\a_j^1(p_i)$ and $\a_j^2(p_i)$ respectively. A morphism $\phi: E \to F$ of vector bundles is said to be a parabolic morphism if $\phi(E_{j}(p_i)) \sub F_{k+1}(p_i)$ whenever $\a^1_{j}(p_i) > \a^2_{k}(p_i).$  It is said to be strongly parabolic if this holds whenever $\a^1_{j}(p_i) \geq \a^2_{k}(p_i).$

\defi \label{definition: parabolichiggs} A \emph{parabolic Higgs bundle} consists in the data of a parabolic vector bundle $E_*$ and a strongly parabolic morphism $\Phi: E \to E \otimes K(D).$ This morphism is usually called the \textit{Higgs field}. We will usually use boldface letters $\bf{E}= (E_*, \Phi)$ to denote parabolic Higgs bundles. \edefi

There are natural notions of sub-bundle and quotient bundle in the categories of parabolic vector bundles and parabolic Higgs bundles; see \cite[p.\ 3]{thaddeus}. Given a parabolic vector bundle $E_*$ on $(\mc{S}, D)$, we say that $F_* \sub E_*$ is a parabolic sub-bundle if $F_*$ is a vector sub-bundle of $E_*$ where the weights and filtration of $F_*$ are induced from the corresponding data for $E_*$. More precisely, given a point $p_i \in D$, we consider the sequence of vector spaces \eq \label{equation:flag} E_{p_i} \cap F_{p_i}  \supseteq E_2(p_i) \cap F_{p_i} \supseteq \dots \supseteq E_r(p_i) \cap F_{p_i}. \eeq We obtain the flag of $F_*$ at $p_i$ by discarding any repetitions in \eqref{equation:flag}. The weights of $F_*$ at $p_i$ are obtained from those of $E_*$ by discarding $\a_j(p_i)$ if $ E_{j-1}(p_i) \cap F_i = E_j(p_i) \cap F_i$.  The filtration and weights of the quotient $E_*/F_*$ are obtained similarly by considering the sequence $$E_{p_i} / F_{p_i}  \supseteq E_2(p_i) \cap (F_{p_i} \cap E_2(p_i)) \supseteq \dots \supseteq E_r(p_i) /(F_{p_i} \cap E_r(p_i)).$$ 

A Higgs sub-bundle $\bf{F} \sub \bf{E}$ is an ordinary parabolic sub-bundle $F_* \sub E_*$ which is $\Phi$-invariant, where $\bf{E}= (E_*, \Phi)$. In this case, $\Phi$ passes to the quotient $E_*/F_*$, which thus also carries the structure of a Higgs bundle.

The \emph{parabolic degree} of a parabolic vector bundle 
$E_*$ is defined as follows \eqs \op{pdeg}(E_*):= \op{deg}(E) + \sum_{p \in D} \sum_{i=1}^r \a_i. \eeqs

The \emph{slope} is then \eqs \mu(E_*)= \frac{ \op{pdeg} (E_*)} { \op{rank}(E)}. \eeqs  We say that a parabolic vector bundle $E_*$ is stable (resp. semistable) if there does not exist a proper sub-bundle $E'_* \sub E_*$ such that $\mu(E'_*) > \mu(E)$ (resp. $ \geq$).  We say that a parabolic Higgs bundle is stable (resp. semistable) if there does not exist a proper $\Phi$-invariant sub-bundle $E'_* \sub E_*$ (i.e. a proper Higgs sub-bundle) such that $\mu(E'_*) > \mu(E)$ (resp. $ \geq$). 

Given a line bundle $L$ and real numbers $0\leq \a(p_i)$ for $i=1,\dots,n$, we define the parabolic bundle $L(\a(p_1) p_1+\dots+\a(p_n) p_n)$ to be the underlying line bundle $L([\a(p_1)] p_1+\dots+[\a(p_n)] p_n)$ endowed with the weight $\a(p_i)-[\a(p_i)]$ at the point $p_i$ and the obvious flag structure. Here $[-]$ denotes the integer part of a non-negative real number. 

The \emph{determinant} of a parabolic vector bundle $E_*$ is defined as 
\eqs \op{det}(E_*)= \op{det}(E) \otimes \lt( \bigotimes_1^n \mc{O}_X ( \sum_{p \in D} (\sum_{i=1}^r \a_j(p)) p ) \rt). \eeqs

Observe that we have $\op{det}(E_*)= E_*$ if $E_*$ is a parabolic line bundle, because in that case all $\a_j(p)$ are zero.

All of the above notions have sheaf-theoretic analogs. We refer the reader to \cite[Sec.\ 2.2]{boden-yok} for the relevant definitions.

\subsection{The main equivalence}

We now specialize to the case of a Riemann surface $\mc{S}$ of genus $g$ with two distinguished points $p, q \in \mc{S}.$  We set $D= p+q$ as above. Let $\bm{w}$ denote the data of weights $0< \a_1(p)<\a_2(p)<1$ and $0< \a_1(q)< \a_2(q)<1.$

We introduce the set $\op{Higgs}^s(\mc{S}, \bm{w}, 2, \mc{O}_{\mc{S}})$ whose elements are isomorphism classes of parabolic Higgs bundles $(E_*, \Phi)$ on $(\mc{S}, D)$ satisfying the following conditions: 
\begin{itemize}
\item $(E_*, \Phi)$ is stable, 
\item $\op{rank} E =2,$
\item $E_*$ has weights $\bm{w},$
\item there is an isomorphism $\op{det}(E_*) \to \mc{O}_{\mc{S}}, $
\item $\op{Tr} \Phi=0.$
\end{itemize}

We define $\op{Higgs}^{ss}(\mc{S}, \bm{w}, 2, \mc{O}_{\mc{S}})$ in the analogous way, where we now only require $(E_*, \Phi)$ to be semistable.

We also record the following important fact.

\prop[see Thm.\ 1.6 of \cite{konno}] \label{proposition:smoothhiggs} The set $\op{Higgs}^s(\mc{S}, \bm{w}, 2, \mc{O}_{\mc{S}})$ has the structure of a smooth (in general non-compact) complex manifold. \eprop

Let $\S$ be a Riemann surface with two boundary components which arises from a Heegaard splitting $(\S, U_0, U_1)$ as in \Cref{subsection:symplecticgeom}. Let $\fk{c}_1= \op{Conj}([c_p])$ and $\fk{c}_2= \op{Conj}([c_q])$ be defined as in \ref{equation: diagram1}.  Then $(\G; \fk{c}_1, \fk{c}_2)$ is an object of $\bf{Gp}^+$. Given $\tau \in (-2, 2)$, let $X^{\tau}_{irr}(\S)= X^{\tau}_{irr}(\G)$ be the associated character variety, defined as in \Cref{section:relativerepschemes}. 

Let $\mc{S}$ be a closed Riemann surface obtained from $\S$ by gluing disks to its two boundary components. Let $p$ and $q$ be distinguished points in the interior of these disks. 

\thm[Simpson \cite{simpson}; see also Theorem 4.12 in \cite{mondello}]  \label{mainequivalence} Fix $\a \in (0,1/2)$ and let $\tau= 2 \cos (2\pi \a).$  Let the weights $\bf{w}$ be defined by letting $\a_1(p)= \a_1(q)=\a$ and $\a_2(p)= \a_2(q)= 1-\a.$ Then there is a diffeomorphism
\eq X_{irr}^{\tau}({\S}) \simeq \op{Higgs}^s(\mc{S}, \bm{w}, 2, \mc{O}_{\mc{S}}). \eeq
\ethm 

\rmk The left-hand side of Simpson's original correspondence involves the character variety of the punctured Riemann surface $\mc{S}-p-q$, as opposed to the Riemann surface with boundary $\S$. Since the character variety only depends on the underlying fundamental group, this distinction is immaterial. \ermk

\subsection{Variations of the weights}  It will be convenient for our purposes to consider the set $\mc{P}= (0, 1/2) \tms (0, 1/2)$ with coordinates $(\a, \b),$ which we will think of as a parameter space.  We will refer to the subset $\{\a=\b \} \sub \mc{P}$ as the \emph{wall} and to the subsets $\Delta^+ = \{ \a> \b\}$ and $\Delta^- = \{ \b > \a\}$ as the \emph{chambers}.

A point $(\a, \b) \in \mc{P}$ determines a choice of weights $\bm{w}( \a, \b)$ on $(\mc{S}, D)$ by setting $\a_1(p) = \a$, $\a_2(p)= 1-\a$, and $\a_1(q)= \b$, $\a_2(q) = 1-\b$.  We would like to study how $\op{Higgs}^{ss}(\mc{S}, \bm{w}(\a, \b), 2, \mc{O}_{\mc{S}})$ changes as we vary $(\a, \b) \in \mc{P}.$

To this end, it will be necessary to adopt a slight change of perspective.  According to \Cref{definition: parabolicvector} and \Cref{definition: parabolichiggs}, we are viewing the weights of a parabolic Higgs bundle $\bf{E}= (E_*, \Phi)$ as being part of the data of the underlying parabolic vector bundle $E_*.$  Instead, let us now view the choice of weights as an independent piece of data, i.e. a parabolic Higgs bundle should now be viewed as a triple $(E_*, \Phi, \bm{w}).$

Observe that as we vary $(\a, \b) \in \mc{P},$ a parabolic Higgs bundle $(E_*, \Phi, \bm{w}(\a, \b))$ remains of constant rank and parabolic degree. The determinant $\op{det}(E_*)$ and the trace $\op{Tr} \Phi$ are also constant.  However, the parabolic degree and slope of sub-bundles $E'_* \sub E_*$ can change.  

\lem \label{lemma: stablechambers} Suppose that $\a \neq \b$ and let $\bm{w}= \bm{w}(\a, \b).$ Then $$\op{Higgs}^{ss}(\mc{S}, \bm{w}, 2, \mc{O}_{\mc{S}})= \op{Higgs}^s(\mc{S}, \bm{w}, 2, \mc{O}_{\mc{S}}).$$ \elem

\pf If $E_* \in \op{Higgs}^{ss}(\mc{S}, \bm{w}, 2, \mc{O}_{\mc{S}})$ is semistable but not stable, then there exists a $\Phi$-invariant sub-bundle $E'_* \sub E$ such that $\mu(E'_*)= \mu(E)=0.$ Hence $\op{pdeg}(E'_*)=0.$  But this means that either $\a_1(p) + \a_2(q) \in \Z$ or $\a_2(p)+\a_1(q) \in \Z.$ This is possible only if $\a_1(p)= \a_2(q)$, i.e. $\a=\b.$ \epf

By combining \Cref{proposition:smoothhiggs} and \Cref{lemma: stablechambers}, it follows that $\op{Higgs}^{ss}(\mc{S}, \bm{w}, 2, \mc{O}_{\mc{S}})$ has the structure of a smooth complex manifold if we let $\bm{w}= \bm{w}(\a, \b)$ for $(\a, \b) \in \Delta^+ \cup \Delta^-$. The following important result describes the topology of this manifold. 




\thm[Boden and Yokagawa, see p.\ 2 and Thm.\ 4.2 in \cite{boden-yok}] \label{theorem: simplyconnected} Suppose that $(\a, \b) \in \Delta^+ \cup \Delta^-$ and let $\bm{w}= \bm{w}(\a, \b).$ Then $\op{Higgs}^{ss}(\mc{S}, \bm{w}, 2, \mc{O}_{\mc{S}})$ has the structure of a smooth complex manifold of complex dimension $6g-2$. If $g\geq 1$, it is connected and simply-connected.   \ethm

Fix $\a \in (0,1/2).$ For $0<\e< \a,$ define $\g: (0, \e) \to (0, 1/2)^2$ by setting $\g(t)= (\a, \a+t).$   Let $\bm{w}(t):= \bm{w}(\g(t)).$ Observe that $\bm{w}(t) \in \Delta^+ \cup \Delta^-$ when $t>0,$ and so it follows by \Cref{lemma: stablechambers} that $\op{Higgs}^{ss}(\mc{S}, \bm{w}, 2, \mc{O}_{\mc{S}})= \op{Higgs}^s(\mc{S}, \bm{w}, 2, \mc{O}_{\mc{S}}).$ Let us analyze what happens as $t \to 0.$

\lem \label{lemma:split} Suppose that $(E_*, \Phi, \bm{w}(t))$ is stable for $t > 0$ and unstable for $t=0.$ Then there exists a unique destabilizing $\Phi$-invariant sub-bundle $E^+_* \sub E_*.$ Thus, $E_*$ can be uniquely presented as a non-split extension of parabolic Higgs bundles $0 \to \textup{\bf{E}}^+ \to \textup{\bf{E}} \to \textup{\bf{E}}^- \to 0.$ 
\elem 

\pf Since $E_*$ is unstable, there exists a $\Phi$-invariant sub-bundle $E^+_* \sub E$ such that $\mu(E^+_*)= \mu(E_*)=0.$  To verify uniqueness, suppose for contradiction that $\tilde{E}^+_* \sub E$ is also destabilizing. Then there is a morphism $\tilde{E}^+_* \to E \to E^-_*:= E/ E^+_*$ of $\Phi$-invariant bundles. But since $\tilde{E}^+_*$ and $E^-_*$ are both tautologically stable (since they are of rank $1$) and of equal slope, this morphism is either zero or an isomorphism; see (3.3) in \cite{thaddeus}.  It cannot be the zero map since $\tilde{E}^+_* \neq E^+_*$ by assumption.  It cannot be an isomorphism since $\tilde{E}^+_*$ and $E^-_*$ have different parabolic structures. This gives the desired contradiction. 

To see that the extension is non-split, suppose for contradiction that $\bf{E}= \bf{E}^+ \oplus \bf{E}^-$. Then $0= \mu(E_*) = \mu(E^+_*) + \mu(E^-_*)$ for all $0 \leq t <\e $. In particular, for $0<t<\e$, we find that $\mu(E^-_*) = - \mu(E^+_*)> 0 = \mu(E_*)$, which contradicts the stability of $\bf{E}$. \epf

\subsection{Dimension count}
Throughout this section, $\a \in (0,1/2)$ is an arbitrary fixed parameter.  We will be considering parabolic Higgs bundles of rank $1$ on $(\mc{S}, D).$  A choice of weights therefore consists in a choice of two real numbers $0< \a_1(p)< 1$ and $0< \a_1(q)<1.$  Let $\bm{w}_{1,t}$ be weights obtained by setting $\a_1(p)=\a$ and $\a_1(q)= 1-\a-t$ and let $\bm{w}_{2,t}$ be obtained by setting $\a_2(p)=1-\a$ and $\a_2(q)=\a+t.$ 

For $t>0$, let $\mc{E}$ be the set of non-split extensions $ 0 \to \textup{\bf{E}}^+ \to \textup{\bf{E}} \to \textup{\bf{E}}^- \to 0$ of parabolic Higgs bundles such that $\textup{\bf{E}}= (E_*, \Phi, \bm{w}(t)) \in \op{Higgs}^{ss}(\mc{S}, \bm{w}(t), 2, \mc{O}_{\mc{S}})$ and such that $\bf{E}^+= (E^+_*, \Phi^+)$ has rank $1$.  Observe that $\bf{E}^+$ and $\bf{E}^-$ are related by the following conditions: 

\begin{itemize}
\item[(i)] $(E^+_*, \Phi^+)$ and $(E^-_*, \Phi^-)$ have rank $1$, 
\item[(ii)] $E^+_* \simeq (E^-_*)^{-1}$ and $\op{Tr} \Phi^+ = - \op{Tr} \Phi^-.$
\end{itemize}

We let $\mc{X}$ be the set of pairs $\lt( (E^+_*, \Phi^+, \bm{w}_{1,t}), (E^-_*, \Phi^-, \bm{w}_{2,t}) \rt)$ satisfying the conditions (i) and (ii) above.  There is a natural surjection 
\eq
\mc{E} \to \mc{X},
\eeq whose fiber $\mc{E}_p$ over a point $p = ((E^+_*, \Phi^+, \bm{w}_{1,t}),  (E^-_*, \Phi^-, \bm{w}_{2,t}) )$ is simply the set of non-split extensions of $(E^-_*, \Phi^-, \bm{w}_{2,t}) $ by $(E^+_*, \Phi^+, \bm{w}_{1,t}).$

Similarly, let $\tilde{\mc{E}}$ be the set of non-split extensions $0 \to \bf{E}^- \to \bf{E} \to \bf{E}^+ \to 0.$ 

The goal of this section is to prove the following claim.

\thm \label{theorem:dimensioncount} The complex dimension of $\mc{E}$ and $\tilde{\mc{E}}$ is at most $4g-1.$ \ethm

\cor \label{corollary:finalsimplyconnected} For $g \geq 2,$ the moduli space $\op{Higgs}^s(\mc{S}, \bm{w}(\a, \a), 2, \mc{O}_{\mc{S}})$ is connected and simply-connected for any $\a \in (0,1/2).$ \ecor

\pf[Proof of \Cref{corollary:finalsimplyconnected}]  
	Recall from the previous section that $\bm{w}(t)= \bm{w}(\a, \a+ t).$  By \Cref{theorem: simplyconnected}, we know that $\op{Higgs}^s(\mc{S}, \bm{w}(t), 2, \mc{O}_{\mc{S}})$ is connected 	and simply-connected and of complex dimension $6g-2$ when $t>0$ and $g \geq 1$.

	Given $0< t_0<\e,$ there is a natural map 
	$\op{Higgs}^s(\mc{S}, \bm{w}(t_0), 2, \mc{O}_{\mc{S}}) \to \op{Higgs}^{ss} (\mc{S}, \bm{w}(0), 2, \mc{O}_{\mc{S}})$ 
	obtained by sending $(E_*, \Phi, \bm{w}(t_0)) \mapsto (E_*, \Phi, \bm{w}(0)).$  
	
	Since stability is an open condition under varying the weights, this image of this map contains $\op{Higgs}^{s} (\mc{S}, \bm{w}(0), 2, \mc{O}_{\mc{S}})$. 
	According to \Cref{lemma:split}, if one restricts the map to the complement of the locus $\mc{E} \cup \tilde{\mc{E}}$ consisting of non-split extensions, then it is an isomorphism onto $\op{Higgs}^{s} (\mc{S}, \bm{w}(0), 2, \mc{O}_{\mc{S}})= \op{Higgs}^{s} (\mc{S}, \bm{w}(\a,\a), 2, \mc{O}_{\mc{S}})$.  But it follows by the theorem that the complement of $\mc{E} \cup \tilde{\mc{E}}$ has complex codimension $ 2g-1.$  If $g \geq 2,$ then $2g-1\geq 3$ so the complement remains connected and simply-connected, which gives the desired result. \epf

It remains to prove \Cref{theorem:dimensioncount}.  For simplicity, we will only consider the case of $\mc{E}.$ It will be apparent that the argument works equally well for $\tilde{\mc{E}}.$

\lem \label{lemma:dimension1} The dimension of $\mc{X}$ is exactly $2g.$ \elem 
\pf Let $\op{Higgs}(\mc{S}, \bm{w}_{1,t})$ be the space of parabolic Higgs bundles of rank $1$ and weights $\bm{w}_{1,t}.$ According to \cite[p.\ 2]{boden-yok}, this space has complex dimension $2(g-1)+2=2g.$ However, it follows by condition (ii) in the definition of $\mathcal{X}$ that the projection $\mc{X} \to \op{Higgs}(\mc{S}, \bm{w}_{1,t})$ taking a pair $((E^+, \Phi^+, \bm{w}_{1,t}), (E^-, \Phi^-, \bm{w}_{2,t}))$ to the first factor is an isomorphism. The claim follows. \epf

Let $\bf{E}= (E_*, \Phi)$ and $\bf{F}= (F_*, \Psi)$ be parabolic Higgs bundles on $(\mc{S}, D).$ Following \cite[p.\ 7]{ggm}, we define a complex of sheaves 
\begin{align*}
C^{\bullet}(\bf{E}, \bf{F}) : \op{ParHom}(E, F) &\to \op{SParHom}(E, F) \otimes K(D) \\
f &\mapsto (f \otimes 1) \Phi - \Psi f,
\end{align*}
where $\op{ParHom}(E, F)$ and $\op{SParHom}(E, F)$ denote the spaces of parabolic, resp. strictly parabolic, morphisms between $E$ and $F$.

We write $C^{\bullet}( \bf{E}) = C^{\bullet}( \bf{E}, \bf{E}).$ We will need the following proposition.

\prop[Prop.\ 2.2 in \cite{ggm}] \label{proposition: dimension2} The space of isomorphism classes of non-split extensions $0 \to \textup{\bf{E}}' \to \textup{\bf{E}} \to \textup{\bf{E}}'' \to 0$ is parametrized by $\bb{P} ( \bb{H}^1(C^{\bullet}(\textup{\bf{E}}', \textup{\bf{E}}''))).$ \eprop

We also record the following useful facts as a lemma. 

\lem \label{lemma:usefulfacts} Let $\textup{\bf{E}}'$ and $\textup{\bf{E}}''$ be non-isomorphic, stable parabolic Higgs bundles of rank $1$. Then we have:
\begin{enumerate}
\item $\bb{H}^0(C^{\bullet}(\textup{\bf{E}}', \textup{\bf{E}}''))=0.$ 
\item $\bb{H}^2(C^{\bullet}(\textup{\bf{E}}', \textup{\bf{E}}'')) \simeq \bb{H}^0(C^{\bullet}(\textup{\bf{E}}'', \textup{\bf{E}}')).$ 
\item $\bb{H}^i(C^{\bullet}(\textup{\bf{E}}', \textup{\bf{E}}''))=0$ for all $i<0$ and $i>2.$ 
\end{enumerate}
\elem

\pf By \cite[Prop.\ 2.2(ii)]{ggm}, the set $\bb{H}^0(C^{\bullet}(\textup{\bf{E}}',\textup{\bf{E}}''))$ parametrizes morphisms of parabolic Higgs bundles from $\textup{\bf{E}}'$ to $\textup{\bf{E}}''$. But since $\textup{\bf{E}}'$ and $\textup{\bf{E}}''$ are non-isomorphic and have equal rank, it follows from \cite[(3.3)]{thaddeus} that the zero map is the only such morphism. This proves (1). Next, (2) is a direct corollary of \cite[2.3(ii)]{ggm}. Finally, (3) follows from \cite[2.3(ii)]{ggm} and the fact that $C^{\bullet}(\textup{\bf{E}}', \textup{\bf{E}}'')$ is concentrated in non-negative degrees, and therefore has no hypercohomology in negative degrees. \epf

It follows from (1) and (2) that $\bb{H}^2(C^{\bullet}(\textup{\bf{E}}',\textup{\bf{E}}''))=0$, which implies that $\op{dim} \bb{H}^1( C^{\bullet}(\textup{\bf{E}}',\textup{\bf{E}}''))= - \chi( C^{\bullet}( \textup{\bf{E}}', \textup{\bf{E}}'')).$  

\lem \label{lemma:riemannroch} Let $\textup{\bf{E}}'$ and $\textup{\bf{E}}''$ be non-isomorphic, stable parabolic Higgs bundles on $(\mc{S}, D)$ of rank $1$. Assume moreover that $\textup{\bf{E}}'$ and $\textup{\bf{E}}''$ have distinct weights. Then $\chi( C^{\bullet}(\textup{\bf{E}}', \textup{\bf{E}}'' ))= -2g.$ \elem

\pf 

By Riemann-Roch, we have 
\begin{align*}
\chi(\op{SParHom}({E}',{E}'') \otimes K(D)) &= \op{deg}(\op{SParHom}({E}',{E}'') \otimes K(D)) + r(1-g) \\
&= \op{deg}(\op{SParHom}({E}', {E}''))+ \op{deg}(K(D)) + r(1-g) \\
&= \op{deg}(\op{SParHom}({E}',{E}'')) + (2g-2) + deg(D)+ 1(1-g).
\end{align*}

Again applying Riemann-Roch, we have
\begin{align*} \chi(\op{SParHom}({E}',{E}'')) &= \op{deg}(\op{SParHom}({E}', {E}''))+ r(1-g) \\
&=  \op{deg}(\op{SParHom}({E}', {E}''))+ (1-g).
\end{align*}

Putting together the above two equations, it follows that
\begin{align*}  \chi(\op{SParHom}({E}',{E}'') \otimes K(D)) &=  \chi(\op{SParHom}({E}',{E}''))  + (2g-2)+\op{deg}(D) \\
&= \chi(\op{SParHom}({E}',{E}'')) + 2g,
 \end{align*}
where we have used the fact that $\op{deg}(D)=2.$
 
Now, by definition we have $\chi(C^{\bullet}(\bf{E}', \bf{E}''))  =  \chi(\op{ParHom}({E}',{E}''))  - \chi(\op{SParHom}({E}',{E}'') \otimes K(D)).$  This is equivalent to $\chi(C^{\bullet}(\bf{E}', \bf{E}''))  =  \chi(\op{ParHom}({E}',{E}'')) - \chi(\op{SParHom}({E}',{E}'') - 2g$ by the above equality.  

Since $\bf{E}'$ and $\bf{E}''$ have distinct weights, there is no difference between parabolic and strictly parabolic morphisms, and it follows that $\op{ParHom}({E}',{E}'')= \op{SParHom}({E}',{E}'').$ Therefore we have $ \chi(\op{ParHom}({E}',{E}'')) - \chi(\op{SParHom}({E}',{E}'')) =0$ and the lemma follows.  \epf

\Cref{lemma:riemannroch}, along with the observation that $\op{dim} \bb{H}^1( C^{\bullet}(\bf{E}', \bf{E}''))= - \chi( C^{\bullet}( \bf{E}', \bf{E}''))$ which was stated after the proof of \Cref{lemma:usefulfacts}, implies that $\op{dim} \bb{H}^1( C^{\bullet}(\bf{E}', \bf{E}''))= 2g$. This leads to the following corollary of \Cref{proposition: dimension2}. 

\cor \label{corollary:dimensionfiber} The dimension of $\mc{E}_p$ is at most $2g-1.$ \ecor

It is now straightforward to prove \Cref{theorem:dimensioncount}.

\pf[Proof of \Cref{theorem:dimensioncount}] 

There exists a natural surjection $\mc{E} \to \mc{X}$, whose fiber over a point $p = ((E^+_*, \Phi^+, \bm{w}_1),  (E^-_*, \Phi^-, \bm{w}_2) )$ is simply the set of extensions of $(E^-_*, \Phi^-, \bm{w}_2) $ by $(E^+_*, \Phi^+, \bm{w}_1).$ The dimension of $\mc{E}$ is bounded above by the sum of the dimensions of $\mc{X}$ and $\mc{E}_p.$  The theorem thus follows by putting together \Cref{lemma:dimension1} and \Cref{corollary:dimensionfiber}.  As noted previously, this argument works just as well for $\tilde{\mc{E}}$ in place of $\mc{E}.$ \epf

\end{document}